\documentclass[a4paper, 12pt]{amsart}
\usepackage[utf8]{inputenc}
\usepackage[T2A]{fontenc}
\usepackage{amsmath,amssymb,amsthm}
\usepackage[a4paper,hmargin=2.5cm,vmargin=2.5cm]{geometry}
\usepackage[english]{babel}
\usepackage{physics}
\usepackage{xcolor}
\usepackage{graphicx} 
\usepackage[colorlinks=true, allcolors=blue]{hyperref}

\makeatletter
\renewcommand{\l@subsection}{\@tocline{2}{2pt}{1.5em}{2.3em}{}}
\makeatother

\makeatletter
\renewcommand{\l@subsubsection}{\@tocline{3}{2pt}{3em}{3em}{}}
\makeatother

\theoremstyle{plain}
\newtheorem{lemma}{Lemma}[section]
\newtheorem{proposition}{Proposition}[section]
\newtheorem{definition}{Definition}[section]
\newtheorem{theorem}{Theorem}[section]
\newtheorem{remark}{Remark}[section]

\newcommand\restr[2]{\ensuremath{\left.#1\right|_{#2}}}

\newcommand\dprod[1]{\left\langle#1\right\rangle}
\newcommand\spanning[1]{\mathrm{span}\left(#1\right)}

\title[Conformally critical metrics and optimal bounds for Dirac eigenvalues]{Conformally critical metrics and optimal bounds for Dirac eigenvalues on spin surfaces}

\author{Pavel Martynyuk}

\begin{document}

\begin{abstract}
We study the minimization problem for eigenvalues of the Dirac operator within a fixed conformal class on a closed spin Riemannian manifold. We establish a criterion for the existence of a minimizer for this variational problem, focusing specifically on the case of closed surfaces. Furthermore, we apply our results to derive isoperimetric inequalities for the Dirac operator on the two-dimensional sphere, providing a complete characterization of its conformal spectrum.
\end{abstract}

\maketitle

\tableofcontents

\section{Introduction.}

\subsection{Overview of Dirac eigenvalue optimization.}

Let $(M,g,\sigma)$ be a closed Riemannian spin manifold equipped with a spin structure $\sigma$.  
This structure yields the Hermitian spinor bundle $(\Sigma_g M,\dprod{\cdot,\cdot}_g)$ over $M$, whose smooth sections are called spinor fields.  
We denote by $\Gamma(\Sigma_g M)$ the space of smooth sections of this bundle.

The classical Dirac operator $D_g$ is an elliptic, first–order differential operator acting on spinor fields.  
Its spectrum is real, discrete, unbounded from both above and below, and each eigenvalue has finite multiplicity.

The study of eigenvalue estimates for $D_g$ is a central topic in spectral geometry. A fundamental lower bound was established by Friedrich~\cite{Friedrich}, who proved that for any eigenvalue $\lambda$ of the Dirac operator on a closed spin manifold of dimension $n \geqslant 2$, 
\[
\lambda^2 \geqslant \frac{n}{4(n-1)} \min_M S_g,
\]
where $S_g$ is the scalar curvature of the metric $g$. While numerous variants of this inequality exist (see~\cite{Ginoux2009DiracSpectrum} for a survey), a particularly relevant refinement for our purposes is due to Hijazi~\cite{Hijazi1986}. He showed that for $n \geqslant 3$, any eigenvalue $\lambda$ satisfies
\[
\lambda^2 \geqslant \frac{n}{n-2} \mu_1,
\]
where $\mu_1$ is the first eigenvalue of the conformal Laplacian $L_g \overset{\mathrm{def}}{=} \Delta_g + \frac{n-2}{4(n-1)}S_g$.

In dimension two, one also has a spinorial analogue of the famous Hersch’s inequality~\cite{Hersch1970}. More precisely, B\"ar in~\cite{Baer1992} proved that for any Riemannian metric $g$ on the two-dimensional sphere $\mathbb{S}^2$, any eigenvalue $\lambda$ of the corresponding Dirac operator satisfies
\[
\lambda^2 \mathrm{Vol}(\mathbb{S}^2,g) \geqslant 4\pi.
\]
B\"ar’s result does not generalize to other manifolds, as was shown in~\cite{GinouxGrosjean}. Moreover, it is conjectured that any manifold of dimension $n \geqslant 3$ admits a Riemannian metric $g$ such that $\ker D_g \neq \{\,0\,\}$. This conjecture was partially verified by Hitchin~\cite{HITCHIN19741} for $n \equiv 0,1,7 \pmod{8}$ and by B\"ar~\cite{Bar1996} for $n \equiv 3 \pmod{4}$.

While minimizing Dirac eigenvalues over all Riemannian metrics of fixed volume fails to yield a strictly positive lower bound, the situation differs significantly when the optimization is restricted to a conformal class. Given the well-behaved transformation properties of the Dirac operator under conformal changes~\cite{HITCHIN19741, Hijazi1986}, it is natural to study the lower bounds of its positive eigenvalues in this setting. We denote by $\lambda_k(D_g)$ the $k$-th positive eigenvalue of $D_g$, counted with multiplicity.
We consider the scale–invariant quantities
\[
\bar\lambda_k(D_g)
\overset{\mathrm{def}}{=}\lambda_k(D_g)\,\mathrm{Vol}(M,g)^{1/n}.
\]
A key result by Lott~\cite{pjm/1102700215} and Ammann~\cite{AMMANN200321} ensures that the infimum of these quantities over a conformal class, defined as
\[
\Lambda_k(M,[g],\sigma) \overset{\mathrm{def}}{=} \inf_{\tilde{g} \in [g]} \bar{\lambda}_k(D_{\tilde{g}})
\]   
is strictly positive. 

\begin{remark}
In the literature $\Lambda_1(M,[g],\sigma)$ is also called the B\"ar-Hijazi-Lott invariant of $(M,[g],\sigma)$.
\end{remark}

The question of whether this infimum is attained was addressed by Ammann~\cite{ammann2009smallest} for $k=1$.
He demonstrated that if
\[
\Lambda_1(M, [g], \sigma) < \Lambda_1(\mathbb{S}^n,g_{\mathrm{st}},\sigma_\mathrm{st}) = \frac{n}{2} \omega_n^{1/n},
\]
then there exists a spinor field $\varphi \in \mathcal{C}^{2,\alpha}(\Sigma M) \cap \mathcal{C}^{\infty}(\Sigma(M \setminus \varphi^{-1}(0)))$ on $(M, g, \sigma)$ such that
    \[
    D_{g} \varphi = \Lambda_1(M,[g],\sigma) \abs{\varphi}_g^{2/(n - 1)} \varphi, \quad \norm{\varphi}_{L^{\frac{2n}{n-1}}} = 1,
    \]
and the metric $g_{min} = \abs{\varphi}^{4/(n - 1)}_gg$, satisfies
\[
\bar{\lambda}_1(g_{min}) = \Lambda_{1}(M,[g],\sigma).
\]

This statement closely resembles the famous Yamabe problem. In particular, as it is shown in~\cite{Ammann2008}, we always have the inequality 
\[
\Lambda_1(M, [g], \sigma) \leqslant \Lambda_1(\mathbb{S}^n,g_{\mathrm{st}},\sigma_\mathrm{st}).
\]
However, it is worth mentioning that it is not known in general when we have a strict inequality. Partial results in this direction can be found for example in~\cite{AmmannHumbertMorel2006,SireXu2021,2021arXiv211203640S}.

Recently, in dimension $2$, Karpukhin, Métras, and Polterovich~\cite{KarpukhinMetrasPolterovich2024} established a link between minimizers of Dirac eigenvalues on Riemann surfaces and harmonic maps into complex projective spaces. This development parallels earlier classical results of Nadirashvili~\cite{Nadirashvili1996} and El Soufi–Ilias~\cite{ElSoufiIlias2008}, which connected maximal metrics for Laplace eigenvalues to harmonic maps and minimal immersions into spheres.

The geometric description of metrics maximizing Laplace eigenvalues, together with the characterization of maximal metrics for Steklov eigenvalues by Fraser and Schoen~\cite{Fraser2013MinimalSA}, initiated the construction of maximizers of the first Laplace eigenvalue~\cite{Nadirashvili1996, PetGAFA, NadirashviliSire2015, Pet18IMRN, KarpukhinNadirashviliPenskoiPolterovichSinDG,  KKMS2024, Pet2024geometricspectraloptimization, KPS2025, Pet2025LMS, Vinokurov_2025} and of the first Steklov eigenvalue~\cite{FraserSchoeninventiones2016, PetJDG, KarpukhinSternJEMS, KKMS2024, Pet2025LMS, Vinokurov_2025}.  

Motivated by the characterization obtained in~\cite{KarpukhinMetrasPolterovich2024}, we seek criteria for the existence of minimizers of $\Lambda_k(M,[g],\sigma)$ for any two-dimensional spin Riemannian manifold $(M,g,\sigma)$. We first prove the following result:

\begin{theorem}\label{main theorem for surfaces}
Let $(M,g,\sigma)$ be a closed spin Riemannian manifold of dimension $2$. Let $k$ be a strictly positive integer and assume that 
\[
\Lambda_k(M,[g],\sigma) < \inf_{l_0 + \ldots + l_r = k}
\left(\Lambda_{l_0}(M,[g],\sigma)^2 + \Lambda_{l_1}(\mathbb{S}^2,[g_\mathrm{st}],\sigma_\mathrm{st})^2 + \ldots + \Lambda_{l_r}(\mathbb{S}^2,[g_\mathrm{st}],\sigma_\mathrm{st})^2 \right)^\frac{1}{2}, 
\]
where 
\begin{itemize}
    \item either $l_0 = 0$ and $\Lambda_{l_0}(M,[g],\sigma) = 0$ or $\Lambda_{l_0}(M,[g],\sigma)$ is attained and $l_0 < k$,
    \item for all $i >0$, $l_i > 0$ and $\Lambda_{l_i}(\mathbb{S}^2,[g_\mathrm{st}],\sigma_\mathrm{st})$ is attained.
\end{itemize}
Then there exist smooth spinor fields $\varphi_1,\ldots,\varphi_d$ on $(M, g, \sigma)$ such that
\begin{itemize}
    \item $\beta \overset{\mathrm{def}}{=} \sum\limits_{i = 1}^d \abs{\varphi_i}_g^2$ satisfies $\norm{\beta}_{L^2(M)} = 1$
    \item $D_{g} \varphi_i = \Lambda_k(M,[g],\sigma) \beta \varphi_i$ for all $i \in \{\,1,\ldots,d\,\}$.
\end{itemize}
The metric $g_\mathrm{min} =\beta^2 g$ satisfies
    \[
    \bar{\lambda}_k(g_{min}) = \Lambda_{k}(M,[g],\sigma).
    \]
Moreover, the set of degeneration of $g_\mathrm{min}$, denoted $\mathrm{Sing}(g_{min})$, is finite. Furthermore,
    \[
    \# \mathrm{Sing}(g_{min}) \leqslant \mathrm{genus}(M) - 1 + \frac{k}{2}.
    \]
\end{theorem}

\begin{remark}
We actually prove a slightly more general result. More precisely, we show that if $(M,g,\sigma)$ is locally conformally flat, then under similar assumptions we can find a generalized metric such that $\Lambda_k(M,[g],\sigma)$ is attained. However, when $\dim M > 2$, this metric would be just H\"older continuous.
\end{remark}

This theorem is a generalization of Ammann's result~\cite{ammann2009smallest} in dimension 2 to higher eigenvalues. As noticed before, the Dirac operator shares many features with the conformal Laplacian. For instance, Ammann's result is a version of Yamabe's problem when we replace the first eigenvalue of the conformal Laplacian by the first non-zero Dirac eigenvalue. As studied in \cite{AmmannHumbert2006,GurskyPerez2022,humbert2025extremisingeigenvaluesgjmsoperators} the Yamabe problem can be generalized by optimization of the other eigenvalues of the conformal Laplacian. We adopt the same approach for higher Dirac eigenvalues. The similarities between the eigenvalues of conformally covariant operators permitted us to adapt techniques developed by Humbert, Petrides, and Premoselli for GJMS operators in~\cite{humbert2025extremisingeigenvaluesgjmsoperators} to our setting.

The existence theorem provided above allows us to derive several qualitative properties of the conformal spectrum.

For any closed Riemannian spin manifold $(M,g,\sigma)$ the sequence $(\Lambda_k(M,[g],\sigma))_{k \geqslant 1}$ is monotone. Recalling that eigenvalues of the Dirac operator on closed Riemann surfaces have even multiplicity, it holds that $\Lambda_{2k-1}(M,[g],\sigma) = \Lambda_{2k}(M,[g],\sigma)$. 
Our existence theorem further implies a strict gap between $\Lambda_{2k}$ and $\Lambda_{2k + 2}$:

\begin{theorem}\label{theorem:strict mononicity}
Let $(M,g,\sigma)$ be a two-dimensional closed Riemannian spin manifold, then the sequence $(\Lambda_{2k}(M,[g],\sigma))_{k \geqslant 1}$ is strictly increasing.
\end{theorem}

Theorem \ref{theorem:strict mononicity} allows the computation of the conformal spectrum of the Dirac operator on $\mathbb{S}^2$. 

\begin{theorem}\label{theorem:spectrum of S^2}
$\Lambda_{2k}(\mathbb{S}^2,[g_\mathrm{st}],\sigma_\mathrm{st}) = 2\sqrt{k\pi}$ for all $k \in \mathbb{Z}_{>0}$, where $2\sqrt{k\pi}$ is the $2k$-th normalized eigenvalue of the Dirac operator on the disjoint union of $k$ unit spheres.
\end{theorem}

\begin{remark}
Since the spectrum of the Dirac operator on a closed surface is symmetric, i.e., $\lambda_{-k} = -\lambda_k$ for any $k > 0$, it follows that $\Lambda_{-2k}(\mathbb{S}^2,[g_\mathrm{st}],\sigma_\mathrm{st}) = -2\sqrt{k\pi}$ for all $k \in \mathbb{Z}_{>0}$.
\end{remark}

This aligns with previous results established in~\cite{KNPP} for the conformal spectrum of the Laplacian on the unit $2$-dimensional sphere, in~\cite{Karpukhin2021} for the conformal spectrum of the Laplacian on the real projective plane and in~\cite{hersch1974isoperimetric} regarding the conformal Steklov spectrum on the unit disk. Moreover, this result can be seen as the generalization of  B\"ar's inequality. By contrast, the computation of the $k$-th conformal invariant of the round sphere associated to eigenvalues of the GJMS operators is still unknown, while in \cite{humbert2025extremisingeigenvaluesgjmsoperators}, the authors proved that such a result cannot occur in high dimension ($n\geq 7$ for eigenvalues of the conformal Laplacian). A result comparable to Theorem \ref{theorem:spectrum of S^2} for the Dirac operator remains widely open in higher dimensions.

\subsection{Organization of the article.}
In Sections~\ref{section: auxiliary part 1} and~\ref{section:auxiliary part 2} we establish several preliminary results concerning sequences of generalized eigenspinors. Section~\ref{section: auxiliary part 1} is devoted to $H^1$ setting, while Section~\ref{section:auxiliary part 2} treats sequences in $W^{1,\frac{2n}{n+1}}$. 

In Sections~\ref{section:existence easy case} and~\ref{section:existence general case} we prove the existence of eigenspinors for generalized metrics $\beta^2 g$, where $\beta \in L^n(M)$. Section~\ref{section:continuity} is concerned with functional $\lambda_k$ for which we prove lower semi-continuity and, in fact, continuity along any direction. 

Section~\ref{section:variational theory} develops the variational theory of generalized eigenvalues: we compute right derivatives of $\lambda_k$ and prove the existence of solutions to subcritical problems. 

Starting from Section~\ref{section:Aubin} we restrict our attention to locally conformally flat manifolds. In Section~\ref{section:Aubin} we prove the upper bound for $\Lambda_k(M,[g],\sigma)$, while in Section~\ref{section:Minimization} we show that, if $\Lambda_k$ is strictly smaller than a certain threshold, then it is attained. 

Section~\ref{section:applications} is devoted to applications of the techniques developed in this work; in particular, we prove Theorems~\ref{theorem:strict mononicity} and~\ref{theorem:spectrum of S^2}. Finally, in Appendix~\ref{appendix}, we establish a Sobolev-type inequality for Dirac operators on closed Riemannian spin manifolds, which is of independent interest.

\subsection{Acknowledgements} 
The author is deeply grateful to Romain Petrides for his invaluable guidance and for many helpful discussions we had during the preparation of the manuscript. His feedback on every version of the draft, from the first to the last, was essential to the final result.

\subsection{Notations.}
We use the letters $\varphi$, $\phi$, and $\psi$ to denote spinor fields, while conformal factors are typically denoted by $\beta$. We denote the volume form associated to $g$ by $dv_g$. 

We usually omit the bundle $\Sigma_g M$ when indicating the regularity of spinor fields; for instance, we simply write $\varphi \in L^p$ to mean that $\varphi$ is an $L^p$-spinor field. Similarly, we write $W^{1,p}$ for the space of $W^{1,p}$-spinor fields. When $p = 2$, we write $H^1$ instead of $W^{1,2}$.  In contrast, when referring to functions on $M$, we explicitly include the manifold in the notation: $L^p$-functions on $M$ are denoted by $L^p(M)$, Sobolev functions on $M$ by $W^{1,p}(M)$, etc.

We say that $f \in L^p_{\geqslant c}(M)$ if $f \in L^p(M)$ and $f \geqslant c$ almost everywhere on $M$.

If $(X,\norm{\cdot})$ is a Banach space, then we use the notation \[
a_n \overset{X}{\to} a \quad (\text{resp. } a_n \overset{X}{\rightharpoonup} a)
\]
to say that the sequence $(a_n)_n$ converges strongly (resp. weakly) to $a$.

We now recall the fundamental Sobolev embedding theorems for the Sobolev spaces $W^{1,p}$ of spinor fields. Ideas of the proof can be found in~\cite{Ammann2002Variational}. These results will be used frequently throughout the text.
\begin{itemize}
    \item[(1)] If $1 \leq p < n$, then for all $q \in \left[ p, \frac{np}{n - p} \right]$, there exists a continuous embedding
    \[
    W^{1,p} \hookrightarrow L^q.
    \]

    \item[(2)] If $p = n$, then
    \[
    W^{1,n} \hookrightarrow L^q, \quad \text{for all } q < \infty.
    \]

    \item[(3)] Suppose that $\alpha \in (0,1)$ and $\frac{1}{p} \leqslant \frac{1  - \alpha}{n}$, then
    \[
    W^{1,p} \hookrightarrow \mathcal{C}^{0,\alpha}.
    \]

    \item[(4)] (Rellich–Kondrashov theorem) For $1 \leq p < n$, the embedding
    \[
    W^{1,p} \hookrightarrow L^q
    \]
    is compact for all $q \in \left[ 1, \frac{np}{n - p} \right)$.
\end{itemize}

\subsection{Setting for the problem.}
In this article we study the following problem: Let $(M,g,\sigma)$ be an $n$-dimensional closed spin Riemannian manifold. The conformal class of $g$ is the set
\[
[g] \overset{\mathrm{def}}{=} \{\,f^2 g \mid f \in \mathcal{C}^{\infty}(M) \quad \text{and} \quad f > 0\,\}.
\]
Let $k > 0$ be a positive integer. We denote by $\lambda_k(f)$ the $k$-th positive eigenvalue of $D_{f^2g}$ and by $\lambda_{-k}(f)$ the $k$-th negative eigenvalue of $D_{f^2g}$. Define 
\[
\Lambda_{k}(M,[g],\sigma) \overset{\mathrm{def}}{=}
\begin{cases}
\inf\limits_{f^2g \in [g]} \bar{\lambda}_k(f) \quad &\text{if } k >0,\\
\sup\limits_{f^2g \in [g]} \bar{\lambda}_k(f) \quad &\text{if } k < 0,
\end{cases}
\]
where $\bar{\lambda}_k(f) \overset{\mathrm{def}}{=} \lambda_k(f)\mathrm{Vol}(M,f^2g)^\frac{1}{n} = \lambda_k(f)\norm{f}_{L^n(M)}$.

Is $\Lambda_k(M,[g],\sigma)$ attained by some smooth Riemannian metric? In general, the answer to this question is negative as can be seen for example from Ammann's result~\cite{ammann2009smallest}. However, if we extend the class of admissible metrics, then we can prove the existence of a minimizer (resp. maximizer) for positive (resp. negative) $\Lambda_k(M,[g],\sigma)$ provided our manifold satisfies certain geometrical and analytical assumptions. Firstly, we adopt an important convention.

\begin{remark}
In order to simplify the notation, we focus exclusively on the case $k > 0$. All results remain true for negative eigenvalues because we can always replace $D_g$ by $-D_g$. Moreover, the most interesting case for us is when $\dim M = 2$, i.e., when $M$ is a surface. In this case the spectrum of $D_g$ is symmetric, i.e., $\lambda_k = -\lambda_k$, and that's why minimizers for $\Lambda_k$ are automatically maximizers for $-\Lambda_k$.
\end{remark}

Since the spectrum of the Dirac operator is unbounded both from below and from above, the minimax description of eigenvalues is quite complicated. Namely, for a fixed metric $g$ we have
\begin{equation}\label{minimax}
\lambda_k^{-1} = \sup_{E \in \mathcal{G}_k\left(H^1 \setminus \ker D_g\right)} \min_{\varphi \in E\setminus \{\,0\,\}}\frac{\int_M \dprod{D_g\varphi,\varphi}_gdv_g}{\int_M\abs{D_g\varphi}_g^2dv_g},    
\end{equation}
where $\mathcal{G}_k\left(H^1 \setminus \ker D_g\right)$ is the Grassmannian of all $k$-dimensional subspaces $E$ in the Sobolev space $H^1$ of spinor fields such that $E \cap \ker D_g = \{\,0\,\}$. The formula~\eqref{minimax} follows easily if one writes $\varphi$ as the sum of eigenspinors.

The Dirac operator is conformally covariant in the following sense:

\begin{lemma}[Conformal covariance of the Dirac operator~\cite{HITCHIN19741,Hijazi1986}]\label{lemma:conformal-Dirac-transf}
Let $(M,g,\sigma)$ be a closed Riemannian spin manifold. Take $f^2g \in [g]$
Then there exists a vector-bundle isometry
\[
T : \Sigma_gM \;\longrightarrow\; \Sigma_{f^2g}M 
\]
such that for every spinor $\varphi \in \Gamma(\Sigma_gM)$,
\[
D_{f^2g} \bigl(T(\varphi)\bigr)
     = T\bigl( f^{-\frac{n+1}{2}}\,D_{g}\bigl( f^{\frac{n-1}{2}}\varphi \bigr)\bigr).
\]
Equivalently, if one defines the modified identification
\[
\tilde{T}(\varphi) \overset{\mathrm{def}}{=} T\bigl( f^{-\frac{n-1}{2}} \varphi\bigr),
\]
then under this identification the Dirac operator transforms as
\[
D_{f^2g}(\varphi) = f^{-1} D_{g}(\varphi),
\]
and the fibrewise pointwise norm satisfies
\[
\abs{\varphi}_{f^2g} \;=\; f^{-\frac{n-1}{2}}\abs{\varphi}_{g}.
\]
\end{lemma}
\begin{remark}
It follows that $\dim_\mathbb{C} \ker D_g$ is a conformal invariant. Moreover, the functional
\[
\varphi \mapsto \int_M \dprod{D_g\varphi , \varphi}_g d\mathrm{vol}_g
\]
is conformally invariant. 
\end{remark}
Thus, it follows that 
\begin{equation}\label{minimax for conformal class}
\lambda_k^{-1}(f) = \sup_{E \in \mathcal{G}_k\left(H^1 \setminus \ker D_g\right)} \min_{\varphi \in E\setminus \{\,0\,\}}\frac{\int_M \dprod{D_g\varphi,\varphi}_gdv_g}{\int_M \frac{1}{f}\abs{D_g\varphi}_g^2dv_g},    
\end{equation}

For $\beta \in L^n(M)$ such that $\beta \geqslant 0$ define the functional
\[
\mathcal{F}(\beta)[\varphi] \overset{\mathrm{def}}{=} \frac{\int_M \dprod{D_g\varphi,\varphi}_gdv_g}{\int_M\frac{1}{\beta}\abs{D_g\varphi}_g^2dv_g},
\]
where $\varphi \in H^1$ (see Section~\ref{section:existence general case} for further details). We can now define generalized eigenvalues for non-smooth conformal factors.

\begin{definition}
For $\beta \in L^n(M)$ such that $\beta \geqslant 0$ define 
\[
\lambda_k^{-1}(\beta) = \sup_{E \in \mathcal{G}_k\left(H^1 \setminus \ker D_g\right)} \min_{\varphi \in E\setminus \{\,0\,\}} \mathcal{F}(\beta)[\varphi].
\] 
\end{definition}

Our goal is to provide conditions that guarantees the existence of a positive function $\beta \in L^n(M)$ that is actually $\mathcal{C}^{0,\alpha}(M)$ such that $\Lambda_k(M,[g],\sigma) = \lambda_k(\beta)\norm{\beta}_{L^n(M)}$, where $(M,g)$ is locally conformally flat. 

\section{Auxiliary results. Part I.}\label{section: auxiliary part 1}
In this section we collect several results that will be used later. 

\subsection{Compactness.}
\begin{lemma}\label{lemma:compact opetator from H^1 to L^2(gamma^2)}
If $\gamma_i \overset{L^n(M)}{\to} \gamma$ and $\varphi_i \overset{H^1}{\rightharpoonup}\varphi$, then
\[
\int_{M}\gamma_i^2\abs{\varphi_i - \varphi}_g^2dv_g \to 0.
\]
\end{lemma}
\begin{proof}
Fix $R > 0$, then
\[
\int_M \abs{\varphi - \varphi_i}_g^2 \gamma_i^2 dv_g = \int_M\abs{\varphi_i - \varphi}_g^2(\gamma_i^2 - \gamma^2)dv_g + \int_{\abs{\gamma} > R}\abs{\varphi - \varphi_i}_g^2\gamma^2 dv_g + \int_{\abs{\gamma} \leqslant R}\abs{\varphi - \varphi_i}_g^2\gamma^2 dv_g
\]
Since the embedding $H^1 \to L^2$ is compact, it follows that 
\[
\abs{\int_{\gamma \leqslant R} \abs{\varphi - \varphi_i}_g^2\gamma^2 dv_g} \leqslant R^2 \norm{\varphi_i - \varphi}^2_{L^2} \underset{i \to \infty}{\to} 0.
\]
Moreover, since $\varphi_i \overset{L^{\frac{2n}{n - 2}}}{\rightharpoonup} \varphi$, it follows that $\exists C > 0$ such that $\norm{\varphi_i - \varphi}_{L^\frac{2n}{n-2}} < C$ for all $i \in \mathbb{N}$.
By H\"older's inequality, we deduce that
\[
\int_{\abs{\gamma} > R}\abs{\varphi - \varphi_i}_g^2\gamma^2 dv_g \leqslant \norm{\varphi - \varphi_i}^{\frac{1}{2}}_{L^\frac{2n}{n-2}}\left(\int_{\abs{\gamma} > R}\gamma^n dv_g\right)^{\frac{2}{n}} < \sqrt{C} \left(\int_{\abs{\gamma} > R}\gamma^n dv_g\right)^{\frac2{n}}.
\]
Applying H\"older's inequality once again, we obtain that
\[
\abs{\int_M\abs{\varphi_i - \varphi}_g^2(\gamma_i^2 - \gamma^2)dv_g} \leqslant
\left(\int_M \abs{\varphi_i - \varphi}_g^\frac{2n}{n - 2}dv_g\right)^{\frac{n - 2}{n}}\norm{\gamma^2 - \gamma_i^2}_{L^\frac{n}{2}}\underset{i \to \infty}{\to} 0,
\]
because $\gamma_i^2 \overset{L^\frac{n}{2}(M)}{\to} \gamma^2$ and $\norm{\varphi_i - \varphi}_{L^\frac{2n}{n - 2}} < C$ for all $i \in \mathbb{N}$.

Thus, for any $R > 0$, we have that
\[
\limsup_{i \to \infty} \int_M \abs{\varphi - \varphi_i}_g^2 \gamma_i^2 dv_g \leqslant \sqrt{C} \left(\int_{\abs{\gamma} > R}\gamma^n dv_g\right)^{\frac{2}{n}}.
\]
Passing to the limit with respect to $R$ we get that
\[
\limsup_{i \to \infty} \int_M \abs{\varphi - \varphi_i}_g^2 \gamma_i^2 dv_g = 0.
\]
\end{proof}

\subsection{Regularity.}
\begin{lemma}\label{lemma:regularity}
Suppose $D_g \varphi = \beta \varphi$, where $\varphi \in W^{1,q_0}$ for some $q_0 > 1$ and $\beta \in L^n_{\geqslant 0}(M)$, then $\varphi \in L^{p}$ for all $p \in [1,+\infty)$.
\end{lemma}
\begin{proof}
Since the spectrum of $D_g$ is discrete, there exists a constant $c \geqslant 0$ such that $D \overset{\mathrm{def}}{=}D_g + c \mathrm{Id}$ is invertible. Note that $D^{-1}$ is a pseudodifferential operator of order $-1$. In particular, $D^{-1} \colon L^q \to W^{1,q}$ is continuous for all $q \in (1,+\infty)$. Therefore, it will be convenient for us to consider the following equation:
\[
D \varphi = (\beta + c)\varphi.
\]

Fix $\varepsilon > 0$. Note that since $\varphi$ is integrable, it follows that
\[
\lim_{m \to \infty} \int_{\abs{\varphi} > m}dv_g = 0.
\]
Thus, by the absolute continuity of the integral, there exists $N_{\varepsilon} \in \mathbb{N}$ such that 
\[
\int_{\abs{\varphi} > N_{\varepsilon}}(\beta + c)^n dv_g < \varepsilon^n.
\]
Define $q_\varepsilon \overset{\mathrm{def}}{=} \chi_{\abs{\varphi} > N_{\varepsilon}}(\beta + c)$ and $\psi_{\varepsilon} \overset{\mathrm{def}}{=}\chi_{\abs{\varphi} \leqslant N_{\varepsilon}}(\beta + c) \varphi$. Consequently, $\psi_{\varepsilon} \in L^n$ and $\norm{q_{\varepsilon}}_{L^n(M)} < \varepsilon$. 

Hence, we have an equation
\begin{equation}\label{equation:generalized eigenspinor}
D\varphi - q_\varepsilon\varphi = \psi_{\varepsilon}.  
\end{equation}

Let us define an operator $\mathcal{H}_{\varepsilon}(\phi) \overset{\mathrm{def}}{=} D^{-1}(q_{\varepsilon}\phi)$. and rewrite~\eqref{equation:generalized eigenspinor} in the following form:
\begin{equation}\label{equation:H_epsilon}
\varphi - \mathcal{H}_\varepsilon \varphi = D^{-1}(\psi_{\varepsilon}).
\end{equation}
Since $\psi_{\varepsilon} \in L^n$, elliptic regularity tells us that $D^{-1}(\psi_{\varepsilon}) \in W^{1,n}$. By Sobolev embedding theorem, we have that $W^{1,n} \hookrightarrow L^p$ for all $p \in [1,+\infty)$. Consequently, the right-hand side of~\eqref{equation:H_epsilon} belongs to $L^p$ for all $p \in [1,+\infty)$, Hence, our goal now is to prove that $\mathcal{H}_{\varepsilon}$ maps $L^p$ to $L^p$
.

Fix $p \geqslant 1$ and take $\psi \in L^p$, then by H\"older's inequality,
\begin{equation}\label{inequality:norm of q_e phi}
\int_M \abs{q_{\varepsilon}\psi}^\frac{np}{n+p}dv_g \leqslant \left(\int_M \abs{q_{\varepsilon}}^n dv_g\right)^\frac{p}{n+p} \left(\int_M \abs{\psi}^p dv_g\right)^\frac{n}{n+p},    
\end{equation}
so $q_\varepsilon\psi \in L^\frac{np}{n+p}$. We will denote $\frac{np}
{n+p}$ by $\hat{p}$. 
By continuity of $D^{-1}$, 
\[
\norm{\mathcal{H}_{\varepsilon}\psi}_{W^{1,\hat{p}}} = \norm{D^{-1}(q_{\varepsilon}\psi)}_{W^{1,\hat{p}}} \leqslant C_{\mathrm{elliptic}} \norm{q_\varepsilon \psi}_{L^{\hat{p}}},
\]
where $C_{\mathrm{elliptic}} \overset{\mathrm{def}}{=} \norm{D^{-1}}_{L^{\hat{p}} \to W^{1,\hat{p}}}$. Consequently,
\[
\norm{\mathcal{H}_{\varepsilon}\psi}_{W^{1,{\hat{p}}}} \leqslant C_{\mathrm{elliptic}}\varepsilon\norm{\psi}_{L^p}.
\]
By Sobolev's inequality, we deduce
\[
\norm{\mathcal{H}_{\varepsilon}\psi}_{L^p} \leqslant C_{\mathrm{Sobolev}}\norm{\mathcal{H}_{\varepsilon}\psi}_{W^{1,\hat{p}}} \leqslant C_{\mathrm{Sobolev}}C_{\mathrm{elliptic}}\varepsilon\norm{\psi}_{L^p}.
\]
Choose $\varepsilon > 0$ such that $\norm{\mathcal{H}_{\varepsilon}}_{L^p \to L^p} < \frac{1}{2}$, then $\mathrm{Id} - \mathcal{H}_{\varepsilon}$ is invertible, so from~\eqref{equation:H_epsilon} it follows that
\[
\varphi = (\mathrm{Id} - \mathcal{H}_{\varepsilon})^{-1}(D^{-1}(\psi_{\varepsilon})) \in L^p.
\]
\end{proof}

\subsection{Sequences of eigenspinors.}

\begin{lemma}\label{lemma: H^1 analogue of lemma 2.1}
Let $(\varphi_i)_{i \in \mathbb{N}}$ be a sequence in $H^1$, $(\lambda_i)_{i \in \mathbb{N}}$ be a sequence of strictly positive real numbers, $(\beta_i)_{i \in \mathbb{N}}$ be a sequence of positive nonzero functions in $L^n(M)$. Assume that for all $i \in \mathbb{N}$ we have
\begin{itemize}
    \item $D_g \varphi_i = \lambda_i \beta_i \varphi_i$,
    \item $\int_M \beta_i^2 \abs{\varphi_i}_g^2 dv_g = 1$,
    \item $(\beta_i)_{i \in \mathbb{N}}$ is bounded in $L^n(M)$,
    \item $\limsup\limits_{i \to \infty} {\lambda_i} < \infty$.
\end{itemize}
For all $i \in \mathbb{N}$ we write $\varphi_i = \psi_i + \kappa_i$, where $\psi_i \in \ker D_g^\perp$ and $\kappa_i \in \ker D_g$. Then
\begin{enumerate}
    \item $(\psi_i)_{i \in \mathbb{N}}$ is bounded in $H^1$,
    \item If $\norm{\kappa_i}_{L^2} \to \infty$, then $\beta_i \overset{L^n}{\rightharpoonup} 0$.
\end{enumerate}
\end{lemma}

\begin{proof}
The first claim follows from the following estimates:
\begin{align*}
\norm{\psi_i}_{H^1} 
&\leqslant C_1\norm{D_g \psi_i}_{L^2} \quad &&\text{(Elliptic estimates)}\\
&= C_1 \abs{\lambda_i} \norm{\beta_i \varphi_i}_{L^2}\\
&= C_1\abs{\lambda_i} \leqslant C_2 \limsup \abs{\lambda_j}
\end{align*}
Now assume that $\norm{\kappa_i}_{L^2} \to \infty$. Define 
\[
\Phi_i \overset{\mathrm{def}}{=} \frac{\varphi_i}{\norm{\kappa_i}_{L^2}}, \quad K_i \overset{\mathrm{def}}{=} \frac{\kappa_i}{\norm{\kappa_i}_{L^2}} \quad \text{and} \quad \Psi_i \overset{\mathrm{def}}{=} \frac{\psi_i}{\norm{\kappa_i}_{L^2}}.
\]
Since $(\psi_i)_{i \in \mathbb{N}}$ is bounded in $H^1$, it follows that $\Psi_i \overset{H^1}{\to} 0$. Moreover, $(K_i)_{i \in \mathbb{N}}$ is bounded in $\left(\ker D_g, \norm{\cdot}_{H^1}\right)$. Since $\dim \ker D_g < \infty$, we may, after passing to a subsequence, assume that $K_i \overset{H^1}{\to} K$. Therefore, $\Phi_i \overset{H^1}{\to} K$. Since $(\beta_i)_{i \in \mathbb{N}}$ is bounded in $L^n$, we may, once again after passing to a subsequence, assume that $\beta_i \overset{L^n}{\rightharpoonup}\beta$.

Note that $\int_M \beta_i \dprod{\varphi_i,\kappa_i}_g dv_g = \frac{1}{\lambda_i}\int_M \dprod{D_g \varphi_i, \kappa_i}_g dv_g = \frac{1}{\lambda_i} \int_M \dprod{\varphi_i,D_g \kappa_i}_g dv_g = 0$.
Thus, $0 = \int_M \beta_i \dprod{\Phi_i,K_i}_gdv_g$. Passing to the limit, we obtain $0 = \int_M \beta \abs{K}^2 dv_g$. Note that $K$ is nonzero because $\norm{K_i}_{H^1} = 1$ for all $i \in \mathbb{N}$. Moreover, $D_g K = 0$, then by the unique continuation property of $D_g$, $\abs{K}_g > 0$ almost everywhere. Consequently, $\beta = 0$ almost everywhere.

Note that we only proved that there is a subsequence of $(\beta_i)_{i \in \mathbb{N}}$ that converges weakly to $0$. However, any subsequence of $\beta_i$ admits a further subsequence that also converges weakly to $0$ by the same argument. Hence, $\beta_i \overset{L^n}{\rightharpoonup} 0$.
\end{proof}

\begin{lemma}\label{lemma:existence of eigenspinors}
Let $(\beta_m)_{m\in\mathbb{N}}$ be a sequence of functions in $L^n_{\geqslant 0}(M)$ such that $\beta_m \overset{L^n(M)}{\to} \beta$, where $\beta$ is a nonzero function. Assume that $(\lambda_k(\beta_m))_{m \in \mathbb{N}}$ is bounded and that for each $m \in \mathbb{N}$ there exist $\varphi^m_1,\ldots,\varphi^m_k$ in $H^1$ such that 
\begin{itemize}
    \item $D_g \varphi^m_j = \lambda_j(\beta_m)\beta_m\varphi^m_j$ for all $j \in \{\,1,\ldots,k\,\}$,
    \item $\int_M\beta_m^2\abs{\varphi^m_j}_g^2dv_g = 1$ for all $j \in \{\,1,\ldots,k\,\}$,
    \item $\int_M \beta_m \dprod{\varphi^m_i,\varphi^m_j}_g = 0$ for all $i \neq j \in \{\,1,\ldots,k\,\}$.
\end{itemize}
Then there exist nonzero spinors $\phi_1,\ldots,\phi_k$ in $H^1$ such that
\begin{itemize}
    \item $\phi_j$ is a limit point of the sequence $(\varphi^m_j)_k$ for all $j \in \{\,1,\ldots,k\,\}$,  
    \item $D_g \phi_j = (\liminf\limits_{m \to \infty} \lambda_j(\beta_m)) \beta \phi_j$ for all $j \in \{\,1,\ldots,k\,\}$,
    \item $\int_M \beta \dprod{\phi_i,\phi_j}_g = 0$ for all $i \neq j \in \{\,1,\ldots,k\,\}$,
\end{itemize}
Moreover, $\lambda_j(\beta) \leqslant \liminf\limits_{m \to \infty} \lambda_j(\beta_m)$ for all $j \in \{\,1,\ldots,k\,\}$.
\end{lemma}

\begin{proof}
By definition, $\lambda_1(\beta_m) \leqslant \ldots \leqslant \lambda_k(\beta_m)$. Therefore, $(\lambda_j(\beta_m))_{m \in \mathbb{N}}$ is bounded for all $j \in \{\,1 \ldots, k\,\}$. After passing to a subsequence, we may assume that 
\begin{equation}\label{equality: limit of lambda_m}
\lim_{m \to \infty}\lambda_j(\beta_m) = \liminf\limits_{m \to \infty} \lambda_j(\beta_m).
\end{equation}
For the sake of simplicity, we will denote $ \liminf\limits_{m \to \infty} \lambda_j(\beta_m)$ by $\lambda_j$.

Since $(\beta_m)_{m \in \mathbb{N}}$ does not converge weakly to $0$, Lemma~\ref{lemma: H^1 analogue of lemma 2.1} implies that $(\varphi^m_j)_{m \in \mathbb{N}}$ is bounded in $H^1$ for all $j \in \{\,1,\ldots,k\,\}$. After passing to a subsequence, we may assume that there exists $\phi_1,\ldots,\phi_k$ in $H^1$ such that $\varphi^m_j 
\overset{H^1}{\rightharpoonup} \phi_j$. We want to prove that $\phi_j 
\neq 0$ for all $j \in \{\,1,\ldots,k\,\}$. Assume the contrary; then by Lemma \ref{lemma:compact opetator 
from H^1 to L^2(gamma^2)}, we get that
$\lim\limits_{m \to \infty}\int_M \beta_m^2 \abs{\varphi^m_j}_g^2dv_g = 0$,
but by assumption $\int_M \beta_m^2\abs{\varphi^m_j}_g^2dv_g = 1$ for all $m \in \mathbb{N}$. Contradiction. 

Note that since $D_g \varphi^m_j \overset{L^2}{\rightharpoonup} D_g\phi_j$, it follows that for any smooth spinor field $\psi \in \mathcal{C}^{\infty}$ we have

\begin{equation}\label{equality: phi is an eigenspinor}
\int_{M}\dprod{D_g \phi_j, \psi}_gdv_g = \lim_{m \to \infty} \int_M \dprod{D_g\varphi^m_j,\psi}_gdv_g = \lim_{m \to \infty} \lambda_k(\beta_m) \int_{M}\dprod{\beta_m\varphi^m_j,\psi}_gdv_g.   
\end{equation}
Recall that $\varphi^m_j \overset{L^\frac{n}{n-1}}{\to} \phi_j$ and $\beta_m \overset{L^n(M)}{\to} \beta$. Therefore, we obtain that 

\begin{equation}\label{equality: limit of beta_m}
\lim_{m \to \infty}\int_{M}\dprod{\beta_m\varphi^m_j,\psi}_gdv_g = \int_{M}\dprod{\beta\phi_j,\psi}_gdv_g.    
\end{equation}
It follows from \eqref{equality: phi is an eigenspinor},\eqref{equality: limit of beta_m} and \eqref{equality: limit of lambda_m} that for all $\psi \in \mathcal{C}^{\infty}$
\[
\int_{M}\dprod{D_g \phi_j, \psi}_gdv_g =\int_M \dprod{ \lambda_j \beta \phi_j, \psi}_gdv_g.
\]
Consequently, $D_g\phi_j = \lambda_j \beta \phi_j$.

By Rellich-Kondrashov theorem, $H^1 \hookrightarrow L^{\frac{2n}{n-1}}$ is compact, so $\dprod{\varphi^i_m,\varphi^j_m}_g \overset{L^\frac{n}{n-1}(M)}{\to} \dprod{\phi_i, \phi_j}_g$. Consequently, we get that 
\[
\lim_{m \to \infty}\int_M\beta_m\dprod{\varphi^i_m,\varphi^j_m}_g dv_g = \int_M \beta  \dprod{\phi_i, \phi_j}_gdv_g.
\]

By definition,
\[
\lambda_j(\beta)^{-1} \geqslant \min_{\psi \in \spanning{\phi_1,\ldots,\phi_j} \setminus \{\,0\,\}} \mathcal{F}(\beta)[\psi] = \lambda_j^{-1}.
\]
Thus, $\lambda_j(\beta) \leqslant \lambda_j$ for all $j \in \{\,1,\ldots,k\,\}$. 
\end{proof}

\section{Existence of eigenspinors. Bounded away from 0 case.}\label{section:existence easy case}
In this section we assume that $\beta \geqslant \delta > 0$ for some positive constant $\delta$. We will show the existence of eigenspinors in this case. Note that under this assumption we have that $\lambda_k(\beta)$ is finite for all $k \in \mathbb{N}_{>0}$. Indeed, since $\beta \geqslant \delta$ almost everywhere, it follows that for all spinors $\Phi$ we have that
\begin{align*}
\frac{\int_M \dprod{D_g\Phi,\Phi}_gdv_g}{\int_M\frac{1}{\beta}\abs{D_g\Phi}_g^2dv_g} \geqslant
\frac{\int_M \dprod{D_g\Phi,\Phi}_gdv_g}{\int_M\frac{1}{\delta}\abs{D_g\Phi}_g^2dv_g}.
\end{align*}
Therefore,
\[
\frac{1}{\lambda_k(\beta)} \geqslant \frac{1}{\lambda_k(\delta)} \Longleftrightarrow \lambda_k(\beta) \leqslant \lambda_k(\delta).
\]

We start this section with a simple observation, that follows directly from the minimax description of generalized eigenvalues.

\begin{lemma}\label{lemma:smooth upper continuity of eigenvalues}
Let $(\beta_m)_{m \in \mathbb{N}}$ be a sequence of strictly positive smooth functions such that $\beta_m \overset{L^n(M)}{\to} \beta$. Assume that $\left(\frac{1}{\beta_m}\right)_{m \in \mathbb{N}}$ is bounded in $L^{\infty}(M)$, then

\[
\limsup\limits_{m \to \infty} \lambda_k(\beta_m) \leqslant \lambda_k(\beta).
\]
\end{lemma}
\begin{proof}
Fix $\varepsilon > 0$. By definition of $\lambda_k(\beta)$, there exists $E \in \mathcal{G}_k\left(H^1 \setminus \ker D_g \right)$ such that 
\[
\min_{\varphi \in E \setminus 0} \mathcal{F}(\beta)[\varphi] \geqslant \frac{1}{\lambda_k(\beta)} - \varepsilon.
\]
In particular, for all $\varphi \in E \setminus \{\,0\,\}$ we have $\int_{M} \frac{1}{\beta} \abs{D_g\varphi}^2 dv_g > 0$ (otherwise $\varphi \in \ker D_g$). Therefore, $\dprod{\varphi,\psi}_\beta \overset{\mathrm{def}}{=} \int_M \frac{1}{\beta} \dprod{D_g\varphi, D_g\psi}dv_g$ defines an Hermitian product on $E$. Let $\varphi_1,\ldots,\varphi_k$ be an orthonormal basis of $(E,\dprod{\cdot,\cdot}_\beta)$. 

Take a subsequence $(\beta_{\varphi(m)})_{m \in \mathbb{N}}$ such 
\[
\lim_{m \to \infty} \lambda_k(\beta_{\varphi(m)}) = \limsup\limits_{m \to \infty}\lambda_k(\beta_m).
\]
Since $\beta_{\varphi(m)}\overset{L^n(M)    }{\to} \beta$, we may, after passing to a subsequence, assume that $\beta_{\varphi(m)} \to \beta$ almost everywhere. Hence, since $\left(\frac{1}{\beta_{\varphi(m)}}\right)$ is bounded in $L^{\infty}$, we may apply dominated convergence theorem and deduce that
\[
\lim_{m \to \infty} \int_M \frac{1}{\beta_{\varphi(m)}} \dprod{D_g\varphi,D_g\psi}_gdv_g = \int_M \frac{1}{\beta} \dprod{D_g\varphi,D_g\psi}_gdv_g.
\]
In other words,
\[
\dprod{\varphi_i,\varphi_j}_{\beta_{\varphi(m)}}= \delta_{ij} + o(1).
\]
Take $\varphi \in E$ such that $\int_M \frac{1}{\beta} \abs{D_g\varphi}dv_g = 1$, i.e., $\varphi = \sum\limits_{j = 1}^ka_j \varphi_j$, where $\sum\limits_{j = 1}^k \abs{a_j}^2 = 1$. Therefore,
\[
\dprod{\varphi,\varphi}_{\beta_{\varphi(m)}} = \sum_{i,j = 1}^k a_i \bar{a}_j(\delta_{ij} + o(1)) = 1 + o(1). 
\]
Note that since $\sum\limits_{j=1}^k \abs{a_j}^2 = 1$, it follows that $o(1)$ depends only  on $m$. Hence, for all $\varphi \in E$ such that $\dprod{\varphi,\varphi}_\beta = 1$ we have
\[
\mathcal{F}(\beta_m)[\varphi] = \frac{\mathcal{F}(\beta)[\varphi]}{1 + o(1)}.
\]
Thus, for $m$ big enough we have 
\[
\min_{\varphi \in E \setminus \{\,0\,\}} \mathcal{F}(\beta_{\varphi(m)})[\varphi] \geqslant \min_{\varphi \in E \setminus \{\,0\,\}} \mathcal{F}(\beta)[\varphi] - \varepsilon \geqslant \frac{1}{\lambda_k(\beta)} - 2\varepsilon.
\]
Consequently, $\lambda_k(\beta_{\varphi(m)})^{-1} \geqslant \lambda_k(\beta)^{-1} - 2\varepsilon$. Since $\varepsilon$ was an arbitrary positive number, it follows that 
\[
\limsup \lambda_k(\beta_m) \leqslant \lambda_k(\beta).
\]
\end{proof}
The following proposition is a straightforward application of Lemma~\ref{lemma:existence of eigenspinors}.
\begin{proposition}[Existence of eigenspinors. Bounded away from zero case]\label{theorem:existence of eigenspinors} Let $\beta$ be a function in $L^n_{\geqslant \delta}(M)$.
For all $k \in \mathbb{N}_{>0}$ there exists $\phi_k \neq 0$ such that
\[
\lambda_k(\beta)^{-1} = \mathcal{F}(\beta)[\phi_{k}].
\]
Moreover, $D_g\phi_{k} = \lambda_k(\beta)\beta\phi_{k}$ and $\int_M \beta \dprod{\phi_i,\phi_j} dv_g = 0$ for all $i \neq j$ in $\mathbb{N}_{>0}$.
\end{proposition}

\begin{proof}
There exists a sequence $(\beta_m)_{m \in \mathbb{N}} \in (\mathcal{C}^{\infty}(M))^{\mathbb{N}}$ such that we have $\beta_m \geqslant \delta$ for all $m \in \mathbb{N}$ and $\beta_m \overset{L^n(M)}{\to} \beta$.
Since $\beta_m \geqslant \delta$ for all $m \in \mathbb{N}$, it follows that $\lambda_k(\beta_m) \leqslant \lambda_k(\delta)$. Moreover, since $\beta_m$ is a strictly positive and smooth function, all other assumptions of Lemma~\ref{lemma:existence of eigenspinors} are satisfied. Hence, there exist nonzero spinors $\phi_1,\ldots,\phi_k$ in $H^1$ such that
\begin{itemize}
    \item $D_g \phi_j = \liminf\limits_{m\to \infty}(\lambda_j(\beta_m)) \beta \phi_j$ for all $j \in \{\,1,\ldots,k\,\}$,
    \item $\int_M \beta \dprod{\phi_i,\phi_j} = 0$ for all $i \neq j \in \{\,1,\ldots,k\,\}$,
\end{itemize}
Also, $\lambda_j(\beta) \leqslant \liminf\limits_{m\to \infty}\lambda_j(\beta_m)$ for all $j \in \{\,1,\ldots,k\,\}$. It follows from Lemma~\ref{lemma:smooth upper continuity of eigenvalues} that $\limsup\lambda_j(\beta_m) \leqslant \lambda_j(\beta)$. Consequently, 
\[
\limsup\limits_{m\to \infty}\lambda_j(\beta_m) \leqslant \lambda_j(\beta) \leqslant \liminf\limits_{m\to \infty} \lambda_j(\beta_m),
\]
which implies that $\lambda_j(\beta) = \liminf\limits_{m\to \infty} \lambda_j(\beta_m)$, i.e. $D_g\phi_j =\lambda_j(\beta)\beta\phi_j$ for all $j \in \{\,1,\ldots,k\,\}$.
\end{proof}

\begin{remark}
Note that we actually proved above that $\lim\limits_{m \to \infty}\lambda_k(\beta_m) = \lambda_k(\beta)$. Indeed, we showed that for any sequence $(\beta_m)_{\mathbb{N}} \in \mathcal{C}^{\infty}_{\geqslant\delta}(M)$ such that $\beta_m  \overset{L^n(M)}{\to}\beta$, there exists a subsequence $(\beta_{\varphi(m)})_{m \in \mathbb{N}}$ such that $\lim\limits_{m \to \infty}\lambda_k(\beta_{\varphi(m)}) = \lambda_k(\beta)$. Now by a standard subsequence argument one can deduce that $\lim\limits_{m \to \infty}\lambda_k(\beta_m) = \lambda_k(\beta)$.
\end{remark}

\begin{proposition}
[Continuity of generalized eigenvalues in $L^n_{\geqslant\delta}(M)$]\label{theorem:continuity in L^n_delta}
Fix $\delta > 0$. Assume that $\beta_m \overset{L^n}{\to}\beta$, where $\beta_m,\beta \geqslant \delta$ almost everywhere. Then $\lambda_k(\beta_m) \to \lambda_k(\beta)$ for all $k \in \mathbb{N}$.
\end{proposition}
\begin{proof}
By the argument above we can approximate each $\beta_m$ by a strictly positive smooth function $\Tilde{\beta}_m$, which is greater than or equal to $\delta$, such that 
\[
\abs{\lambda_k(\beta_m) - \lambda_k(\Tilde{\beta}_m)} < 2^{-m}
\]
and 
\[
\norm{\beta_m - \Tilde{\beta}_m}_{L^n} < 2^{-m}.
\]
Consequently, $\Tilde{\beta}_m \overset{L^n}{\to} \beta$ and $\lim\limits_{m \to \infty} \lambda_k(\Tilde{\beta}_m) = \lambda_k(\beta)$. By construction of the sequence $(\Tilde{\beta}_m)_{m \in \mathbb{N}}$, we have that $\lim\limits_{m \to \infty}\lambda_k(\beta_m) = \lambda_k(\beta)$. 
\end{proof}

\section{Existence of eigenspinors. General case.}\label{section:existence general case}
In this section we develop our theory of generalized eigenvalues and eigenspinors when our function $\beta \in L^n(M)$ is just non-negative almost everywhere. In particular, in this section we allow $\beta$ to be $0$ on a set of positive measure. Therefore, we have to develop the following convention: $\frac{r}{0} = +\infty$ for $r > 0$, $\frac{r}{0} = -\infty$ for $r < 0$ and $\frac{0}{0} = 0$, so the main problem for us will be the spinors $\varphi$ such that the set $\{\,x \in M \mid \beta(x) = 0 \quad \text{and} \quad \abs{D_g\varphi(x)} = 0\,\}$ has positive measure. For $\beta \geqslant 0$ almost everywhere we can still define 
\[
\lambda_k(\beta)^{-1} \overset{\mathrm{def}}{=} \sup_{E \in \mathcal{G}_k\left(H^1 \setminus \ker D_g\right)} \min_{\varphi \in E\setminus \{\,0\,\}} \mathcal{F}(\beta)[\varphi],
\]
however, now $\int_M \frac{1}{\beta}\abs{D_g\varphi}^2dv_g$ can be equal to $+\infty$. Consequently, it is possible that for certain $\beta \in L^n_{\geqslant 0}(M)$ we obtain $\lambda_k(\beta) = +\infty$, but since we are looking for minimizers in the conformal class $[g]$, this is simply an inconvenience.

Our first result is related to upper semi-continuity of $\lambda_k(\beta)$.
\begin{lemma}\label{lemma:easy upper continuity of eigenvalues. Positive case}
For all $\beta \in L^n(M)$ we have that
\[
\limsup_{m \to \infty}\lambda_k\left(\beta + \frac{1}{m}\right) \leqslant  \lambda_k(\beta) 
\]
\end{lemma}
\begin{proof}
If $\lambda_k(\beta) = +\infty$, then there is nothing to prove, so we restrict our attention to the case $\lambda_k(\beta) < +\infty$.

Fix $\varepsilon > 0$. Since $\lambda_k(\beta)$ is finite, there exists $E \in \mathcal{G}_k(H^1 \setminus \ker D_g)$ such that 
\[
\min_{\varphi \in E \setminus 0} \mathcal{F}(\beta)[\varphi] \geqslant \frac{1}{\lambda_k(\beta)} - \varepsilon.
\]
Thus, in particular, for all $\varphi \in E \setminus \{\,0\,\}$ we have
$\int_M \frac{1}{\beta}\abs{D_g \varphi}^2 dv_g < +\infty$. 

Note that if $\int_M \frac{1}{\beta}\abs{D_g \varphi}_g^2 dv_g = 0$, then $\abs{D_g \varphi}_g$ vanishes almost everywhere. Indeed, $D_g \varphi$ vanishes almost everywhere on a set $\{\,x \in M \mid \beta(x) > 0\,\}$. Moreover, since $\int_M \frac{1}{\beta}\abs{D_g \varphi}_g^2 dv_g < +\infty$, it follows that the set $\{\,x \in M \mid \beta(x) = 0 \quad \text{and} \quad \abs{D_g\varphi(x)}_g > 0\,\}$ is negligible. Consequently, $D_g\varphi = 0$. Since $E \cap \ker D_g = \{\,0\,\}$, it follows that $\varphi = 0$.

Note that $\forall \varphi,\psi \in E$ we have 
\[
\int_{M}\frac{1}{\beta}\abs{\dprod{D_g\varphi,D_g\psi}_g}dv_g \leqslant\int_M \frac{1}{\beta} (\abs{D_g\varphi}_g^2 + \abs{D_g\psi}_g^2)dv_g < \infty.
\]
Thus, $\dprod{\varphi,\psi}_{\beta} \overset{\mathrm{def}}{=} \int_M \frac{1}{\beta} \dprod{D_g\varphi,D_g\psi}_gdv_g$ is a well-defined Hermitian product on $E$. 

Now take a subsequence $\left(\beta + \frac{1}{m_j}\right)_{m \in \mathbb{N}}$ such 
\[
\lim_{m \to \infty} \lambda_k\left(\beta + \frac{1}{m_j}\right) = \limsup\limits_{m \to \infty} \lambda_k\left(\beta + \frac{1}{m}\right).
\]
For any $\varphi,\psi \in E$ we have that $\frac{1}{\beta}\dprod{D_g\varphi,D_g\psi}$ is integrable. Clearly, 
\[
\frac{1}{\beta + \frac{1}{m_j}}\abs{\dprod{D_g\varphi, D_g\psi}_g} \leqslant \frac{1}{\beta}\abs{\dprod{D_g\varphi,D_g\psi}_g} \quad \text{almost everywhere.}
\]
Hence,
by dominated convergence theorem.
\[
\lim_{m \to \infty} \int_M \frac{1}{\beta + \frac{1}{m_j}} \dprod{D_g\varphi,D_g\psi}_gdv_g = \int_M \frac{1}{\beta} \dprod{D_g\varphi,D_g\psi}_gdv_g.
\]
Using now exactly the same reasoning as in the proof of Lemma~\ref{lemma:smooth upper continuity of eigenvalues} we deduce that
\[
\limsup_{m \to \infty}\lambda_k\left(\beta + \frac{1}{m}\right) \leqslant \lambda_k(\beta).
\]
\end{proof}

We now establish the existence of generalized eigenspinors for $\beta \in L^n_{\geqslant 0}(M)$. Obviously, if we want to find $\varphi_k \neq 0$ such that $D_g \varphi_k = \lambda_k(\beta)\beta \varphi_k$, it is necessary to assume that $\lambda_k(\beta)$ is finite.

\begin{proposition}[Existence of eigenspinors. Non-negative  case]\label{theorem:existence of eigenspinors. Positive case.}
Let $\beta \in L^n$ be a non-negative nonzero function. For all $k \in \mathbb{N}_{>0}$ such that $\lambda_k(\beta)$ is finite, there exists a nonzero $\phi_k$ such that
\[
\lambda_k(\beta)^{-1} = \mathcal{F}(\beta)[\phi_{k}].
\]
Moreover, $D_g\phi_{k} = \lambda_k(\beta)\beta\phi_{k}$, $\int_M \beta \dprod{\phi_i,\phi_j}_g dv_g= 0$ for all $i \neq j$ in $\{\,1,\ldots, k\,\}$ and $\int_M \beta^2 \abs{\phi_j}_g^2dv_g = 1$ for all $j \in \{\,1\ldots,k\,\}$.
\end{proposition}
\begin{proof}
As one might expect, the proof of this theorem resembles the proof of Theorem~\ref{theorem:existence of eigenspinors}.

Since $\beta \geqslant 0$ almost everywhere, we have that $\beta_m \overset{\mathrm{def}}{=}\beta + \frac{1}{m} \geqslant \frac{1}{m}$ almost everywhere. By Proposition~\ref{theorem:existence of eigenspinors}, for each $m \in \mathbb{N}$ there exist nonzero spinors $\phi^m_1,\ldots,\phi^m_k$ in $H^1$ such that
\begin{itemize}
    \item $D_g \phi^m_j = \lambda_j(\beta_m) \beta_m \phi^m_j$ for all $j \in \{\,1,\ldots,k\,\}$,
    \item $\int_M \beta_m \dprod{\phi^m_i,\phi^m_j}_g = 0$ for all $i \neq j \in \{\,1,\ldots,k\,\}$.
\end{itemize}
Since $\phi^m_j$ is nonzero and $\beta_m$ is strictly positive almost everywhere, it follows that
\[
\int_M \beta_m^2\abs{\phi^m_j}_g^2dv_g > 0,
\]
so we may, after dividing $\phi^m_j$ by a suitable constant, assume that
\[
\int_M \beta_m^2\abs{\phi^m_j}^2_gdv_g = 1,
\]
Also, by Lemma~\ref{lemma:easy upper continuity of eigenvalues. Positive case}, we have that 
\[
\limsup_{m \to \infty}  \lambda_k\left(\beta_m\right) \leqslant \lambda_k(\beta) < \infty,
\]
Now we are allowed to use Lemma~\ref{lemma:existence of eigenspinors} which tells us that there exist nonzero spinors $\phi_1,\ldots,\phi_k$ in $H^1$ such that
\begin{itemize}
    \item $D_g \phi_j = \liminf\limits_{m \to \infty}(\lambda_j(\beta_m)) \beta \phi_j$ for all $j \in \{\,1,\ldots,k\,\}$,
    \item $\int_M \beta \dprod{\phi_i,\phi_j} = 0$ for all $i \neq j \in \{\,1,\ldots,k\,\}$,
\end{itemize}
Moreover, $\lambda_j(\beta) \leqslant \liminf\lambda_j(\beta_m)$ for all $j \in \{\,1,\ldots,k\,\}$. It follows from Lemma~\ref{lemma:easy upper continuity of eigenvalues. Positive case} that $\limsup\lambda_j(\beta_m) \leqslant \lambda_j(\beta)$. Consequently, 
\[
\limsup\limits_{m \to \infty}\lambda_j(\beta_m) \leqslant \lambda_j(\beta) \leqslant \liminf\limits_{m \to \infty} \lambda_j(\beta_m),
\]
which implies that $\lambda_j(\beta) = \liminf \lambda_j(\beta_m)$, i.e. $D_g\phi_j =\lambda_j(\beta)\beta\phi_j$ for all $j \in \{\,1,\ldots,k\,\}$.
In order to conclude we have to show that $\int_M\beta^2 \abs{\phi_j}^2_gdv_g = 1$. By Lemma~\ref{lemma:existence of eigenspinors} we obtain a subsequence $\left(\phi^{m_j}_j\right)_{m_j \in \mathbb{N}}$ such that $\phi^{m_j}_j \overset{H^1}{\rightharpoonup} \phi_j$. Hence, from Lemma~\ref{lemma:compact opetator from H^1 to L^2(gamma^2)} we deduce that
\[
\lim_{m \to \infty} \int_M \beta^2_{m_j}\abs{\phi^{m_j}_j - \phi_j}_g^2 dv_g = 0.
\]
Since $\beta_{m_j} = \beta + \frac{1}{m_j}$ and $\left(\phi^{m_j}_j - \phi_j\right)_{m \in \mathbb{N}}$ is bounded in $L^\frac{2n}{n-2}$ (indeed, this sequence is bounded in $H^1$, so, by Sobolev's embedding theorem, it is also bounded in $L^\frac{2n}{n-1}$), it follows that
\[
\lim_{m \to \infty} \int_M \beta^2\abs{\phi^{m_j}_j - \phi_j}^2 dv_g = 0.
\]
Moreover, for exactly the same reasons,
\[
\lim_{m \to \infty} \int_M \beta\abs{\phi^{m_j}_j}^2 
= \lim_{m \to \infty} \int_M \beta_{m_j}^2\abs{\phi^{m_j}_j}^2 dv_g
= 1.
\]
Consequently,
\[
\int_M \beta^2\abs{\phi_j}_g^2dv_g = \lim_{m \to \infty} \int_M \beta^2\abs{\phi^{m_j}_j}_g^2 = 1.
\]

\end{proof}

\begin{remark}
It follows easily from the proof of Proposition~\ref{theorem:existence of eigenspinors. Positive case.} that if $\lambda_k(\beta)$ is finite, we have that $\lim\limits_{m \to +\infty}\lambda_k\left(\beta + \frac{1}{m}\right) = \lambda_k(\beta)$
\end{remark}

\section{Some results on the continuity of generalized eigenvalues.}\label{section:continuity}
In this section, we prove that the functional $\lambda_k \colon L^n_{\geqslant 0}(M) \to [0,+\infty]$ satisfies reasonable  continuity properties. First of all, we show that $\lambda_k$ is a lower semi-continuous functional. This result will permit us to use Ekeland's variational principle. 

\begin{lemma}\label{lemma:lower semi-continuity of lambda}
Let $(\beta_m)_{m \in \mathbb{N}}$ be a sequence in $L^n_{\geqslant 0}(M)$ such that $\beta_m \overset{L^n(M)}{\to} \beta$. Then
\[
\lambda_k(\beta) \leqslant \liminf\limits_{m \to \infty} \lambda_k(\beta_m).
\]   
\end{lemma}
\begin{proof}
If $\liminf \lambda_k(\beta_m) = +\infty$, then there is nothing to prove, so we may assume that $\liminf\lambda_k(\beta_m)$ is finite. Take a subsequence $\left(\lambda_k(\beta_{\varphi(m)})\right)$ such that 
\[
\lim\limits_{m \to \infty}\lambda_k\left(\beta_{\varphi(m)}\right) = \liminf\limits_{m \to \infty} \lambda_k(\beta_m).
\]
Without loss of generality, we may assume that $\lambda_k\left(\beta_{\varphi(m)}\right)$ is finite for all $m \in \mathbb{N}$. Then Theorem~\ref{theorem:existence of eigenspinors. Positive case.} tells us that all assumptions of Lemma~\ref{lemma:existence of eigenspinors} can be satisfied. Therefore, we deduce that 
\[
\lambda_k(\beta) \leqslant \lim\limits_{m \to \infty} \lambda_k\left(\beta_{\varphi(m)}\right) = \liminf\limits_{m \to \infty} \lambda_k(\beta_m).
\]
\end{proof}
The following lemma, heuristically, says that $\lambda_k$ is continuous in all positive directions.
\begin{lemma}
[Directional continuity of generalized eigenvalues]\label{lemma:continuity of generalized eigenvalues in positive direction}
Let $\beta$ and $b$ be two functions in $L^n_{\geqslant 0}(M)$, then 
\begin{align*}
\lim_{t \to 0+}\lambda_k(\beta + t b) = \lambda_k(\beta).   
\end{align*}
\end{lemma}
\begin{proof}
We need to prove that for any sequence of positive numbers $(t_m)_{m \in \mathbb{N}}$ we have that 
\[
\lim_{m \to \infty} \lambda_k(\beta + t_mb) = \lambda_k(\beta).
\]
By Lemma~\ref{lemma:lower semi-continuity of lambda}, we already know that 
\[
\lambda_k(\beta) \leqslant \liminf\limits_{m \to \infty} \lambda_k(\beta + t_mb),
\]
so we only have to prove that 
\[
\limsup\limits_{m \to \infty} \lambda_k(\beta + t_m b) \leqslant \lambda_k(\beta).
\]
If $\lambda_k(\beta) = +\infty$, then there is nothing to prove, so we may assume that $\lambda_k(\beta)$ is finite. The rest of the proof is essentially the same as the proof of Lemma~\ref{lemma:easy upper continuity of eigenvalues. Positive case}.

Take a subsequence $\left(t_{\varphi(m)}\right)_{m \in \mathbb{N}}$ such that
\[
\lim_{m \to \infty}\lambda_k\left(\beta + t_{\varphi(m)}b\right) = \limsup\limits_{m \to \infty} \lambda_k(\beta + t_mb).
\]
 
Since $\lambda_k(\beta)$ is finite, Theorem~\ref{theorem:existence of eigenspinors. Positive case.} tells us that there exists $E \in \mathcal{G}_k(H^1 \setminus \ker D_g)$ such that 
\[
\min_{\varphi \in E \setminus 0} \mathcal{F}(\beta)[\varphi] = \frac{1}{\lambda_k(\beta)}.
\]
Thus, in particular, for all $\varphi \in E \setminus \{\,0\,\}$ we have
$\int_M \frac{1}{\beta}\abs{D_g \varphi}_g^2 dv_g < +\infty$. 

As in the proof of Lemma~\ref{lemma:easy upper continuity of eigenvalues. Positive case}, one shows that $\dprod{\varphi,\psi}_{\beta} \overset{\mathrm{def}}{=} \int_M \frac{1}{\beta} \dprod{D_g\varphi,D_g\psi}_g dv_g$ is a well-defined Hermitian product on $E$. 

For any $\varphi,\psi \in E$ we have that $\frac{1}{\beta}\dprod{D_g\varphi,D_g\psi}$ is integrable. Clearly, 
\[
\frac{1}{\beta + t_{\varphi(m)}b}\abs{\dprod{D_g\varphi, D_g\psi}_g} \leqslant \frac{1}{\beta}\abs{\dprod{D_g\varphi,D_g\psi}_g} \quad \text{almost everywhere.}
\] 
Hence,
by the dominated convergence theorem.
\[
\lim_{m \to \infty} \int_M \frac{1}{\beta + t_{\varphi(m)}b} \dprod{D_g\varphi,D_g\psi}_gdv_g = \int_M \frac{1}{\beta} \dprod{D_g\varphi,D_g\psi}_gdv_g.
\]
Using now exactly the same reasoning from the proof of Lemma~\ref{lemma:smooth upper continuity of eigenvalues} we deduce that
\[
\limsup\limits_{m \to \infty}\lambda_k\left(\beta + t_{m}b\right) \leqslant \lambda_k(\beta).
\]
Thus, we have shown that 
\[
\limsup \lambda_k(\beta + t_m b) \leqslant \lambda_k(\beta) \leqslant \liminf \lambda_k(\beta + t_m b).
\]
Consequently, $\lim\limits_{m \to \infty}\lambda_k(\beta + t_m b) = \lambda_k(\beta)$.
\end{proof}

\section{Auxiliary results. Part II.}\label{section:auxiliary part 2}
In previous sections, we repeatedly used the following idea: consider a sequence of equations
\[
D_g \varphi_m = \beta_m \varphi_m,
\]
where $\beta_m \overset{L^n(M)}{\to} \beta$ and $\beta \neq 0$. If we assume that
\[
\int_M \beta_m^2 \abs{\varphi_m}_g^2\, dv_g = 1 \quad \text{for all } m \in \mathbb{N},
\]
or more generally that the sequence
\[
\left(\int_M \beta_m^2 \abs{\varphi_m}_g^2\, dv_g \right)_{m \in \mathbb{N}}
\]
is bounded, then $(\varphi_m)_{m \in \mathbb{N}}$ admits a subsequence that converges weakly in $H^1$. Moreover, denoting this weak limit by $\varphi$, we obtain that $\varphi$ is non-zero and that it satisfies
\[
D_g \varphi = \beta \varphi.
\]

However, the assumption $\int_M \beta_m^2 \abs{\varphi_m}_g^2\, dv_g = 1$ is quite artificial. In the upcoming sections, we will often deal with sequences of generalized eigenspinors satisfying the normalization
$
\int_M \beta_m |\varphi_m|^2\, dv_g = 1.
$
This condition is considerably more natural since, roughly speaking, it means that $L^2(\beta^2_mg)$-norm of $\varphi_m$ is equal to $1$.

Unfortunately, this hypothesis does not suffice to guarantee weak convergence in $H^1$. Nevertheless, it \emph{does} allow us to extract a subsequence converging weakly in the Sobolev space $W^{1, \frac{2n}{n+1}}$. In this section, we collect several useful results concerning this space.

\subsection{Compactness.}
\begin{lemma}\label{lemma:compact operator from W^1,2n/n+1 to L^2(gamma)}
If $\beta_i \overset{L^n(M)}{\to} \beta$ and $\varphi_i \overset{W^{1,\frac{2n}{n+1}}}{\rightharpoonup}\varphi$, then
\[
\int_{M}\beta_i\abs{\varphi_i - \varphi}_g^2dv_g \to 0.
\]
\end{lemma}
\begin{proof}
Fix $R > 0$, then
\[
\int_M \abs{\varphi - \varphi_i}_g^2 \beta_i dv_g = \int_M\abs{\varphi_i - \varphi}_g^2(\beta_i - \beta)dv_g + \int_{\abs{\beta} > R}\abs{\varphi - \varphi_i}_g^2\beta dv_g + \int_{\abs{\beta} \leqslant R}\abs{\varphi - \varphi_i}_g^2\beta dv_g.
\]
Since $W^{1,\frac{2n}{n+1}} \to L^2$ is compact, it follows that 
\[
\abs{\int_{\beta \leqslant R} \abs{\varphi - \varphi_i}_g^2\beta dv_g} \leqslant R \norm{\varphi_i - \varphi}^2_{L^2} \underset{i \to \infty}{\to} 0.
\]
Moreover, since $\varphi_i \overset{L^{\frac{2n}{n - 1}}}{\rightharpoonup} \varphi$, it follows that $\exists C > 0$ such that $\norm{\varphi_i - \varphi}_{L^\frac{2n}{n-1}} < C$ for all $i \in \mathbb{N}$.
By H\"older's inequality, we deduce that
\[
\int_{\abs{\beta} > R}\abs{\varphi - \varphi_i}_g^2\beta dv_g \leqslant \norm{\varphi - \varphi_i}^{\frac{1}{2}}_{L^\frac{2n}{n-1}}\left(\int_{\abs{\beta} > R}\beta^n dv_g\right)^{\frac{1}{n}} < \sqrt{C} \left(\int_{\abs{\beta} > R}\beta^n dv_g\right)^{\frac{1}{n}}.
\]
Applying H\"older's inequality once again, we obtain that
\[
\abs{\int_M\abs{\varphi_i - \varphi}_g^2(\beta_i - \beta)dv_g} \leqslant
\left(\int_M \abs{\varphi_i - \varphi}_g^\frac{2n}{n - 1}dv_g\right)^{\frac{n - 1}{n}}\left(\int_M \abs{\beta - \beta_i}^n dv_g \right)^\frac{1}{n} \underset{i \to \infty}{\to} 0,
\]
because $\beta_i \overset{L^n(M)}{\to} \beta$ and $\norm{\varphi_i - \varphi}_{L^\frac{2n}{n - 1}} < C$ for all $i \in \mathbb{N}$.

Thus, for any $R > 0$ we have that
\[
\limsup_{i \to \infty} \int_M \abs{\varphi - \varphi_i}_g^2 \beta_i dv_g \leqslant \sqrt{C} \left(\int_{\abs{\beta} > R}\beta^n dv_g\right)^{\frac{1}{n}}.
\]
Passing to the limit with respect to $R$ we get that
\[
\limsup_{i \to \infty} \int_M \abs{\varphi - \varphi_i}_g^2 \beta_i dv_g = 0.
\]
\end{proof}

\begin{lemma}\label{lemma: strong local convergence}
Consider a sequence of equations
\[
D_g \varphi_m = \beta_m \varphi_m,
\]
where $\varphi_m \in W^{1,\frac{2n}{n+1}}$, $\beta_m \in L^n_{\geqslant 0}(M)$.
Assume that $\varphi_m \overset{W^{1,\frac{2n}{n+1}}}{\rightharpoonup} \varphi$ and $\beta_m \overset{L^n}{\rightharpoonup} \beta$. Consider a set
\[
A \overset{\mathrm{def}}{=} \left\{\,x \in M \mid \forall \delta > 0, \quad \limsup_{m \to +\infty} \int_{B_{g}(x,\delta)} \beta_m^n dv_g  > \frac{1}{2^n}K(n)^\frac{-n}{4} \,\right\},
\]
where $K(n) = \Lambda_1^{-1}(\mathbb{S}^n,[g_{\mathrm{st}}],\sigma_{\mathrm{st}})$. If $\beta_m$ converges strongly to $\beta$ in $L^p(M)$ for some $p > \frac{2n}{n+1}$, then $\varphi_m$ converges strongly to $\varphi$ in $W^{1,\frac{2n}{n+1}}_{\mathrm{loc}}\left(\Sigma(M \setminus A)\right)$.
\end{lemma}
\begin{proof}
By Rellich-Kondrashov theorem, $\varphi_m \overset{L^\frac{n}{n-1}}{\to} \varphi$. Hence, after passing to the weak limit, we get that $D_g\varphi = \beta \varphi$. 

Once again, by Rellich-Kondrashov theorem, we have that $\varphi_m \overset{L^\frac{2n}{n+1}}{\to} \varphi$, so it remains to show only that $D_g\varphi_m$ converges strongly to $D_g\varphi$ in $W^{1,\frac{2n}{n+1}}_{\mathrm{loc}}\left(\Sigma(M \setminus A)\right)$.

Fix $x \in M \setminus A$, then $\exists \delta > 0$ such that 
\[
\limsup_{m \to \infty}\int_{B_g(x,\delta)} \beta_m^n \leqslant \frac{1}{2^n}K(n)^\frac{-n}{4}.
\]
Take $\eta \in \mathcal{C}^{\infty}_c(B_g(x,\delta))$. Denote $\varphi_m - \varphi$ by $\psi_m$. We have to show that $D_g(\eta \psi_m) \overset{L^\frac{2n}{n+1}}{\to} 0$.

Note that 
$D_g(\eta \psi_m) = \grad_g \eta \cdot \psi_m + \eta D_g\psi_m.$
By Rellich-Kondrashov theorem, the embedding $W^{1,\frac{2n}{n+1}} \hookrightarrow L^\frac{2n}{n+1}$ is compact. Hence, 
$\grad_g \eta \cdot \psi_m \overset{L^\frac{2n}{n+1}}{\to} 0$.

Now fix $\varepsilon > 0$, then by H\"older's inequality
\begin{align*}
\norm{\eta D_g \psi_m}_{L^\frac{2n}{n+1}} 
= \norm{\eta (\beta_m \varphi_m - \beta \varphi)}_{L^\frac{2n}{n+1}}
&= \norm{\eta (\beta_m \psi_m + (\beta_m - \beta) \varphi)}_{L^\frac{2n}{n+1}}\\
&\leqslant \norm{\eta \beta_m \psi_m}_{L^\frac{2n}{n+1}} + \norm{\eta(\beta_m - \beta)\varphi}_{L^\frac{2n}{n+1}}.
\end{align*}
By lemma~\ref{lemma:regularity}, $\varphi \in L^q$ for all $1 \leqslant q < \infty$. Therefore, since $\beta_m \overset{L^p(M)}{\to} \beta$, we deduce from H\"older's inequality that $\norm{\eta(\beta_m - \beta)\varphi}_{L^\frac{2n}{n+1}} = o(1)$.

We now estimate $\norm{\eta \beta_m \psi_m}_{L^\frac{2n}{n+1}}$. Note that, by H\"older's inequality,
\begin{align*}
\int_M \abs{\eta \beta_m \psi_m}_g^\frac{2n}{n+1}dv_g \leqslant
\left(\int_{B_g(x,\delta)} \beta_m^n dv_g\right)^\frac{1}{n+1} 
\left(\int_M \beta_m \eta^2 \abs{\psi_m}_g^2 dv_g\right)^\frac{n}{n+1}.
\end{align*}
Moreover,
\begin{align*}
\int_M \beta_m \eta^2 \abs{\psi_m}_g^2 dv_g 
&= \int_M \eta^2 \dprod{D_g \psi_m - (\beta_m - \beta)\varphi, \psi_m}_g dv_g\\
&= \int_M \eta^2 \dprod{D_g \psi_m, \psi_m}_g dv_g - \int_M \eta^2(\beta_m - \beta)\dprod{\varphi,\psi_m}_gdv_g.
\end{align*}

Moreover, 
\[
\int_M \eta^2 \dprod{D_g \psi_m, \psi_m}_g dv_g = \int_M \dprod{D_g (\eta \psi_m), \eta\psi_m}_g dv_g + o(1).
\]  
Indeed, we have 
\[
\dprod{D_g (\eta \psi_m), \eta\psi_m}_g = \eta^2 \dprod{D_g \psi_m, \psi_m}_g + \eta\dprod{\grad_g \eta \cdot \psi_m,\psi_m}_g,
\]
and, by Relich-Kondrashov theorem, $\eta\dprod{\grad_g \eta \cdot \psi_m,\psi_m}_g = o(1)$. 

By Theorem~\ref{theorem: Sobolev Constant}, $\exists B > 0$ such that
\[
\int_M \dprod{D_g (\eta \psi_m), \eta\psi_m}_g dv_g \leqslant \frac{3}{2}K(n)\left(\int_M\abs{D_g (\eta \psi_m)}_g^\frac{2n}{n+1} dv_g\right)^\frac{n+1}{n} + B\norm{\eta \psi_m}^2_{L^2}
\]
By Rellich-Kondrashov theorem, $W^{1,\frac{2n}{n+1}} \hookrightarrow L^2$ is compact. Hence, $ B\norm{\eta \psi_m}^2_{L^2} = o(1)$.
Moreover, by H\"older's inequality 
\begin{align*}
\abs{\int_M \eta^2(\beta_m - \beta)\dprod{\varphi,\psi_m}_gdv_g} 
\leqslant \norm{\eta^2(\beta_m - \beta)}_{L^p(M)}\norm{\varphi}_{L^\frac{2np}{np + p - n}} \norm{\psi_m}_{L^\frac{2n}{n-1}} = o(1),
\end{align*}
because $\norm{\psi_m}_{L^\frac{2n}{n-1}}$ is bounded.

Let's summarize. We have shown that
\begin{align*}
\norm{D_g (\eta\psi_m)}_{L^\frac{2n}{n+1}} 
&\leqslant \norm{\grad_g \eta \cdot \psi_m}_{L^\frac{2n}{n + 1}} + \norm{\eta D_g\psi_m}_{L^\frac{2n}{n+1}}\\
&\leqslant o(1) + \norm{\eta D_g\psi_m}_{L^\frac{2n}{n + 1}}\\
&\leqslant o(1) + \norm{\eta \beta_m \psi_m}_{L^\frac{2n}{n + 1}} + \norm{\eta(\beta_m - \beta)\varphi}_{L^\frac{2n}{n + 1}}\\
&\leqslant o(1) + \left(\int_{B_g(x,\delta)}\beta_m^n dv_g\right)^\frac{2}{n}\left(\frac{3}{2}K(n)\norm{D_g(\eta \psi_m)}_{L^\frac{2n}{n+1}}^2 + o(1)\right)^\frac{1}{2}\\
&\leqslant o(1) + \frac{1}{4\sqrt{K(n)}} \cdot \frac{\sqrt{3K(n)}}{\sqrt{2}}\norm{D_g(\eta\psi_m)}_{L^\frac{2n}{n+1}}.
\end{align*}
Consequently,
\[
\norm{D_g (\eta\psi_m)}_{L^\frac{2n}{n+1}}  = o(1),
\]
i.e., $\psi_m$ converges to $0$ in $W^{1,\frac{2n}{n+1}}_{\mathrm{loc}}\left(\Sigma(M \setminus A)\right)$.
\end{proof}

\subsection{Sequences of eigenspinors.}

We state and prove an analogue of Lemma~\ref{lemma: H^1 analogue of lemma 2.1} for $W^{1,\frac{2n}{n+1}}$ (see also Lemma $2.2$ in \cite{humbert2025extremisingeigenvaluesgjmsoperators}).

\begin{lemma}\label{lemma:analogue of lemma 2.1} Consider a sequence of equations
\[
D_g \varphi_i = \lambda_i \beta_i \varphi_i,
\]
where $\varphi_i \in W^{1,\frac{2n}{n+1}}$, $\beta_i \in L^{p_i}_{\geqslant 0}(M)$ and $p_i \geqslant n$. Moreover, we assume that
\begin{itemize}
    \item $\int_M \beta_i \abs{\varphi_i}_g^2 dv_g = 1$,
    \item $\sup\limits_{i \in \mathbb{N}}\left(\int_M \beta_i^{p_i} dv_g\right)$ is finite,
    \item $\limsup \lambda_i < \infty$ and $\lambda_i > 0$ for all $i \in \mathbb{N}$.
\end{itemize}
For all $i \in \mathbb{N}$ we write $\varphi_i = \psi_i + \kappa_i$, where $\psi_i \in \ker D_g^\perp$ and $\kappa_i \in \ker D_g$. Then
\begin{enumerate}
    \item $(\psi_i)_{i \in \mathbb{N}}$ is bounded in $W^{1,\frac{2n}{n+1}}$,
    \item $\exists c > 0$ such that $\norm{\psi_i}_{W^{\frac{2n}{n+1}}} \geqslant c$ and $\lambda_i > c$ for all $i \in \mathbb{N}$.
    \item If a sequence $(\varphi_i)_{i \in \mathbb{N}}$ is unbounded in ${W^{1,\frac{2n}{n+1}}}$, then we have that $p_i \to n$, $\beta_i \overset{L^n(M)}{\rightharpoonup} 0$ and $\psi_i \overset{W^{1,\frac{2n}{n+1}}}{\rightharpoonup} 0$.
\end{enumerate}
\end{lemma}

\begin{proof}
The first claim follows from the following estimates:
\begin{align*}
\norm{\psi_i}_{W^{1,\frac{2n}{n+1}}} 
&\leqslant C_1\norm{D_g \psi_i}_{L^\frac{2n}{n+1}} \quad &&\text{(Elliptic estimates)}\\
&= C_1 \lambda_i \norm{\beta_i \varphi_i}_{L^\frac{2n}{n+1}}\\
&\leqslant C_2 \norm{\beta_i}_{L^n}\left(\int_M \beta_i \abs{\varphi}_g^2 dv_g\right)^\frac{1}{2} \quad &&\text{(H\"older's inequality)}\\
&\leqslant C_3.
\end{align*}

Moreover,
\begin{align*}
\int_M \lambda_i\beta_i \abs{\psi_i}_g^2dv_g 
&= \int_M \dprod{\lambda_i\beta_i(\varphi_i - \kappa_i),\psi_i}_gdv_g \\
&= \int_M \dprod{\lambda_i\beta_i\varphi_i,\psi_i}_gdv_g - \int_M \dprod{\kappa_i,\lambda_i\beta_i\psi_i}_gdv_g\\
&= \int_M \dprod{D_g\varphi_i,\phi_i - \kappa_i}_gdv_g - \int_M\dprod{\kappa_i,D_g\psi_i - \lambda_i\beta_i\kappa_i}_gdv_g\\
&= \int_M \lambda_i\beta_i\abs{\varphi_i}_g^2dv_g  + \int_M \lambda_i\beta_i\abs{\kappa_i}_g^2dv_g \geqslant \lambda_i.
\end{align*}
Since $\lambda_i > 0$, it follows that 
\[
\int_M \beta_i \abs{\psi_i}_g^2dv_g \geqslant 1.
\]
Therefore,
\begin{align*}
1 
&\leqslant \int_M\beta_i\abs{\psi_i}_g^2\\
&\leqslant \norm{\beta_i}_{L^n}\norm{\psi_i}^2_{L^\frac{2n}{n-1}} \quad &&\text{(H\"older's inequality)}\\
&\leqslant C_{\mathrm{Sobolev}} \norm{\beta_i}\norm{\psi_i}^2_{W^{1,\frac{2n}{n+1}}} \quad &&\text{(Sobolev's inequality)}\\
&\leqslant C\norm{\psi_i}^2_{W^{1,\frac{2n}{n+1}}},
\end{align*}
so there exists $c > 0$ such that $\norm{\psi_i}_{W^{1,\frac{2n}{n+1}}} > c$ for all $i \in \mathbb{N}$. 

Using H\"older's inequality and elliptic estimates, we have already shown that $\exists C_1 > 0$ such that
\[
\norm{\psi_i}_{W^{1,\frac{2n}{n+1}}} \leqslant C_1 \lambda_i \norm{\beta_i}_{L^n}\left(\int_M \beta_i \abs{\varphi}_g^2 dv_g\right)^\frac{1}{2} = C_1 \lambda_i \norm{\beta_i}_{L^n}.
\]
Hence, $\exists \tilde{c} > 0$ such that 
$\lambda_i \geqslant \tilde{c}$ for all $i \in \mathbb{N}$.

After passing to a subsequence, we may assume that $\norm{\varphi_i}_{W^{1,\frac{2n}{n+1}}} \to \infty$. Define 
\[
\Phi_i \overset{\mathrm{def}}{=} \frac{\varphi_i}{\norm{\varphi_i}_{W^{1,\frac{2n}{n+1}}}}, \quad K_i \overset{\mathrm{def}}{=} \frac{\kappa_i}{\norm{\varphi_i}_{W^{1,\frac{2n}{n+1}}}} \quad \text{and} \quad {\Psi_i \overset{\mathrm{def}}{=} \frac{\psi_i}{\norm{\varphi_i}_{W^{1,\frac{2n}{n+1}}}}}.
\]
Since $(\psi_i)_{i \in \mathbb{N}}$ is bounded in $W^{1,\frac{2n}{n+1}}$, it follows that $\Psi_i \overset{W^{1,\frac{2n}{n+1}}}{\to} 0$. Moreover, $(K_i)_{i \in \mathbb{N}}$ is bounded in $\left(\ker D_g, \norm{\cdot}_{W^{1,\frac{2n}{n+1}}}\right)$. Since $\dim \ker D_g < \infty$, we may, after passing to a subsequence, assume that $K_i \overset{W^{1,\frac{2n}{n+1}}}{\to} K$. Therefore, $\Phi_i \overset{W^{1,\frac{2n}{n+1}}}{\to} K$. Since $(\beta_i)_{i \in \mathbb{N}}$ is bounded in $L^n(M)$, we may, once again after passing to a subsequence, assume that $\beta_i \overset{L^n}{\rightharpoonup}\beta$. Since $0 = \int_M \beta_i \dprod{\varphi_i,\kappa_i}_gdv_g$, it follows that $0 = \int_M \beta_i \dprod{\Phi_i,K_i}_gdv_g$. Passing to a limit, we obtain $0 = \int_M \beta \abs{K}^2_g dv_g$. Note that $K$ is nonzero because $\norm{\Phi_i}_{W^{1,\frac{2n}{n+1}}} = 1$ for all $i \in \mathbb{N}$. Moreover, $D_g K = 0$, then by the unique continuation property for $D_g$, $\abs{K}_g > 0$ almost everywhere. Consequently, $\beta = 0$ almost everywhere.

Note that we only proved that there is a subsequence of $(\beta_i)_{i \in \mathbb{N}}$ that converges weakly to $0$. However, any subsequence of $\beta_i$ admits a further subsequence that also converges weakly to $0$ by the same argument. Hence, $\beta_i \overset{L^n(M)}{\rightharpoonup} 0$.

Now let's prove that $\psi_i \overset{W^{1,\frac{2n}{n+1}}}{\rightharpoonup} 0$. Since $(\psi_i)_{i \in \mathbb{N}}$ is bounded in $W^{1,\frac{2n}{n+1}}$, we may, after passing to a subsequence, assume that $\psi_i \overset{W^{1,\frac{2n}{n+1}}}{\rightharpoonup} \psi$. Without loss of generality, we may assume that $\lambda_i \to \lambda$. Since $\lambda_i \geqslant \tilde{c} > 0$ for all $i \in \mathbb{N}$, it follows that $\lambda > 0$. Fix a smooth spinor $\Psi \in \mathcal{C}^{\infty}$, then
\begin{align*}
\int_M \dprod{D_g \psi, \Psi}_gdv_g 
&= \int_M \dprod{\psi, D_g\Psi}dv_g \\
&= \lim_{i \to \infty} \int_M \dprod{\psi_i,D_g \Psi}_gdv_g\\
& = \lim_{i \to \infty} \int_M \dprod{D_g\psi_i, \Psi}_gdv_g\\
&= \lim_{i \to \infty} \int_M \dprod{\lambda_i \beta_i\varphi_i,\Psi}_gdv_g\\
&\leqslant \limsup_{i \to +\infty}\lambda_i\norm{\Psi}_{L^\infty}\left(\int_M \beta_i dv_g\right)^{\frac{1}{2}} \left(\int_M \beta_i\abs{\varphi_i}_g^2dv_g\right)^{\frac{1}{2}}\\
&\leqslant \limsup_{i \to +\infty}\lambda_i\norm{\Psi}_{L^\infty}\left(\int_M \beta_i dv_g\right)^{\frac{1}{2}} = 0.
\end{align*}
Hence, $D_g\psi = 0$, but 
\begin{align*}
\int_M \dprod{\psi,\psi}_gdv_g = \lim_{i \to \infty} \int_M \dprod{\psi_i,\psi}_gdv_g =0,
\end{align*}
because $\psi \in \ker D_g$ and $\psi_i \in \ker D_g^{\perp}$. We obtain that $\psi = 0$. Once again note that we only proved that there is a subsequence of $(\psi_i)_{i \in \mathbb{N}}$ that converges weakly to $0$. However, any subsequence of $\psi_i$ admits a further subsequence that also converges weakly to $0$ by the same argument. Hence, $\psi_i \overset{W^{1,\frac{2n}{n+1}}}{\rightharpoonup} 0$.

We now prove $\lim\limits_{i \to \infty}p_i = n$. Assume the converse, then we may, after passing to a subsequence, assume that $\inf p_i = p_{\infty} > n$. By H\"older inequality,
\[
\int_M \beta_i \abs{\psi_i}^2dv_g \leqslant \norm{\psi_i}^2_{L^\frac{2p_{\infty}}{p_{\infty}-1}}\norm{\beta_i}_{L^{p_{\infty}}}.
\]
Since $\int_M \beta_i \abs{\psi_i}_g^2dv_g \geqslant 1$, it follows that $ \norm{\psi_i}^2_{L^\frac{2p_{\infty}}{p_{\infty}-1}}\norm{\beta_i}_{L^{p_{\infty}}} \geqslant 1$. However, $(\beta_i)_{i \in \mathbb{N}}$ is bounded in $L^{p_{\infty}}(M)$ and $\norm{\psi_i}_{L^\frac{2p_{\infty}}{p_{\infty}-1}} \to 0$, because $\psi_i \overset{W^{1,\frac{2n}{n+1}}}{\rightharpoonup} 0$ and $W^{1,\frac{2n}{n+1}} \hookrightarrow L^\frac{2p_{\infty}}{p_{\infty}-1}$ is compact. This leads to a  contradiction.

\end{proof}

\begin{lemma}
Let $(\beta_m)_{m\in\mathbb{N}}$ be a sequence in $L^n_{\geqslant 0}(M)$ such that ${\beta_m \overset{L^n(M)}{\to} \beta}$, where $\beta$ is a nonzero function. Assume that $(\lambda_k(\beta_m))_{m \in \mathbb{N}}$ is bounded. Let $\varphi_m$ be a generalized eigenspinor corresponding to the generalized eigenvalue $\lambda_k(\beta_m)$, i.e.,
\[
D_g \varphi_m = \lambda_k(\beta_m) \beta_m \varphi_m,
\]
and assume that
\[
\int_M \beta_m \abs{\varphi_m}_g^2 \, dv_g = 1 \quad \text{and} \quad \sup\int_M \beta_m^n dv_g < +\infty.
\]
Then $(\varphi_m)_{m \in \mathbb{N}}$ admits a subsequence that converges weakly in $W^{1,\frac{2n}{n+1}}(M)$.
Moreover, if we denote this weak limit by $\varphi$, then $\varphi$ is nonzero and satisfies 
\[
D_g\varphi = \lambda \beta \varphi,
\]
where $\lambda = \liminf \lambda_k(\beta_m) \geqslant \lambda_k(\beta)$.
\end{lemma}

\begin{proof}
After passing to a subsequence, we may assume that $\lim\limits_{m \to \infty} \lambda_k(\beta_m) = \lambda$. By Lemma~\ref{lemma:lower semi-continuity of lambda}, it follows that $\lambda \geqslant \lambda_k(\beta)$.
By Lemma~\ref{lemma:analogue of lemma 2.1}, $(\varphi_m)_{W^{1,\frac{2n}{n+1}}}$ is bounded in $W^{1,\frac{2n}{n+1}}$. We may, after passing to a subsequence, assume that $\varphi_m \overset{W^{1,\frac{2n}{n+1}}}{\rightharpoonup} \varphi$. By Lemma~\ref{lemma:compact operator from W^1,2n/n+1 to L^2(gamma)}, $\varphi$ is nonzero. Passing to the weak limit in $D_g \varphi_m = \lambda_k(\beta_m)\beta_m\varphi_m$, we deduce that $D_g \varphi = \lambda \beta \varphi$. 
\end{proof}

\section{Variational theory for generalized eigenvalues.}\label{section:variational theory}
In this section we follow~\cite{humbert2025extremisingeigenvaluesgjmsoperators}
and use the following notations:
\begin{itemize}
    \item $i(k) \overset{\mathrm{min}}{=} \min\{\,i \in \mathbb{N}_{>0} \mid \lambda_i(\beta) = \lambda_k(\beta)\,\}$,
    \item  $I(k) \overset{\mathrm{def}}{=} \max\{\,i \in \mathbb{N}_{>0} \mid \lambda_i(\beta) = \lambda_k(\beta)\,\}$.
    \item $E_k(\beta) \overset{\mathrm{def}}{=} \ker (D_g - \lambda_k(\beta)\beta\mathrm{Id})$,
    \item $Q(\beta, \varphi) = \int_M \beta \abs{\varphi}^2dv_g$,
    \item $B(\beta,\varphi,\psi) = \int_M \beta \dprod{\varphi,\psi}dv_g$.
\end{itemize}

\subsection{First variation of generalized eigenvalues.}
\begin{lemma}[Right derivative of $\lambda_k(\beta)$]\label{lemma:right derivative of lambda_k}
Assume that $\beta \in L^n_{\geqslant 0}(M)$ and $\lambda_k(\beta)$ is finite. Then for any $b \in L^n_{\geqslant 0}(M)$ we have
\begin{align*}
\lim_{t \to 0+} \frac{\lambda_k(\beta + tb) - \lambda_k(\beta)}{t} 
&= \min_{V \in \mathcal{G}_{k - i(k) + 1}(E_k(\beta))} \max_{\varphi \in V \setminus \{0\,\}} \left(-\lambda_k(\beta)\frac{\int_M b\abs{\varphi^2}_g dv_g}{\int_M \beta \abs{\varphi}_g^2dv_g}\right)\\
&= \max_{V \in \mathcal{G}_{I(k) - k + 1}(E_k(\beta))} \min_{\varphi \in V \setminus \{0\,\}} \left(-\lambda_k(\beta)\frac{\int_M b\abs{\varphi^2}_gdv_g}{\int_M \beta \abs{\varphi}_g^2dv_g}\right)
\end{align*}
\end{lemma}

\begin{proof}
By Lemma~\ref{lemma:continuity of generalized eigenvalues in positive direction}, we know that $\lim\limits_{t \to 0+} \lambda_k(\beta + t b) = \lambda_k(\beta)$, so there exists $\delta > 0$ such that $\lambda_k(\beta + tb)$ is finite for all $t \in (0,\delta)$. 

We denote by 
\[
\varphi^t_{i(k)},\ldots,\varphi^t_{I(k)}
\]
a family of $Q(\beta + tb,\cdot)$-orthogonal generalized eigenspinors associated with the eigenvalues $\lambda_{i(k)}(\beta + tb),\ldots, \lambda_{I(k)}(\beta + tb)$. Moreover we assume that for all $t \in (0,\delta)$ and for all $j \in \{\,i(k),\ldots,I(k)\,\}$ we have
\[
\int_M (\beta + tb)^2\abs{\varphi^t_j}_g^2 dv_g = 1.
\]
By Lemma~\ref{lemma: H^1 analogue of lemma 2.1}, $(\varphi^t_j)_t$ is bounded in $H^1$, so we may, after passing to a subsequence, assume that $\varphi^t_j \overset{H^1}{\rightharpoonup}\varphi_j$. By Lemma~\ref{lemma:existence of eigenspinors}, we have that $D_g\varphi_j = \left(\lim\limits_{t \to 0+}\lambda_j(\beta + tb)\right) \beta \varphi_j$ and $B(\beta,\varphi_i,\varphi_j) = 0$ for $i \neq j$. Moreover, Lemma~\ref{lemma:continuity of generalized eigenvalues in positive direction} implies that $\lim\limits_{t \to 0+}\lambda_j(\beta + tb) = \lambda_j(\beta) = \lambda_k(\beta)$, where the last equality is true since $i(k) \leqslant j \leqslant I(k)$.

Lemma~\ref{lemma:compact opetator from H^1 to L^2(gamma^2)} implies that 
\[
\int_M (\beta + tb)^2 \abs{\varphi_j^t - \varphi_j}_g^2dv_g = 0.
\]
Consequently, $\int_M \beta \abs{\varphi_j}^2_g = \lim\limits_{t \to 0+} \int_M (\beta + tb)^2\abs{\varphi^t_j}_g^2 dv_g = 1$.
Moreover,
\begin{align*}
\lim_{t \to 0+} \int_M \abs{D_g \varphi^t_j}_g^2 = \lim_{t \to 0+} \lambda_j(\beta + tb)^2 \int_M (\beta + tb)^2 \abs{\varphi^t_j}_g^2dv_g =
\lambda_k(\beta) \int_M\beta^2\abs{\varphi_j}_g^2dv_g.
\end{align*}
Thus, $\lim_{t \to 0+}\norm{\varphi^t_j}_{H^1} = \norm{\varphi_j}_{H^1}$. Also, we know that $\varphi^t_j \overset{H^1}{\rightharpoonup}\varphi_j$. Combining these two observations, we deduce that $\varphi^t_j$ converges strongly to $\varphi_j$ in $H^1$.

Since $\varphi^t_j \overset{L^\frac{2n}{n-1}}{\to} \varphi_j$ and $\beta+tb \overset{L^n}{\to} \beta$, we have that 
\[
\lim_{t \to 0+}Q(\beta+tb,\varphi^t_j) = \lim_{t \to 0+}\int_M (\beta + tb)\abs{\varphi^t_j}_g^2dv_g = \int_M \beta \abs{\varphi_j}_g^2dv_g = Q(\beta,\varphi).
\]
Therefore, 
\[
\phi^t_{i(k)} \overset{\mathrm{def}}{=} \frac{\varphi^t_{i(k)}}{Q\left(\beta+tb,\varphi^t_{i(k)}\right)},\ldots,\phi^t_{I(k)} \overset{\mathrm{def}}{=}\frac{\varphi^t_{I(k)}}{Q\left(\beta+tb,\varphi^t_{I(k)}\right)}
\]
is a family of $Q(\beta + tb,\cdot)$-orthonormal generalized eigenspinors associated to the eigenvalues $\lambda_{i(k)}(\beta + tb),\ldots, \lambda_{I(k)}(\beta + tb)$. Moreover, for all $j \in \{\,1,\ldots,k\,\}$ we have that
\[
\phi^t_j \overset{H^1}{\to} \phi_j \overset{\mathrm{def}}{=} \frac{\varphi_j}{Q\left(\beta,\varphi_j\right)}
\]
and $\phi_{i(k)},\ldots,\phi_{I(k)}$ is a $Q(\beta,\cdot)$-orthonormal basis of $E_k(\beta)$.

For $j \in \{\,i(k),\ldots,I(k)\,\}$ we define $R^t_j = \phi^t_j - \pi_k\left(\phi^t_j\right)$, where $\pi_k$ is the orthogonal projection on $E_k(\beta)$ with respect to $Q(\beta,\cdot)$, i.e. for all $\psi \in H^1$ we have
\[
\pi_k(\psi) = \sum_{i = i(k)}^{I(k)}B(\beta,\psi,\phi_i)\phi_i.
\]
Consequently,
\begin{equation}\label{equation:R^t_j}
D_g R^t_j - \lambda_k(\beta) \beta R^t_j 
= \left(\lambda_j(\beta + tb) - \lambda_k(\beta)\right)\beta \phi^t_j + \lambda_j(\beta + tb)tb\phi^t_j.    
\end{equation}
For $j \in \{\,i(k),\ldots,I(k)\,\}$ and $t > 0$ we define
\begin{equation}\label{equality: alpha}
\alpha^t_j = \abs{\lambda_j(\beta + tb) - \lambda_k(\beta)} + t + \norm{R^t_j}_{H^1}
\end{equation}
and
\begin{equation}\label{equaility: R,tau and delta}
\tilde{R}^t_i = \frac{R^t_i}{\alpha^t_i}, \quad \tau^t_i = \frac{t}{\alpha^t_i}, \quad \delta^t_i = \frac{\abs{\lambda_j(\beta + tb) - \lambda_k(\beta)}}{\alpha^t_i}.
\end{equation}
We may, after passing to a subsequence, assume that 
\[
\tilde{R}^t_j \overset{H^1}{\rightharpoonup} \tilde{R}_j, \quad
\tau^t_j \to \tau_j, \quad \delta^t_j \to \delta_j.
\]
Passing to the weak limit in~\eqref{equation:R^t_j} we obtain
\begin{equation}\label{equation: R_j}
D_g \tilde{R}_j - \lambda_k(\beta) \beta \tilde{R}_j 
= \delta_j\beta \phi_j + \lambda_j(\beta)\tau_jb\phi_j.  
\end{equation}
Thus,
\[
D_g\left(\tilde{R}^t_j - \tilde{R}^t_j\right) = \lambda_k(\beta)\beta\left(\tilde{R}^t_j - \tilde{R}_j\right) + o(1), 
\]
in $L^2$ as $t \to 0+$. Since $\tilde{R}^t_j \overset{H^1}{\rightharpoonup} \tilde{R}$, it follows that $\beta\left(\tilde{R}^t_j - \tilde{R}_j\right) \overset{L^2}{\to} 0$. Indeed, by Lemma~\ref{lemma:compact opetator from H^1 to L^2(gamma^2)}, we have that
\[
\lim_{t \to 0+}\int_M \beta^2 \abs{\tilde{R}^t_j - \tilde{R}_j}^2 = 0. 
\]
Consequently. $D_g\left(\tilde{R}^t_j - \tilde{R}_j\right) = o(1)$ in $L^2$ as $t \to 0+$. Moreover, by Rellich-Kondrashov theorem, $\tilde{R}^t_j \overset{L^2}{\to} R_j$, i.e. $\norm{\tilde{R}^t_j - \tilde{R}^t_j}_{L^2} = o(1)$ as $t \to 0+$. Hence, using elliptic estimates, we get that
\[
\norm{\tilde{R}^t_j - \tilde{R}^t_j}_{H^1} \leqslant C_{\mathrm{elliptic}}\left(\norm{D_g\tilde{R}^t_j - D_g\tilde{R}^t_j}_{L^2} + \norm{\tilde{R}^t_j - \tilde{R}^t_j}_{L^2}\right) = o(1) \quad \text{as } t \to 0+.
\]
Therefore, after passing to a subsequence, we may assume that $\tilde{R}^t_j \overset{H^1}{\to} \tilde{R}_j$ as $t \to 0+$. From~\eqref{equality: alpha} and~\eqref{equaility: R,tau and delta} we deduce that
\begin{equation}\label{equality: sum of R, delta and tau}
\norm{\tilde{R}_j}_{H^1} + \abs{\delta_j} + \tau_j = 1.
\end{equation}
Now we integrate~\eqref{equation: R_j} against $\phi_j$ and we obtain that 
\begin{align*}
\delta_j Q(\beta,\phi_j) + \tau_j \lambda_k(\beta) \int_M b\abs{\phi_j}^2dv_g 
&= \int_M \dprod{D_g \tilde{R}_j,\phi_j}_gdv_g -\lambda_k(\beta) B(\beta,\tilde{R}_j,\varphi_j) \\
&= \int_M \dprod{\tilde{R}_j,D_g\phi_j}_gdv_g - 0\\
&= \lambda_k(\beta)\int_M\beta\dprod{\tilde{R}_j,\phi_j}_gdv_g \\
&= \lambda_k(\beta) B(\beta,\tilde{R}_j,\varphi_j) = 0.
\end{align*}

Assume that $\tau_j = 0$, then since $Q(\beta,\phi_j) = 1$, it follows that $\delta_j = 0$. Therefore, from~\eqref{equation: R_j} we deduce that $D_g \tilde{R}_j = \lambda_k(\beta)\beta\tilde{R}_j$, i.e. $\tilde{R}_j \in E_k(\beta)$. However, by construction of $\tilde{R}_j$, we have that $B(\beta,\tilde{R}_j,\phi) = 0$ for all $\phi$ in $E_k(\beta)$, so in particular we have that $0 = Q(\beta,\tilde{R}_j)$. Consequently, $\tilde{R}_j = 0$ but this contradicts~\eqref{equality: sum of R, delta and tau}.

Thus, $\tau_j \neq 0$ and 
\[
\frac{\delta_j}{\tau_j} = -\lambda_k(\beta) \frac{\int_M b \abs{\phi_j}_g^2}{Q(\beta,\phi)}.
\]
Integrating~\eqref{equation: R_j} against $\phi_i$ for $i \neq j$, we 
obtain that $\lambda_k(\beta)\int_M b\dprod{\phi_i,\phi_j}dv_g = 0$, 
so $\phi_{i(k)},\ldots,\phi_{I(k)}$ is an orthogonal family for the 
bilinear form 
\[
(\varphi,\psi) \mapsto - 
\lambda_k(\beta)\int_Mb\dprod{\varphi,\psi}_gdv_g. 
\]
Moreover, the 
family $\left(\frac{\delta^t_i}{\tau^t_i}\right)_{i(k)\leqslant i 
\leqslant I(k)}$ is non-decreasing in $i$ for every fixed $t > 0$. 
Therefore,
\[
\frac{\delta_{i(k)}}{\tau_{i(k)}} \leqslant \ldots \leqslant \frac{\delta_{I(k)}}{\tau_{I(k)}}.
\]
Using min-max formulae for orthonormal diagonalization, we get that
\[
\frac{\delta_i}{\tau_i} = \min_{V \in \mathcal{G}_{i - i(k) + 1}(E_k(\beta))} \max_{\varphi \in V \setminus \{0\,\}} \left(-\lambda_k(\beta)\frac{\int_M b\abs{\varphi^2}_gdv_g}{\int_M \beta \abs{\varphi}_g^2dv_g}\right).
\]
Since the right-hand side is independent of the choice of a subsequence as $t \to 0+$, we deduce that the directional derivative exists and
\[
\lim_{t \to 0+}\frac{\lambda_k(\beta + tb) - \lambda_k(\beta)}{t} = \lim_{t \to 0+} \frac{\delta^t_k}{\tau^t_k} = \frac{\delta_k}{\tau_k}.
\]
\end{proof}

\subsection{Renormalized eigenvalue functionals.}
Fix $p \geqslant n$. For $\beta \in L^p(M)$ we define the following volume-renormalized version of $\lambda_k(\beta)$:
\[
\bar{\lambda}^p_k(\beta) \overset{\mathrm{def}}{=} 
\begin{cases}
\lambda_k(\beta)\norm{\beta}_{L^p} \quad &\text{if }\lambda_k(\beta) < +\infty\\
+\infty \quad &\text{otherwise}.
\end{cases}
\]
The proof of the following lemma can be carried out step by step in the same manner as the proof of Proposition $4.2$ in~\cite{humbert2025extremisingeigenvaluesgjmsoperators}.
\begin{lemma}\label{lemma:Euler-Lagrange for lambda_k}
Assume that $\beta \in L^p_{\geqslant 0}(M)$ is a local minimum of $\bar{\lambda}^k_p$ and that $\lambda_k(\beta)$ is finite. Then, for any subspace $V \in \mathcal{G}_{k - i(k) + 1}(E_k(\beta))$, there exist $r \in \{\,i(k),\ldots,k\,\}$, a $Q(\beta,\cdot)$-orthonormal family $(\varphi_r, \ldots, \varphi_k)$ of elements of $V$, and positive numbers $d_r, \ldots, d_k$ satisfying $\sum\limits_{i = r}^k d_i = 1$ such that 
\[
\beta^{p - 1} = \sum_{i = r}^k d_i \abs{\varphi_i}^2_g \quad \text{almost everywhere on } M.
\]
\end{lemma}

\begin{remark}\label{remark:regularity}
Combining the conclusion of Lemma~\ref{lemma:Euler-Lagrange for lambda_k} with Lemma~\ref{lemma:regularity}, we deduce using a bootstrap argument that $\beta \in C^{0,\alpha}(M)$ for some $\alpha \in (0,1)$. Moreover, if $p = 2$ which is the case for example when $p = n = 2$, we have that $\beta \in \mathcal{C}^{\infty}(M)$. 
\end{remark}

For $p \geqslant n$ we consider the following variational problem:
\[
\Lambda^p_k(M,[g],\sigma) \overset{\mathrm{def}}{=} \inf_{\beta \in L^p_{\geqslant 0}(M)} \bar{\lambda}^p_k(\beta).
\]

The proof of the next lemma follows exactly the same argument as that of Proposition $4.4$ in~\cite{humbert2025extremisingeigenvaluesgjmsoperators}. However, we note that, in our setting, only the lower semi-continuity of the generalized eigenvalues has been established. Nevertheless, Ekeland's variational principle remains applicable in this case.
\begin{lemma}\label{lemma: existence of almost minimizers}
Let $p \geqslant n$. Define $q \overset{\mathrm{def}}{=} \frac{p - 1}{p}$. For any $\varepsilon > 0$ there exist $\beta_{\varepsilon} \in L^p_{\geqslant 0}(M)$ with $1 \leqslant \norm{\beta_{\varepsilon}}_{L^p} \leqslant 1 +\varepsilon$ such that
\[
\bar{\lambda}^p_k(\beta_{\varepsilon}) \leqslant \Lambda^p_k(M,[g],\sigma) + \varepsilon^2
\]
a strictly positive integer $l_{\varepsilon} \leqslant k$, an orthonormal family $\varphi^{\varepsilon}_{l_{\varepsilon}},\ldots,\varphi^{\varepsilon}_k$ with respect to $Q(\beta_{\varepsilon},\cdot)$ of eigenspinors in $E_k(\beta_{\varepsilon})$, positive numbers $d_{l_{\varepsilon}}, \ldots, d_k$ satisfying $\sum\limits_{i = l_{\varepsilon}}^k d_i = 1$ and a function $f_{\varepsilon} \in L^q(M)$ such that $\norm{f_{\varepsilon}}_{L^q} \leqslant \varepsilon$ and
\[
\bar{\lambda}^p_k(\beta_{\varepsilon})\sum_{i = l_{\varepsilon}}^k d_i \abs{\varphi^{\varepsilon}_i}^2_g \leqslant \bar{\lambda}^p_k(\beta_{\varepsilon})\beta^{p-1}_{\varepsilon} + f_{\varepsilon} \quad \text{almost everywhere on } M.
\]
\end{lemma}

We are ready to establish the existence of optimizers for $\Lambda^p_k(M,[g],\sigma)$ for $p > n$.

\begin{theorem}\label{theorem:existence of supercritical eigenvalue}
Assume that $p > n$. Then, for any $k \in \mathbb{N}_{>0}$, $\Lambda^p_k(M,[g],\sigma)$ is attained, there exists $\beta \in L^p_{\geqslant 0}(M)$ such that $\bar{\lambda}^p_k(\beta) = \Lambda^p_k(M,[g],\sigma)$.
\end{theorem}

\begin{proof}
By Lemma~\ref{lemma: existence of almost minimizers}, for all $m \in \mathbb{N}_{>0}$ there exist $\beta_m \in L^p_{\geqslant 0}(M)$ such that 
\[
1 \leqslant \norm{\beta_m}_{L^p} \leqslant 1 + \frac{1}{m}
\]
and
\begin{equation*}
\bar{\lambda}^p_k(\beta_m) \leqslant \Lambda^p_k(M,[g],\sigma) + \frac{1}{m^2},
\end{equation*}
a strictly positive integer $l_{m} \leqslant k$, an $Q(\beta_{m},\cdot)$-orthonormal family $\varphi^{m}_{l_m},\ldots,\varphi^{m}_k$ of eigenspinors in $E_k(\beta_{m})$, positive numbers $d_{l_{m}}, \ldots, d_k$ satisfying $\sum\limits_{i = l_{m}}^k d_i = 1$ and a function $f_{m} \in L^q(M)$, where $q \overset{\mathrm{def}}{=} \frac{p}{p-1},$ such that $\norm{f_{m}}_{L^q} \leqslant \frac{1}{m}$ and
\[
\bar{\lambda}^p_k(\beta_{m})\sum_{i = l_{m}}^k d_i \abs{\varphi^{m}_i}_g^2 \leqslant \bar{\lambda}^p_k(\beta_{m})\beta^{p-1}_{m} + f_{m} \quad \text{almost everywhere on } M.
\]
We may, after passing to a subsequence, assume that $l_m = l$ does not depend on $m$. Since $p > n$, it follows from Lemma~\ref{lemma:analogue of lemma 2.1} that $(\varphi^m_i)_{m \in \mathbb{N}}$ is bounded in $W^{1,\frac{2n}{n+1}}$. We may, after passing to a subsequence, assume that $\varphi^m_i \overset{W^{1,\frac{2n}{n+1}}}{\rightharpoonup} \varphi_i$ for all $i \in \{\,l,\ldots,k\,\}$. Finally, we may assume that $\beta_m \overset{L^p}{\rightharpoonup} \beta$, $\lim\limits_{m \to \infty}\lambda_k(\beta_m) = \lambda$ and $\lim\limits_{m \to \infty} d^m_i = d_i$ for all $i \in \{\,l,\ldots,k\,\}$. Note that by Lemma~\ref{lemma:analogue of lemma 2.1}, we have $\lambda > 0$.

First of all, we note that $D_g \varphi_i = \lambda \beta \varphi_i$. Indeed, for a fixed smooth spinor $\Phi$ we have that
\begin{align*}
\int_M \dprod{D_g\varphi_i,\Phi}_gdv_g 
= \int_M \dprod{\varphi_i,D_g\Phi}_gdv_g
&= \lim_{m \to \infty} \int_M \dprod{\varphi_i^m,D_g\Phi}_gdv_g \\
&= \lim_{m \to \infty} \int_M \dprod{D_g\varphi_i^m,\Phi}_gdv_g\\
&= \lim_{m \to \infty} \int_M \dprod{\lambda_k(\beta_m)\beta_m\varphi_i^m,\Phi}_gdv_g\\
&=\int_M \dprod{\lambda\beta\varphi_i,\Phi}_gdv_g,
\end{align*}
where the last equality follows from the fact that $\beta_m \overset{L^p(M)}{\rightharpoonup} \beta$ and $\varphi_i^m \overset{L^\frac{p}{p-1}}{\to}\varphi_i$. Consequently, $D_g\varphi_i = \lambda \beta \varphi_i$.

By Lemma~\ref{lemma:regularity}, we have that $\varphi_i \in L^r$ for all $r \in [1,+\infty)$. Since $\beta \in L^p(M)$, where $p > n$, it follows that $p > 2$. Therefore, $\lambda\beta\varphi_i \in L^2$. Indeed, by Sobolev embedding theorem and H\"older's inequality, we have
\[
\int_M \beta^2 \abs{\varphi_i}_g^2dv_g \leqslant \left(\int_M \beta^p dv_g\right)^\frac{2}{p} \left(\int_M \abs{\varphi_i}_g^\frac{2p}{p-2} \right)^\frac{p-2}{p} < +\infty.
\]
By elliptic regularity, $(\varphi_i)_i$ is bounded in $H^1$. 

By the Rellich-Kondrashov theorem, we may assume after passing to a subsequence that $\varphi^m_i \overset{L^\frac{2p}{p-2}}{\to}\varphi_i$. Consequently, 
\[
\delta_{ij} = \lim_{m \to \infty} \int_M \beta_m \dprod{\varphi^m_i,\varphi^m_j}_g dv_g = \int_M \beta \dprod{\varphi_i,\varphi_j}_gdv_g,
\]
i.e. $(\varphi_l,\ldots,\varphi_k)$ is a $Q(\beta,\cdot)$-orthonormal family. Moreover, 
\[
\left( \sum_{i = l}^k d^m_i \abs{\varphi^m_i}_g^2 \right)^\frac{1}{p - 1} \overset{L^p(M)}{\to} \left( \sum_{i = l}^k d_i \abs{\varphi_i}_g^2 \right)^\frac{1}{p - 1}.
\]
Since pointwise inequalities are preserved under weak convergence and $\lambda > 0$, we obtain
\[
\sum_{i = l}^k d_i \abs{\varphi_i}_g^2 \leqslant \beta^{p - 1} \quad \text{almost everywhere on } M.
\]
Consequently,
\[
1 = \sum_{i = l}^m d_i = \sum_{i = l}^m \int_M d_i \beta\abs{\varphi_i}_g^2dv_g \leqslant \int_M \beta^p dv_g. 
\]
However, $\norm{\beta}_{L^p} \leqslant \liminf \norm{\beta_m}_{L^p} = 1$. We deduce that $\norm{\beta}_{L^p}$. Since $L^p(M)$ is uniformly convex, it follows that 
\[
\beta_m \overset{L^p(M)}{\to} \beta.
\]
By Lemma~\ref{lemma:lower semi-continuity of lambda}, it follows $\lambda_k(\beta) \leqslant \lim\limits_{m \to \infty}\lambda_k(\beta_m) = \lambda$. Therefore,
\[
\bar{\lambda}^p_k(\beta) = \lambda_k(\beta) \leqslant \lambda = \lim_{m \to \infty}\bar{\lambda}^p_k(\beta_m) = \Lambda^p_k(M,[g],\sigma). 
\]
\end{proof}

We conclude this section by showing that $\Lambda_k(M,[g],\sigma)$ is actually equal to the infimum of $\bar{\lambda}_k$ in $[g]$.
\begin{proposition}
\[
\Lambda_k(M,[g],\sigma) = \inf_{\beta \in \mathcal{C}^{\infty}_{>0}(M)} \bar{\lambda}_k(\beta).
\]
\end{proposition}
\begin{proof}
Fix $\varepsilon > 0$ and take $\beta \in L^n_{\geqslant 0}(M)$ such that $\bar{\lambda}_k(\beta) \leqslant \Lambda_k(M,[g],\sigma) + \varepsilon$. By Lemma~\ref{lemma:continuity of generalized eigenvalues in positive direction}, we can find $\delta > 0$ such that $\bar{\lambda}_k(\beta + \delta) \leqslant \bar{\lambda}_k(\beta) + \varepsilon$. Now by Proposition~\ref{theorem:continuity in L^n_delta}, we can approximate $\beta + \delta$ by a smooth strictly positive function $\tilde{\beta}$ such that $\bar{\lambda}_k(\tilde{\beta}) \leqslant \bar{\lambda}_k(\beta + \delta) + \varepsilon$. Consequently,
\[
\bar{\lambda}_k(\tilde{\beta}) \leqslant \Lambda_k(M,[g],\sigma) + 3\varepsilon.
\]
That is,
\[
 \Lambda_k(M,[g],\sigma) = \inf_{\beta \in \mathcal{C}^{\infty}_{>0}}\bar{\lambda}_k(\beta).
\]
\end{proof}

\section{Generalized spinorial analogue of Aubin's inequality.}\label{section:Aubin}
In~\cite{Ammann2008}, Ammann, Grosjean, Humbert and Morel proved the following result:
\begin{theorem}[\cite{Ammann2008}]
Let $(M,g,\sigma)$ be an $n$-dimensional closed Riemannian spin manifold. Then,
\[
\Lambda_1(M,[g],\sigma) \leqslant \Lambda_1(\mathbb{S}^n) = \frac{n}{2}\omega_n^\frac{1}{n},
\]
where $\omega_n$ stands for the volume of the standard sphere $\mathbb{S}^n$.
\end{theorem}
In this section, we extend their theorem to eigenvalues of arbitrary index on locally conformally flat manifolds. We begin by proving the following lemma, which will be instrumental in the sequel.

\begin{lemma}\label{lemma:lcf manifods}
Let $(M,g)$ be a locally conformally flat Riemannian manifold. Then for each finite subset $A \subset M$ there exists a metric $\tilde{g} \in [g]$ such that $\tilde{g}$ is flat around each point of $A$.
\end{lemma}
\begin{proof}
Denote elements of $A$  by $x_1,\ldots, x_r$. Since $g$ is locally conformally flat, it follows that for each $x_i \in A$  there exists an open ball $B_g(x_i,2r_i)$ and a smooth function ${u_i \in \mathcal{C}^{\infty}(B_g(x_i,3r_i))}$ such that $g = e^{2u_i}\xi$, where $\xi$ is a standard flat metric. Without loss of generality, we may assume that $B_g(x_i,3r_i) \cap B_g(x_j,3r_j) = \varnothing$ for $i \neq j$. Take a cut-off function $\eta_i$ such that $\restr{\eta_i}{B_g(x_i,r_i)} = 1$ and $\restr{\eta_i}{M \setminus B_g(x_i,2r_i)} = 0$. Consider a smooth function $u = -\sum\limits_{i = 1}^r \eta_i u_i$, then $\tilde{g} = e^{2u}g$ is a desired metric. Indeed, $\restr{u}{B_g(x_i,r_i)} = -u_i$, i.e.,
\[
\tilde{g} = e^{2u}g = e^{2u}e^{2u_i}\xi = e^{-2u_i + 2u_i}\xi = \xi.
\]
\end{proof}

Having established the necessary preliminary results, we now state and prove the generalized Aubin-type inequality.

\begin{theorem}\label{theorem:Aubin-type}
Let $(M,g,\sigma)$ be an $n$-dimensional closed locally conformally flat Riemannian spin manifold. Then
\begin{equation*}
\Lambda_k(M,[g],\sigma) \leqslant \inf \{\,\left(\Lambda_{l_0}(M,[g],\sigma)^n + \Lambda_{l_1}(\mathbb{S}^n)^n + \ldots + \Lambda_{l_r}(\mathbb{S}^n)^n\right)^\frac{1}{n}\,\},
\end{equation*}
where the infimum is taken over the set of all $r,l_0,\ldots,l_r$ in $\mathbb{N}$ such that
\begin{enumerate}
    \item $l_0 < k$ and $\Lambda_{0}(M,[g],\sigma) = 0$,
    \item $l_0 + l_1 + \ldots + l_r = k$.
\end{enumerate}
\end{theorem}
\begin{proof}
Let $r,l_0,\ldots,l_r$ in $\mathbb{N}$ be such that
\[
l_0 + \ldots + l_r = k.
\]
Fix $\varepsilon > 0$. Then for all $i \in \{\,1\ldots,r\,\}$ there exists $\beta_i \in \mathcal{C}^\infty_{>0}(\mathbb{S}^n)$ such that $\norm{\beta_i}_{L^n} = 1$ and 
\begin{equation*}
\lambda_{l_i}(\beta_i) \leqslant \Lambda_{l_i}(\mathbb{S}^n) + \varepsilon. 
\end{equation*}
If $l_0 = 0$, we let $\beta_0 = \mathrm{Vol}(M,g)^{-\frac{1}{n}}$. Otherwise, there exists $\beta_0 \in \mathcal{C}^{\infty}_{>0}(M)$ such that $\norm{\beta_0}_{L^n} = 1$ and
\begin{equation*}
\lambda_{l_0}(\beta_0) \leqslant \Lambda_{l_0}(M,[g],\sigma) + \varepsilon. 
\end{equation*}

By Theorem~\ref{theorem:existence of eigenspinors. Positive case.}, for all $i \in \{\,1,\ldots,r\,\}$ and $j \in \{\,1,\ldots,l_i\,\}$, there exists $\varphi^i_j \in H^1(\Sigma\mathbb{S}^n)$ such that
\begin{itemize}
    \item $\int_{\mathbb{S}^n} \beta_i \dprod{\varphi^i_{j_1},\varphi^i_{j_2}}_{g_{\mathrm{st}}}dv_{g_\mathrm{st}} = \delta_{j_1,j_2}$ for all $j_1,j_2 \in \{\,1,\ldots,l_i\,\}$,
    \item $D_{g_{\mathrm{st}}}\varphi^i_j = \lambda^i_j \beta_i \varphi^i_j$, where $\lambda^i_j \overset{\mathrm{def}}{=} \lambda_j(\beta_i)$.
\end{itemize}
Fix a point $p \in \mathbb{S}^n$. Then the stereographic projection $\pi \colon (\mathbb{S}^n,g_{\mathrm{st}}) \to (\mathbb{R}^n,\xi)$ is a conformal diffeomorphism. More precisely, $(\pi^{-1})^*g_{\mathrm{st}}(x) = U(x)^2\xi$, where $U(x) = \frac{2}{1 + \abs{x}^2}$. As it is shown in~\cite{Bourguignon1992}, there is an isomorphism of vector bundles
\[
T \colon \Sigma(\mathbb{S}^n\setminus\{\,p\,\}) \to \Sigma\mathbb{R}^n, 
\]
which is a fiberwise isometry. We define $\tau \colon \Sigma(\mathbb{S}^n\setminus\{\,p\,\}) \to \Sigma\mathbb{R}^n$ by 
\[
\tau(\varphi) = T(U^{\frac{n-1}{2}}\varphi).
\]
Then, we have $D_{\xi}\tau(\varphi) = U\tau(D_{g_\mathrm{st}}\varphi)$ (see~\cite{Hijazi1986}) and $\abs{\tau(\varphi)}_{\xi} = U^{\frac{n-1}{2}}\abs{\varphi}_{g_\mathrm{st}}$. We define $\tilde{\beta}_i \colon \mathbb{R}^n \to \mathbb{R}$ by 
\[
\tilde{\beta_i}(x) = U(x) \beta_i(\pi^{-1}(x)).
\]
Therefore, for $\tilde{\varphi}^i_j \overset{\mathrm{def}}{=} \tau(\varphi^i_j)$ we have that
\begin{equation*}
D_{\xi}\tilde{\varphi}^i_j = \lambda^i_j \tilde{\beta}_i\tilde{\varphi}^i_j.  
\end{equation*}
Moreover,
\begin{equation*}
\int_{\mathbb{R}^n}\tilde{\beta}_i\dprod{\tilde{\varphi}^i_{j_1},\tilde{\varphi}^i_{j_2}}_{\xi}dv_{\xi} = \delta_{j_1,j_2} \quad \text{and} \quad \int_{\mathbb{R}^n}\tilde{\beta}_i^n dv_{\xi} = 1. 
\end{equation*}
For all $a > 0$, we define
\begin{equation*}
\psi^{i,a}_j(x) \overset{\mathrm{def}}{=} a^\frac{n - 1}{2}\tilde{\varphi}^i_j(ax) \quad \text{and} \quad \beta_{i,a}(x) \overset{\mathrm{def}}{=} a\tilde{\beta}_i(ax).
\end{equation*}
Therefore, we get that
\begin{align*}
D_{\xi}\psi^{i,a}_j(x) 
= a^\frac{n+1}{2} D_{\xi}\tilde{\varphi}^i_j(ax) 
&= \lambda^i_j\beta_{i,a}(x)\psi^{i,a}_l(x).
\end{align*}
Moreover, 
\[
\int_M \beta_{i,a}^ndv_{\xi} = 1.
\]
We now fix $r$ distinct points $x_1,\ldots,x_r$ on $M$.xLet $\delta > 0$ and $\eta_i \in \mathcal{C}^{\infty}(M)$ be a cut-off function such that 
\[
\chi_{B_g(x_i,\delta)} \leqslant \eta_i \leqslant \chi_{B_g(x_i,2\delta)}.
\]
If $\delta$ is small enough, then the balls $B_g(x_i,2\delta)$ are disjoint and the exponential mapping 
\[
\exp_{x_i}^{-1} \colon B_g(x_i,2\delta) \to V_i
\]
is a diffeomorphism, where $V_i \subset \mathbb{R}^n$ is an open neighborhood of $0$. Moreover, by Lemma~\ref{lemma:lcf manifods}, we may assume that $g = \xi$ on $B_g(x_i,2\delta)$ for all $i \in \{\,1,\ldots,r\,\}$.
As it is shown in~\cite{Bourguignon1992}, there is an isomorphism of vector bundles
\[
T_i \colon \Sigma B_g(x_i,2\delta) \to \Sigma V_i.
\]
and $D_g \varphi = T_i^{-1}D_\xi T_i(\varphi)$.
We define
\[
\mu_i = \lambda^i_{l_i} \quad \text{for all } 1 \leqslant i \leqslant r, \quad \text{and} \quad \mu_{0} = \lambda^{0}_{l_0}(1-\delta_{l_0,0}),
\]
that is, $\mu_0$ is equal to $0$ if $l_0 = 0$, and to $\lambda^0_{l_0}$ otherwise. Also, for $a > 0$ and $c > 0$, we let
\[
\beta_{a} \overset{\mathrm{def}}{=} \mu_0\beta_0 + \sum_{i=1}^r \mu_i\eta_i\beta_{i,a}\circ \exp^{-1}_{x_i} + \frac{1}{\sqrt{a}}
\]
and
\[
\Psi^{i,a}_j = \eta_i T_i^{-1}\left(\restr{\psi^{i,a}_j}{V_i}\right) \quad \text{for } i \in \{\,1,\ldots,r\,\} \text{ and } j \in \{\,1,\ldots,l_i\,\}.
\]
In particular, we obtain
\begin{align*}
T_i(D_g \Psi^{i,a}_j) = D_\xi(\eta_i \psi^{i,a}_j).
\end{align*}
As in~\cite{humbert2025extremisingeigenvaluesgjmsoperators}, we have that 
\[
\lim_{a \to +\infty}\int_M \beta_{a}^n dv_g = \sum_{i = 0}^r\mu_i^n.
\]
Now we want to show the following inequality:
\begin{equation}\label{eq:<Dpsi,psi>}
\lim_{a \to \infty} \int_M \dprod{D_g \Psi^{i,a}_j, \Psi^{i',a}_{j'}}dv_g 
= \lambda^i_j\delta_{i,i'}\delta_{j,j'}
\end{equation}
Indeed, if $i \neq i'$, then $\mathrm{supp}(\eta_i) \cap \mathrm{supp}(\eta_{i'}) = \varnothing$ and the integral above is equal to $0$. Hence, we may assume that $i = i'$. Consequently,
\begin{align*}
\int_M \dprod{D_g \Psi^{i,a}_j, \Psi^{i,a}_{j'}}dv_g 
&= \int_{V_i} \eta_i^2\dprod{D_{\xi}\psi^{i,a}_j,\psi^{i,a}_{j'}}_{\xi}dv_\xi + \int_{V_i} \eta_i \dprod{\grad_\xi \eta_i \cdot \psi^{i,a}_j,\psi^{i,a}_{j'}}dv_\xi.
\end{align*}
Observe that
\begin{equation}\label{eq:L^2-norm of psi goes to 0}
\lim_{a \to \infty} \int_{\mathbb{R}^n}\abs{\psi^{i,a}_j(x)}^2dv_\xi(x) 
= 0,
\end{equation}
because 
\begin{equation}\label{eq:L^2 psi-varphi}
\int_{\mathbb{R}^n}\abs{\psi^{i,a}_j(x)}^2dv_\xi(x) 
=
 \int_{\mathbb{R}^n}a^{n-1}\abs{\tilde{\varphi}^i_j(ax)}dv_\xi(x)
=\frac{1}{a}\int_{\mathbb{R}^n}\abs{\tilde{\varphi}^i_j(x)}^2dv_\xi(x).
\end{equation}
Thus,
\[
\lim_{a \to \infty} \int_{V_i} \eta_i \dprod{\grad_\xi \eta_i \cdot \psi^{i,a}_j,\psi^{i,a}_{j'}}dv_\xi = 0
\]
Moreover, the first term tends to $\lambda^i_j\delta_{j,j'}$ as $a$ goes to $\infty$. 

We now want estimate $\int_M \frac{1}{\beta_a}\dprod{D_g \Psi^{i,a}_j,D_g\Psi^{i',a}_{j'}}dv_g$. Once again, we may assume that $j = j'$ because otherwise this integral is equal to $0$. We note that 
\begin{align*}
\int_M \frac{1}{\beta_a}\dprod{D_g \Psi^{i,a}_j,D_g\Psi^{i,a}_{j'}}dv_g 
&= \int_{V_i} \frac{1}{\beta_a}\eta_i^2\dprod{D_{\xi}\psi^{i,a}_j,D_\xi\psi^{i,a}_{j'}}_{\xi}dv_\xi \\
&+ \int_{V_i} \frac{1}{\beta_a}\eta_i\dprod{D_{\xi}\psi^{i,a}_j,\grad_\xi \eta_i \cdot \psi^{i,a}_{j'}}_{\xi}dv_\xi \\
&+ \int_{V_i} \frac{1}{\beta_a}\eta_i\dprod{\grad_\xi \eta_i \cdot \psi^{i,a}_j,D_{\xi} \psi^{i,a}_{j'}}_{\xi}dv_\xi \\
&+\int_{V_i} \frac{1}{\beta_a}\dprod{\grad_\xi \eta_i \cdot \psi^{i,a}_j,\grad_\xi \eta_i \cdot \psi^{i,a}_{j'}}_{\xi}dv_\xi
\end{align*}
The first integral tends to $\frac{(\lambda^i_j)^2}{\mu_i}$ as $a$ goes to $\infty$. For the second and the third ones, we have
\begin{align*}
\int_{V_i} \frac{1}{\beta_a}\eta_i\dprod{D_{\xi}\psi^{i,a}_j,\grad_\xi \eta_i \cdot \psi^{i,a}_{j'}}_{\xi}dv_\xi
= 
\int_{V_i} \frac{\lambda^i_j\beta_{i,a}}{\beta_a}\eta_i\dprod{\psi^{i,a}_j,\grad_\xi \eta_i \cdot \psi^{i,a}_{j'}}_{\xi}dv_\xi.
\end{align*}
Note that $\frac{\eta_i\beta_{i,a}}{\beta_a} \leqslant 1$. Thus, using~\eqref{eq:L^2-norm of psi goes to 0}, we deduce that
\[
\lim_{a \to \infty} \int_{V_i} \frac{1}{\beta_a}\eta_i\dprod{D_{\xi}\psi^{i,a}_j,\grad_\xi \eta_i \cdot \psi^{i,a}_{j'}}_{\xi}dv_\xi = \lim_{a \to \infty} \int_{V_i} \frac{1}{\beta_a}\eta_i\dprod{\grad_\xi \eta_i \cdot \psi^{i,a}_j,D_{\xi} \psi^{i,a}_{j'}}_{\xi}dv_\xi = 0
\]
Finally,
\begin{align*}
\abs{\int_{V_i} \frac{1}{\beta_a}\dprod{\grad_\xi \eta_i \cdot \psi^{i,a}_j,\grad_\xi \eta_i \cdot \psi^{i,a}_{j'}}_{\xi}dv_\xi}
\leqslant \norm{\grad_\xi \eta_i}_{L^\infty(M)}^2 \abs{\int_{V_i} \sqrt{a}\abs{\psi^{i,a}_j}_\xi \abs{\psi^{i,a}_{j'}}dv_\xi},
\end{align*}
where we used $\beta_a \geqslant \frac{1}{\sqrt{a}}$. By H\"older's inequality,
\begin{align*}
\int_{\mathbb{R}^n} \sqrt{a}\abs{\psi^{i,a}_j}_\xi \abs{\psi^{i,a}_{j'}}dv_\xi 
\leqslant \sqrt{a}\norm{\psi^{i,a}_{j}}_{L^2(\Sigma\mathbb{R}^n)}\norm{\psi^{i,a}_{j}}_{L^2(\Sigma\mathbb{R}^n)}
= \frac{\sqrt{a}}{a}\norm{\tilde{\varphi}^{i,a}_{j}}_{L^2(\Sigma\mathbb{R}^n)}\norm{\tilde{\varphi}^{i,a}_{j}}_{L^2(\Sigma\mathbb{R}^n)},
\end{align*}
where we used~\eqref{eq:L^2 psi-varphi} for the equality. Consequently,
\[
\lim_{a \to \infty} \abs{\int_{V_i} \frac{1}{\beta_a}\dprod{\grad_\xi \eta_i \cdot \psi^{i,a}_j,\grad_\xi \eta_i \cdot \psi^{i,a}_{j'}}_{\xi}dv_\xi} = 0.
\]
Therefore we deduce that 
\[
\lim_{a \to \infty}\int_M \frac{1}{\beta_a}\dprod{D_g \Psi^{i,a}_j,D_g\Psi^{i,a}_{j'}}dv_g  = \frac{(\lambda^i_j)^2}{\mu_i}\delta_{i,i'}\delta_{j,j'}.
\]
Moreover, if we take any spinor $\varphi$ on $M$ such that $D_g \varphi = \lambda \beta_0 \varphi$, where $\lambda \in \mathbb{R}$, then
\begin{align*}
\lim_{a \to \infty} \int_M \dprod{D_g \varphi, \Psi^{i,a}_j}_g dv_g = \lim_{a \to \infty} \int_M \dprod{\varphi, D_g\Psi^{i,a}_j}_g dv_g = 0     
\end{align*}
and 
\begin{align*}
\lim_{a \to \infty}\int_M \frac{1}{\beta_a} \dprod{D_g \varphi, D_g \Psi^{i,a}_j}_g dv_g = 0
\end{align*}
\[
\Psi^{0,a}_j = \begin{cases}
0 \quad &\text{if} \quad l_0 = 0,\\
\varphi_j \quad &\text{otherwise},
\end{cases}
\]
where $D_g \varphi_j = \lambda_{l_0}^0\beta_0 \varphi_j$.
Consequently, for all $i,i' \in \{\,0,\ldots,r\,\}$ for all $j \leqslant l_i$ and $j' \leqslant l_{i'}$ we have
\begin{equation*}
\lim_{a \to \infty} \int_M \dprod{D_g \Psi^{i,a}_j, \Psi^{i',a}_{j'}}dv_g 
= \mu^i_j\delta_{i,i'}\delta_{j,j'}
\end{equation*}
and 
\begin{equation*}
\lim_{a \to \infty} \int_M \frac{1}{\beta_a}\dprod{D_g \Psi^{i,a}_j,D_g\Psi^{i,a}_{j'}}dv_g =  \frac{(\lambda^i_j)^2}{\mu_i} \delta_{i,i'}\delta_{j,j'}
\end{equation*}
Define a vector space 
\[
V_a = 
\begin{cases}
\spanning{\Psi^{i,a}_j \colon i \in \{\,0,\ldots,r\,\}, j \in \{\,1,\ldots,l_i\,\}} \quad \text{if} \quad l_0 > 0,\\
\spanning{\Psi^{i,a}_j \colon i \in \{\,1,\ldots,r\,\}, j \in \{\,1,\ldots,l_i\,\}} \quad \text{otherwise}.
\end{cases}
\]
As we have proved, $\dim V_a = k$ for $a$ big enough. Note that
\[
\Lambda_k(M,[g],\sigma)^{-1} \geqslant \norm{\beta_a}^{-1}_{L^n(M)} \min_{\varphi \in V_a \setminus \{\,0\,\}} \mathcal{F}(\beta_a)[\varphi]
\]
Choose $\varphi_a$ such that $\mathcal{F}(\beta_a)[\varphi_a] = \min_{\varphi \in V_a \setminus \{\,0\,\}} \mathcal{F}(\beta_a)[\varphi]$. Then without loss of generality, 
\[
\varphi_a = \sum_{i,j}\alpha_{a,ij}\Psi^{i,a}_j, \quad \text{where }\sum_{i,j}\abs{\alpha_{a,ij}}^2 = 1.
\]
After passing to a subsequence, we may assume that $\alpha_{a,ij} \to \alpha_{ij}$ as $a$ goes to $\infty$. Thus,
\[
\lim_{a \to \infty} \int_M \dprod{D_g \varphi_a, \varphi_a} dv_g = \sum_{ij}\alpha^2_{ij}\lambda^i_j 
\]
and 
\[
\lim_{a \to \infty} \int_M \frac{1}{\beta_a}\dprod{D_g \varphi_a, D_g \varphi_a} dv_g = \sum_{ij}\alpha^2_{ij}\frac{(\lambda^i_j)^2}{\mu_i} \leqslant \sum_{ij}\alpha^2_{ij}\lambda^i_j.
\]
Hence, 
\begin{align*}
\Lambda_k(M,[g],\sigma)^{-1}
&\geqslant \left(\sum_{i = 0}^r\mu_i^n\right)^{-\frac{1}{n}}.
\end{align*}
Recall that 
\[
\left(\sum_{i = 0}^r\mu_i^n\right)^{-\frac{1}{n}}  \geqslant \left(\Lambda_{l_0}(M,[g],\sigma)^n + \Lambda_{l_1}(\mathbb{S}^n)^n + \ldots + \Lambda_{l_r}(\mathbb{S}^n)^n + (r + 1) \varepsilon \right)^{-\frac{1}{n}}.
\]
Since $\varepsilon > 0$ was arbitrary, it follows that
\[
\Lambda_k(M,[g],\sigma)^{-1} \geqslant \left(\Lambda_{l_0}(M,[g],\sigma)^n + \Lambda_{l_1}(\mathbb{S}^n)^n + \ldots + \Lambda_{l_r}(\mathbb{S}^n)^n\right)^{-\frac{1}{n}}.
\]
\end{proof}

\section{Minimization of eigenvalues.}
\label{section:Minimization}
The goal of this section is to prove the following theorem
\begin{theorem}\label{theorem:the main theorem}
Let $(M,g,\sigma)$ be an $n$-dimensional closed locally conformally flat Riemannian spin manifold. Assume that $\Lambda_k(M,[g],\sigma)$ is not attained. Then:
\begin{itemize}
    \item Either $\exists l_0 \in \{\,1,\ldots,k - 1\,\}$ and $l_1 \in \{\,1,\ldots, k - 1\,\}$ such that $l_0 + l_1 = k$, $\Lambda_{l_0}(M,[g],\sigma)$ is attained and
\[
\Lambda_k(M,[g],\sigma)^n = \Lambda_{l_0}(M,[g],\sigma)^n + \Lambda_{l_1}(\mathbb{S}^n,[g_{\mathrm{st}}],\sigma_{\mathrm{st}})^n
\]
    \item or 
\[
\Lambda_k(M,[g],\sigma) = \Lambda_k(\mathbb{S}^n,[g_{\mathrm{st}}],\sigma_{\mathrm{st}}).
\]
\end{itemize}
\end{theorem}
First of all, by Theorem~\ref{theorem:existence of supercritical eigenvalue}, we know that, for all $p > n$, $\Lambda^p_k(M,[g],\sigma)$ is attained at some $\beta_p \in L^p(M)$. Therefore, by Lemma~\ref{lemma:Euler-Lagrange for lambda_k}, there exists $l_p \in \{\,1,\ldots,k\,\}$ such that $\lambda_{l_p}(\beta_p) = \lambda_k(\beta_p)$, a $Q(\beta_p,\cdot)$-orthonormal family $(\varphi_{l_p,p},\ldots,\varphi_{k,p})$ in $E_k(\beta_p)$ and positive numbers $d_{l_p,p},\ldots,d_{k,p}$ satisfying $\sum\limits_{i = l_p}^k d_{i,p} = 1$ such that
\begin{equation}\label{equation:beta_p^p-1}
\beta^{p - 1}_p = \sum^k_{i = l_p}d_{i,p}\abs{\varphi_{i,p}}^2_g.
\end{equation}
After passing to a subsequence, we may assume that $l_p$ does not depend on $p$ and is equal to $l$. Moreover, we can find eigenspinors $\varphi_{1,p},\ldots,\varphi_{l-1,p}$ associated to eigenvalues $\lambda_1(\beta_p),\ldots,\lambda_{l - 1}(\beta_p)$ such that the family $(\varphi_{1,p},\ldots,\varphi_{k,p})$ is $Q(\beta_p,\cdot)$-orthonormal. In order to simplify the notations, we denote $\lambda_i(\beta_p)$ by $\lambda_{i,p}$.

After passing to a subsequence, we may assume that 
$\lim\limits_{p \to n} d_{i,p} = d_i$ and $\lim\limits_{p \to \infty} \lambda_{i,p} = \lambda_i$. Note that $\sum\limits_{i = 1}^k d_i = 1$ and $\lambda_k = \Lambda_k(M,[g],\sigma)$. Moreover, by Lemma~\ref{lemma:analogue of lemma 2.1}, $\lambda_i > 0$. Define two sets
\[
I \overset{\mathrm{def}}{=} \{\,i \in \{\,l,\ldots,k\,\} \colon d_i > 0\,\} \quad \text{and} \quad J \overset{\mathrm{def}}{=} \{\,l,\ldots,k\,\} \setminus I.
\]

The proofs of the following two lemmas are identical to that of Lemma 6.1 and Lemma 6.2 in~\cite{humbert2025extremisingeigenvaluesgjmsoperators}.
\begin{lemma}
If $i \in I$, then $(\varphi_{i,p})$ is bounded in $W^{1,\frac{2n}{n+1}}$.
\end{lemma}

\begin{lemma}\label{lemma:tilde beta}
For $p > n$, we define
\[
\tilde{\beta}_p = \left(\beta_p^{p-1} - \sum_{i \in J} d_{i,p} \abs{\varphi_{i,p}}^2_g \right)^\frac{1}{p-1}.
\]
Then $\lim\limits_{p \to n+0} \norm{\beta_p - \tilde{\beta}_p}_{L^p(M)} = 0$.
\end{lemma}


We may, after passing to a subsequence, assume that $\beta_p \overset{L^n(M)}{\rightharpoonup} \beta$. Moreover, if $(\varphi_{i,p})_p$ is bounded in $W^{1,\frac{2n}{n+1}}$, we may assume that $\varphi_{i,p} \overset{W^{1,\frac{2n}{n+1}}}{\rightharpoonup} \varphi_i$ for all $i \in \{\,l,\ldots,k\,\}$.

By Lemma~\ref{lemma:tilde beta}, we also have that $\tilde{\beta}_p \overset{L^n(M)}{\rightharpoonup} \beta$. Note that 
\[
\tilde{\beta}_p^{p-1} = \sum_{i \in I} d_{i,p}\abs{\varphi_{i,p}}^2_g.
\]
By Rellich-Kondrashov theorem, the right-hand side of this equality converges strongly in $L^q(M)$ for $q_0 < \frac{n}{n-1}$. Denote this strong limit by $\Upsilon_1$. After passing to a subsequence, we may assume that $\beta_p^{p - 1}$ converges to $\Upsilon_1$ almost everywhere. Consequently, $\beta_p$ converges almost everywhere to $\Upsilon_1^\frac{1}{n-1}$.

Moreover, by the characterization of strong convergence in $L^{q_0}(M)$, there exists a function $\Upsilon_2 \in L^{q_0}(M)$ such that $\beta_p^{p-1} \leqslant \Upsilon_2$ almost everywhere. Without loss of generality, we may assume that $\Upsilon_2 \geqslant 1$ almost everywhere, because we can always replace $\Upsilon_2$ by $\Upsilon_2 + 1$. Therefore, $\Upsilon_2^\frac{1}{p-1} \leqslant \Upsilon_2^\frac{1}{n-1}$ almost everywhere. In particular, $\tilde{\beta}_p \leqslant \Upsilon_2^\frac{1}{n-1}$ almost everywhere.

Let's prove that $\tilde{\beta}_p$ converges strongly in $L^{q_1}(M)$ for some $q_1 > \frac{2n}{n+1}$. First of all, note that we can choose $q_0$ such that
$(n-1)q_0 > \frac{2n}{n+1}$ for all $n\geqslant 2$. Let's fix $q_1$ in this interval. Since $\frac{q_1}{n-1} < q_0$, it follows that $\Upsilon_2^\frac{q_1}{n-1} \in L^1(M)$. Hence, by the dominated convergence theorem, $(\tilde{\beta}_p)_{p}$ converges strongly to $\Upsilon_1^\frac{1}{n-1}$ in $L^{q_1}(M)$.

Since $\tilde{\beta}_p \overset{L^n(M)}{\rightharpoonup}\beta$ and since the weak limit is unique, it follows that $\tilde{\beta}_p \overset{L^{q_1}(M)}{\to} \beta$ for some $q_1 > \frac{2n}{n+1}$. Thus, by Lemma~\ref{lemma:tilde beta}, $\beta_p \overset{L^{q_1}(M)}{\to} \beta$. It means that we are in the setting of Lemma~\ref{lemma: strong local convergence}. This observation allows us to prove the following proposition.

\begin{lemma}\label{lemma:final version of local convergence}
Define 
\[
A \overset{\mathrm{def}}{=} \left\{\,x \in M \mid \forall \delta > 0, \quad \limsup_{m \to +\infty} \int_{B_{g}(x,\delta)} \beta_m^n dv_g  > \frac{1}{2^n}\left(\frac{\Lambda_1(\mathbb{S}^n,[g_{\mathrm{st}}],\sigma_{\mathrm{st}})}{\Lambda_k(M,[g],\sigma)}\right)^\frac{n}{4} \,\right\}
\]
Then, as $p \to n$, we have $\beta_p \overset{L^n_\mathrm{loc}(M\setminus A)}{\to} \beta$ and $\varphi_{i,p} \overset{W^{1,\frac{2n}{n+1}}_\mathrm{loc}(\Sigma(M \setminus A))}{\to} \varphi_i$ for all $i$ such that $(\varphi_{i,p})_p$ is bounded.
\end{lemma}
\begin{proof}
Let $i$ be an index such that $(\varphi_{i,p})_{p}$ is bounded in $W^{1,\frac{2n}{n+1}}$. Note that such index exists because $I \neq \varnothing$. Then, by Lemma~\ref{lemma: strong local convergence}, we have that $(\varphi_{i,p})_p$ converges strongly in $W^{1,\frac{2n}{n+1}}_\mathrm{loc}(\Sigma(M \setminus A))$. Consequently, $\tilde{\beta}_p$ converges strongly to $\beta$ in $L^n_\mathrm{loc}(M \setminus A)$. By Lemma~\ref{lemma:tilde beta}, we have that $\beta_p \overset{L^n_\mathrm{loc}(M \setminus A)}{\to} \beta$.
\end{proof}

Recall that we assumed that $\Lambda_k(M,[g],\sigma)$ is not attained. This implies that ${A \neq \varnothing}$. Indeed, if it was not the case, then, by Lemma~\ref{lemma:final version of local convergence}, $\beta_p \overset{L^n(M)}{\to}\beta$. Also, since in this case $\norm{\beta}_{L^n(M)} = 1$, it follows from Lemma~\ref{lemma:analogue of lemma 2.1}, that $(\varphi_{i,p})_p$ is bounded in $W^{1,\frac{2n}{n+1}}$. Consequently, by Lemma~\ref{lemma:final version of local convergence}, $\varphi_{i,p} \overset{W^{1,\frac{2n}{n+1}}}{\to} \varphi_i$. In particular, we obtain that $(\varphi_i)_{1 \leqslant i \leqslant k}$ is a $Q(\beta,\cdot)$-orthonormal family and $D_g \varphi_i = \lambda_i \beta \varphi_i$ for all $i \in \{\,1,\ldots,k\,\}$. By Lemma~\ref{lemma:lower semi-continuity of lambda}, we have that $\lambda_k(\beta) \leqslant \liminf\limits_{p \to n}\lambda_k(\beta_p) = \lambda_k = \Lambda_k(M,[g],\sigma)$. However, since $\norm{\beta}_{L^n(M)} = 1$, it follows that $\Lambda_k(M,[g],\sigma) \leqslant \lambda_k(\beta)$. Thus, $\Lambda_k(M,[g],\sigma)$ is attained by $\beta$.

There are two different cases, namely, either $\beta \equiv 0$, or $\beta  \not \equiv 0$. We start with the first one.

\subsection{The null case.}
Assume that $\beta \equiv 0$, then Lemma~\ref{lemma:analogue of lemma 2.1} motivates us to work with $(\psi_{i,p})_p$ instead of $(\varphi_{i,p})_{p}$, where $\varphi_{i,p} = \psi_{i,p} + \kappa_{i,p}$ with $\psi_{i,p} \in ker D_g^\perp$ and $\kappa_{i,p} \in \ker D_g$.  By Lemma~\ref{lemma:analogue of lemma 2.1}, $(\psi_{i,p})_p$ is bounded in $W^{1,\frac{2n}{n+1}}$, so we may, after passing to a subsequence, assume that $\psi_{i,p} \overset{W^{1,\frac{2n}{n+1}}}{\rightharpoonup} \psi_i$. 

If $(\varphi_{i,p})_p$ is not bounded in $W^{1,\frac{2n}{n+1}}$, then Lemma~\ref{lemma:analogue of lemma 2.1} guarantees that $\psi_i = 0$. Assume now that $(\varphi_{i,p})$ is bounded in $W^{1,\frac{2n}{n+1}}$, then the weak limit of this sequence, which we denoted by $\varphi_i$, verifies $D_g \varphi_i = \lambda_i \beta_i \varphi_i$. Recall that $\beta \equiv 0$, and hence $\varphi_i \in \ker D_g$. Consequently, $\psi_i = 0$, because $\psi_i$ is exactly the orthogonal projection of $\varphi_i$ on $\ker D_g^\perp$.

First, we write the results of Lemma~\ref{lemma: strong local convergence} for this sequences.
\begin{lemma}\label{lemma:local convergence for orthogonal projections}
After passing to a subsequence, we have that $\psi_{i,p}$ converges strongly to $0$ in $W^{1,\frac{2n}{n+1}}_{\mathrm{loc}}(\Sigma(M \setminus A))$ as $p$ goes to $n$.
\end{lemma}

\begin{proof}
We have just shown that $\psi_{i,p} \overset{W^{1,\frac{2n}{n+1}}}{\rightharpoonup} 0$. If $(\varphi_{i,p})_{p}$ is bounded in $W^{1,\frac{2n}{n+1}}$, then so is $(\kappa_{i,p})_p$. Since $\dim \ker D_g < \infty$, it follows that we may, after passing to a subsequence, assume that $(\kappa_{i,p})_p$ converges strongly in $W^{1,\frac{2n}{n+1}}$. Thus, by Lemma~\ref{lemma: strong local convergence}, $\psi_{i,p}$ converges strongly to $0$ in $W^{1,\frac{2n}{n+1}}_{\mathrm{loc}}(\Sigma(M \setminus A))$.

Now, assume that $(\varphi_{i,p})_p$ is not bounded in $W^{1,\frac{2n}{n+1}}$. Lemma~\ref{lemma:analogue of lemma 2.1} implies that $\beta \equiv 0$. Choose a function $\eta \in \mathcal{C}^{\infty}_c(M \setminus A)$ such that $\norm{\eta}_{L^\infty(M\setminus A)} \leqslant 1$. Then
\begin{align*}
D_g(\eta \psi_{i,p}) = \eta \lambda_{i,p}\beta_p \varphi_{i,p} + \grad_g \eta \cdot \psi_{i,p}.   
\end{align*}
By Rellich-Kondrashov theorem, $\grad_g \eta \cdot \psi_{i,p} \overset{L^\frac{2n}{n+1}}{\to} 0$. By H\"older's inequality,
\begin{align*}
\norm{\eta \lambda_{i,p}\beta_p \varphi_{i,p}}_{L^\frac{2n}{n+1}(M)} 
&\leqslant \lambda_{i,p}\left(\int_M\abs{\eta \beta_p}^ndv_g\right)^\frac{1}{n+1} \left(\int_M \beta_p \abs{\varphi_{i,p}}^2_g dv_g\right)^\frac{n}{n+1}\\
&= \lambda_{i,p} \norm{\eta \beta_p}^\frac{n}{n+1}_{L^n(M)}.
\end{align*}
By Lemma~\ref{lemma:final version of local convergence}, $\eta \beta_p \overset{L^n(M)}{\to} 0$, because $\beta = 0$, and $\lambda_{i,p} \to \lambda_i$ as $p$ goes to $n$, it follows that $\lim\limits_{p \to n} \norm{\eta \lambda_{i,p}\beta_p \varphi_{i,p}}_{L^\frac{2n}{n+1}} = 0$. Once again, by Rellich-Kondrashov theorem, $\eta\psi_{i,p} \overset{L^\frac{2n}{n+1}}{\to} 0$. Thus, $\psi_{i,p}$ converges strongly to $0$ in $W^{1,\frac{2n}{n+1}}_\mathrm{loc}(\Sigma(M \setminus A))$.
\end{proof}

As in~\cite{humbert2025extremisingeigenvaluesgjmsoperators}, we fix $\delta > 0$ such that $B(x_1,2\delta) \cap B(x_2,2\delta) = \varnothing$ for all $x_1 \neq x_2 \in A$. Take $\eta \in \mathcal{C}^\infty([0,\infty))$ such that $\restr{\eta}{[0,\delta]} \equiv 1$, $\restr{\eta}{[2\delta,+\infty)} \equiv 0$ and $\sup \abs{\eta^{(m)}} \leqslant \frac{C}{\delta^m}$, where $C$ does not depend on $\delta$. We introduce a function $\chi \overset{\mathrm{def}}{=} \sum\limits_{x \in A} \eta(d_g(x,\cdot))$, where $d_g$ is a metric on $(M,g)$.   We also define $\Psi_{i,p} \overset{\mathrm{def}}{=} \chi \psi_{i,p}$ for each $i \in \{\,1,\ldots,k\,\}$. Our aim now is to understand the behavior of 
\[
\int_M \dprod{D_g\Psi_{i,p},\Psi_{j,p}}_g dv_g 
\quad \text{and} \quad 
\int_M \frac{1}{\beta_p + \delta}\dprod{D_g \Psi_{i,p}, D_g \Psi_{j,p}}_g dv_g
\] 
as $p$ goes to $n$ and $\delta$ goes to $0$.
We start by proving the following lemma:
\begin{lemma}\label{lemma:7.6}
\[
\lim_{p \to n}\abs{\int_M \dprod{D_g\Psi_{i,p},\Psi_{j,p}}_g dv_g - \int_M \dprod{D_g\psi_{i,p},\psi_{j,p}}_g dv_g} = 0.
\]
\end{lemma}
\begin{proof}
By definition of $\Psi_{i,p}$, we get that
\begin{align*}
\int_M \dprod{D_g \Psi_{i,p},\Psi_{j,p}}_gdv_g = \int_M \left(\chi^2 \dprod{D_g \psi_{i,p},\psi_{j,p}}_g + \chi \dprod{\grad_g \chi \cdot \psi_{i,p},\psi_{j,p}}_g\right)dv_g.
\end{align*}
Recall that, by Lemma~\ref{lemma:local convergence for orthogonal projections}, $\psi_{i,p}$ converges strongly to $0$ in $W^{1,\frac{2n}{n+1}}(\Sigma(M \setminus A))$. Therefore, we have that the second integral goes to $0$ as $p$ goes to $n$.

Moreover,
\begin{align*}
\int_M \chi\dprod{D_g \psi_{i,p},\psi_{j,p}}_gdv_g 
&= \left(\int_{\bigcup_{x \in A}B(x,\delta)} + \int_{M \setminus \bigcup_{x \in A}B(x,\delta)}\right)\chi\dprod{D_g \psi_{i,p},\psi_{j,p}}_gdv_g\\
&=\int_{\bigcup_{x \in A}B(x,\delta)}\dprod{D_g \psi_{i,p},\psi_{j,p}}_gdv_g + o(1),
\end{align*}
because $\restr{\chi}{\bigcup_{x \in A}B(x,\delta)} \equiv 1$.
Applying Lemma~\ref{lemma:local convergence for orthogonal projections}, we deduce that 
\begin{align*}
\int_{M \setminus \bigcup_{x \in A}B(x,\delta)} \dprod{D_g \psi_{i,p},\psi_{j,p}}_gdv_g = o(1).
\end{align*}
Thus,
\[
\int_M \chi\dprod{D_g \psi_{i,p},\psi_{j,p}}_gdv_g = \int_M \dprod{D_g \psi_{i,p},\psi_{j,p}}_g dv_g + o(1),
\]
which in turn yields that
\[
\lim_{p \to n}\abs{\int_M \dprod{D_g\Psi_{i,p},\Psi_{j,p}}_g dv_g - \int_M \dprod{D_g\psi_{i,p},\psi_{j,p}}_g dv_g} = 0.
\]
\end{proof}
We now proceed to the integral $\int_M \frac{1}{\beta_p + \delta}\dprod{D_g \Psi_{i,p},D_g \Psi_{j,p}}_g dv_g$ for which the similar statement holds.

\begin{lemma}\label{lemma:7.7}
\[
\abs{\int_M \frac{1}{\beta_p + \delta}\dprod{D_g \Psi_{i,p},D_g \Psi_{j,p}}_g dv_g - \int_M \frac{1}{\beta_p + \delta}\dprod{D_g \psi_{i,p},D_g \psi_{j,p}}_g dv_g} = 0.
\]
\end{lemma}
\begin{proof}
Note that
\begin{align*}
\int_M \frac{1}{\beta_p + \delta}\dprod{D_g \Psi_{i,p},D_g \Psi_{j,p}}_g dv_g 
&= \int_M \frac{\abs{\grad_g\chi}_g^2}{\beta_p + \delta}\dprod{\psi_{i,p},\psi_{j,p}}_g dv_g\\ 
&+ \int_M\frac{\chi}{\beta_p + \delta}\dprod{\grad_g \chi \cdot \psi_{i,p},D_g\psi_{j,p}}_gdv_g \\
&+ \int_M\frac{\chi}{\beta_p + \delta}\dprod{D_g\psi_{i,p},\grad_g \chi \cdot \psi_{j,p}}_gdv_g \\
&+ \int_M \frac{\chi^2}{\beta_p + \delta}\dprod{D_g\psi_{i,p},D_g \psi_{j,p}}_gdv_g. 
\end{align*}

Since $\beta_p \geqslant 0$ almost everywhere, it follows that $\frac{1}{\beta_p + \delta} \in L^\infty(M)$. This observation implies that, due to Lemma~\ref{lemma:local convergence for orthogonal projections}, the first, the second and the third summand converge to $0$.

We now consider the last one. By Lemma~\ref{lemma:local convergence for orthogonal projections}, we have
\begin{align*}
\int_M {\chi^2}\frac{\dprod{D_g \psi_{i,p}, D_g \psi_{j,p}}_g}{\beta_p + \delta} dv_g 
&= \left(\int_{\bigcup_{x \in A} B(x,\delta)}  + \int_{M \setminus \bigcup_{x \in A} B(x,\delta)}\right)   {\chi^2}\frac{\dprod{D_g \psi_{i,p}, D_g \psi_{j,p}}_g}{\beta_p + \delta} dv_g\\
&= \int_{\bigcup_{x \in A}B(x,\delta)}\frac{\dprod{D_g \psi_{i,p}, D_g \psi_{j,p}}_g}{\beta_p + \delta} dv_g + o(1).
\end{align*}
Moreover, applying Lemma~\ref{lemma:local convergence for orthogonal projections} once again, we get that
\[
\int_{M \setminus \bigcup_{x \in A}B(x,\delta)}\frac{\dprod{D_g \psi_{i,p}, D_g \psi_{j,p}}_g}{\beta_p + \delta} dv_g = o(1).
\]
Thus,
\[
\int_M \frac{\dprod{D_g \Psi_{i,p}, D_g \Psi_{j,p}}_g}{\beta_p + \delta} dv_g  = \int_M \frac{\dprod{D_g \psi_{i,p}, D_g \psi_{j,p}}_g}{\beta_p + \delta} dv_g + o(1). 
\]
\end{proof}

We use now that the metric $g$ is locally conformally flat. Therefore, by Lemma~\ref{lemma:lcf manifods}, we may assume that $g$ is flat around each point of $A$. Moreover, $\delta$ is chosen in such a way that $g = g_{\mathrm{st}}$, where $g_{\mathrm{st}}$ is the standard round metric on $\mathbb{S}^n$, on $B(x,2\delta)$ for all $x \in A$. 

Let's fix $r$ distinct points $y_1,\ldots,y_r$ in $\mathbb{S}^n$ such that $B_{g_{\mathrm{st}}}(y_i,2\delta) \cap B_{g_{\mathrm{st}}}(y_j,2\delta) = \varnothing$ for all $\delta$ small enough and $i \neq j$. Note that there is an isometry $f_i \colon B_g(x_i,2\delta) \to B_{g_{\mathrm{st}}}(y_i,2\delta)$. As it is shown in~\cite{Bourguignon1992}, $f_i$ induces an isomorphism of vector bundles 
\[
T_i \colon \Sigma_g B_g(x_i,2\delta) \to \Sigma_{g_\mathrm{st}} B_{g_\mathrm{st}}(y_i,2\delta),
\]
which is a fiberwise isometry. Moreover, $D_{g_\mathrm{st}} T_i(\phi) = T_i(D_g(\phi))$ for all $\phi \in \Sigma_g(B_g(x_i,2\delta))$. 

We define test spinors $\tilde{\Psi}_{j,p} \overset{\mathrm{def}}{=} \sum\limits_{i = 1}^rT_i\left(\restr{\Psi_{j,p}}{B_g(x_i,2\delta)}\right)$. Note that these spinors are well defined, because $(\Psi_{i,p})_{i,p}$ are supported in $\bigcup\limits_{i = 1}^r B_g(x_i,2\delta)$. We also introduce a function $\theta_{p} \overset{\mathrm{def}}{=} \beta_p \circ f^{-1}$, where $\theta_p$ is extended by zero outside of the image of $f$. 

Consequently,
\[
\int_{\mathbb{S}^n} \dprod{D_{g_\mathrm{st}}\tilde{\Psi}_{i,p},\tilde{\Psi}_{j,p}}_{g_\mathrm{st}} dv_{g_\mathrm{st}} = \int_{M}\dprod{D_g\Psi_{i,p},\Psi_{j,p}}_g dv_g
\]
and
\[
\int_{\mathbb{S}^n} \frac{\dprod{D_{g_\mathrm{st}} \tilde{\Psi}_{i,p}, D_{g_\mathrm{st}} \tilde{\Psi}_{j,p}}_{g_{\mathrm{st}}}}{\theta_p + \delta} dv_{g_\mathrm{st}} = \int_M \frac{\dprod{D_g \Psi_{i,p}, D_g \Psi_{j,p}}_g}{\beta_p + \delta} dv_g.
\]
We are now ready to prove Theorem~\ref{theorem:the main theorem} when $\beta \equiv 0$.

\begin{proof}[Proof of Theorem~\ref{theorem:the main theorem} when $\beta \equiv 0$]
Consider a vector space $W_p \overset{\mathrm{def}}{=} \spanning{\tilde{\Psi}_{1,p},\ldots,\tilde{\Psi}_{k,p}}$. Let $\tilde{\Psi}_p = \sum\limits_{i = 1}^k a_{i,p}\tilde{\Psi}_{i,p}$ be a spinor such that 
that $\sum\limits_{i = 1}^k \abs{a_{i,p}}^2 = 1$. By Lemma~\ref{lemma:7.6} and Lemma~\ref{lemma:7.7}, for $\psi_p = \sum\limits_{i = 1}^k a_i \psi_{i,p}$, we have
\[
\int_{\mathbb{S}^n} \dprod{D_{g_\mathrm{st}}\tilde{\Psi},\tilde{\Psi}}_{g_\mathrm{st}} dv_{g_\mathrm{st}} = \int_M \dprod{D_g \psi_p, \psi_p}_g dv_g + o(1)
\]
and
\[
\int_{\mathbb{S}^n} \frac{\abs{D_{g_\mathrm{st}} \tilde{\Psi}}^2_{g_{\mathrm{st}}}}{\theta_p + \delta} dv_{g_\mathrm{st}} = \int_M \frac{\abs{D_g \psi_p}^2}{\beta_p + \delta}dv_g +  o(1) \leqslant \int_M \frac{\abs{D_g \psi_p}^2}{\beta_p}dv_g +  o(1).
\]
Consequently, 
\[
\min_{\tilde{\Psi} \in W_p \setminus \{\,0\,\}} \frac{\int_{\mathbb{S}^n} \dprod{D_{g_\mathrm{st}}\tilde{\Psi},\tilde{\Psi}}_{g_\mathrm{st}} dv_{g_\mathrm{st}}}{\int_{\mathbb{S}^n} \frac{\abs{D_{g_\mathrm{st}} \tilde{\Psi}}^2_{g_{\mathrm{st}}}}{\theta_p + \delta} dv_{g_\mathrm{st}}} \geqslant \frac{1}{\lambda_k(\beta_p)} + o(1).
\]
Hence, 
\begin{align*}
\Lambda_k(\mathbb{S}^n,[g_\mathrm{st}],\sigma_\mathrm{st})^{-1}
\geqslant \norm{\theta_p + \delta}_{L^n(\mathbb{S}^n)}^{-1} (\lambda_k(\beta_p)^{-1} + o(1)).
\end{align*}
As $\delta$ goes to $0$, we obtain that 
\[
\Lambda_k(\mathbb{S}^n,[g_\mathrm{st}],\sigma_\mathrm{st})^{-1}
\geqslant \norm{\theta_p}_{L^n(\mathbb{S}^n)}^{-1} (\lambda_k(\beta_p)^{-1} + o(1)).
\]
And since $\norm{\theta_p}_{L^n(\mathbb{S}^n)} \leqslant \norm{\beta_p}_{L^n(M)}$, we deduce that, passing to the limit as $p$ goes to $n$, 
\[
\Lambda_k(\mathbb{S}^n,[g_\mathrm{st}],\sigma_\mathrm{st}) \leqslant \Lambda_k(M,[g],\sigma).
\]
However, by Theorem~\ref{theorem:Aubin-type}, we have that the above inequality is in fact an equality. 
\end{proof}

\subsection{The non-null case.}
We now assume that $\beta \not \equiv 0$. Following notations in~\cite{humbert2025extremisingeigenvaluesgjmsoperators}, we let $\lambda_1',\ldots, \lambda_m'$ be the distinct values of $\lambda_1,\ldots, \lambda_k$. For $i \in \{\,1,\ldots,k\,\}$, we define the sets of indices $S_i \overset{\mathrm{def}}{=}\{\, j \in \{\,1,\ldots,k\,\} \mid \lambda_j = \lambda_i' \,\}$. 

We define $V_{i,p} \overset{\mathrm{def}}{=} \spanning{\varphi_{j,p} \mid j \in S_i}$ and $V_p = \bigoplus\limits_{i = 1}^n V_{i,p}$. Hence, $V_p = \spanning{\varphi_{1,p},\ldots, \varphi_{k,p}}$. The bilinear form $B(\beta_p,\cdot,\cdot) = \int_M \beta_p \dprod{\cdot,\cdot}_g dv_g$ defines an Hermitian scalar product on $V_{i,p}$ and $(\varphi_{j,p})_{j \in S_i}$ is an orthonormal basis of $V_{i,p}$ for $B(\beta_p,\cdot,\cdot)$. 

Note that $B(\beta,\cdot,\cdot)$ is an Hermitian form on $V_{i,p}$, that's why we can find a basis $(\phi_{j,p})_{j \in S_i}$ of $V_{i,p}$ such that $B(\beta_p,\phi_{j_1,p},\phi_{j_2,p}) = \delta_{j_1,j_2}$ for all $j_1,j_2 \in S_i$ and $B(\beta,\phi_{j_1,p},\phi_{j_2,p}) = 0$ for all $j_1 \neq j_2 \in S_i$.

Assume that $j \in S_i$. Note that there exist complex numbers $(\alpha_{r,p})_{r \in S_i}$ such that $\sum\limits_{r \in S_i} \abs{\alpha_{r,p}}^2 = 1$ and $\phi_{j,p} = \sum\limits_{r \in S_i} \alpha_{r,p}\varphi_{r,p}$. Therefore,
\begin{equation}\label{equation:phi_i,p}
D_g \phi_{j,p} = \beta_p \sum\limits_{r \in S_i} \lambda_{r,p}\alpha_{r,p}\varphi_{r,p} = \lambda_i' \beta_p \phi_{j,p} + \beta_p \tilde{\phi}_{j,p},
\end{equation}
where $\tilde{\phi}_{j,p} = \sum\limits_{r \in S_i} (\lambda_{r,p} - \lambda_i')\alpha_{r,p}\varphi_{r,p}$. Since $(\varphi_{r,p})_{p}$ is bounded in $W^{1,\frac{2n}{n+1}}$ and $\lambda_{r,p} \to \lambda_i'$ as $p$ goes to $n$, it follows that $\norm{\tilde{\phi}_{j,p}}_{W^{1,\frac{2n}{n+1}}} = o(1)$ and  $\lim\limits_{p \to n} Q(\beta_p,\tilde{\phi}_{j,p}) = 0$, where $Q(\beta_p,\cdot) = \int_M \beta_p \abs{\cdot}^2_g dv_g$. 

Note that $(\phi_{j,p})_p$ is bounded in $W^{1,\frac{2n}{n+1}}$, hence we may assume that $\phi_{j,p} \overset{W^{1,\frac{2n}{n+1}}}{\rightharpoonup} \phi_j$. Thus, passing to the weak limit in~\eqref{equation:phi_i,p}, we deduce that $D_g \phi_j = \lambda_i'\beta\phi_j$, where $j \in S_i$. Moreover for all $j_1 \neq j_2 \in S_i$ we have $B(\beta,\phi_{j_1},\phi_{j_2}) = 0$. Indeed, by Rellich-Kondrashov theorem, $\phi_{j_1,p} \overset{L^\frac{2n}{n+1}}{\to}\phi_{j_1}$ and, by Sobolev embedding theorem, $\phi_{j_2,p} \overset{L^\frac{2n}{n-1}}{\rightharpoonup} \phi_{j_2}$, that's why
\[
0 = \lim_{p \to n}B(\beta,\phi_{j_1,p},\phi_{j_2,p}) = B(\beta,\phi_{j_1},\phi_{j_2}).
\]

Applying the same construction to all $i \in \{\,1,\ldots,m\,\}$, we obtain spinors $\phi_1,\ldots,\phi_k$ such that $D_g \phi_j = \lambda'_i \beta \phi_j$, where $j \in S_i$, and $B(\beta,\phi_{j_1},\phi_{j_2}) = 0$ for $j_1 \neq j_2 \in S_i$. Note that we can not say anything about $B(\beta,\phi_{j_1},\phi_{j_2})$ if $j_1 \in S_{i_1}$ and $j_2 \in S_{i_2}$ with $i_1 \neq i_2$. 

Define a set $\Omega_0 \overset{\mathrm{def}}{=} {\{\,i \in \{\,1,\ldots,k\,\} \mid \beta^\frac{1}{2} \phi_i \not \equiv 0\,\}}$. We demote that its cardinality by $\omega_0$. Note that $B(\beta,\cdot,\cdot)$ defines an Hermitian product on $\spanning{\phi_j \mid j \in \Omega_0}$. Moreover, $(\phi_j)_j \in \Omega_0$ is an orthonormal family with respect to $B(\beta,\cdot,\cdot)$. This observation implies that $\dim \spanning{\phi_j \mid j \in \Omega_0} = \omega_0$. In particular, 
\[
\lambda_{\omega_0}(\beta)^{-1} \geqslant \min_{\phi \in \spanning{\phi_j \mid j \in \Omega_0} \setminus \{\,0\,\}} \frac{\int_M \dprod{D_g \phi,\phi}_g dv_g}{\int_M \frac{\abs{D_g \phi}_g^2}{\beta}dv_g} \geqslant \lambda_k^{-1}.
\]
Thus, 
\begin{equation}\label{eq:Lambda(M)}
\Lambda_{\omega_0}(M,[g],\sigma) \leqslant \lambda_{\omega_0}(\beta)\norm{\beta}_{L^n(M)} \leqslant \lambda_k \norm{\beta}_{L^n(M)} = \Lambda_k(M,[g],\sigma) \norm{\beta}_{L^n(M)}.
\end{equation}

We now consider $j \in \{\,1,\ldots,k\,\} \setminus \Omega_0$. Our first claim is that $\phi_j = 0$. Indeed, by definition of $\Omega_0$, $\beta \phi_j = 0$ and therefore $D_g \phi_j = 0$. By the unique continuation property, either $\phi_j = 0$ or $\abs{\phi_j}_g > 0$ almost everywhere.if the latter occurs, it would follow that $\beta \equiv 0$, which is a contradiction. 

In particular, Lemma~\ref{lemma:final version of local convergence} implies that $\phi_{j,p}$ converges strongly to $0$ in $W^{1,\frac{2n}{n+1}}_\mathrm{loc}(\Sigma_g (M \setminus A))$. Thus, using exactly the same reasoning as for the sequences $(\psi_{i,p})_{p}$ in the null-case, we can construct  $(k-\omega_o)$-dimensional vector subspaces $W_p$ of spinors on $\mathbb{S}^n$ such that   
\[
\min_{\tilde{\Psi} \in W_p \setminus \{\,0\,\}} \frac{\int_{\mathbb{S}^n} \dprod{D_{g_\mathrm{st}}\tilde{\Psi},\tilde{\Psi}}_{g_\mathrm{st}} dv_{g_\mathrm{st}}}{\int_{\mathbb{S}^n} \frac{\abs{D_{g_\mathrm{st}} \tilde{\Psi}}^2_{g_{\mathrm{st}}}}{\theta_p + \delta} dv_{g_\mathrm{st}}} \geqslant \frac{1}{\lambda_k(\beta_p)} + o(1).
\]
However, if we define $\theta_p$ as in the null-case, we obtain that
\[
\int_{\mathbb{S}^n}\theta_p^n dv_{g_\mathrm{st}} = \int_{\bigcup_{x \in A}B(x,2\delta)} \beta_p^n dv_g = \int_M \beta_p^n dv_g -  \int_{M \setminus \bigcup_{x \in A}B(x,2\delta)} \beta_p^n dv_g. 
\]
As $p$ goes to $n$, the first integral tends to $1$, while,  by Lemma~\ref{lemma:final version of local convergence},  the second one tends to $\int_{M \setminus \bigcup_{x \in A}B(x,2\delta)} \beta^n dv_g$. Consequently, for all $p$ we have 
\[
\Lambda_{k - \omega_o}(\mathbb{S}^n,[g_\mathrm{st}],\sigma_\mathrm{st})^{-1} \geqslant \left(1 - \int_{M \setminus \bigcup_{x \in A}B(x,2\delta)} \beta^n dv_g + o(1)\right)^{-\frac{1}{n}}\left(\frac{1}{\lambda_k(\beta)} + o(1)\right),
\]
and hence
\begin{equation}\label{eq:Lambda(S^n)}
\Lambda_{k - \omega_0}(\mathbb{S}^n,[g_\mathrm{st}],\sigma_\mathrm{st}) \leqslant \left(1 - \int_{M} \beta^n dv_g \right)^{\frac{1}{n}} \Lambda_k(M,[g],\sigma).
\end{equation}

We are now ready to finish the proof of Theorem~\ref{theorem:the main theorem}.

\begin{proof}[Proof of Theorem~\ref{theorem:the main theorem} when $\beta \not\equiv 0$] Combining inequalities~\eqref{eq:Lambda(M)} and~\eqref{eq:Lambda(S^n)}, we get that
\begin{equation}\label{eq:Lambda(M) + Lambda(S^n)}
\Lambda_{\omega_0}(M,[g],\sigma)^n + \Lambda_{k - \omega_o}(\mathbb{S}^n,[g_\mathrm{st}],\sigma_\mathrm{st})^n \leqslant \Lambda_k(M,[g],\sigma)^n.    
\end{equation}
Moreover, Theorem~\ref{theorem:Aubin-type} implies that the above inequality is an equality.

If $\omega_0 = 0$, then we are done, so we assume that it is not the case. It requires to show that if $\omega_0 \neq 0$, then $\Lambda_{\omega_o}(M,[g],\sigma)$. 

Since~\eqref{eq:Lambda(M) + Lambda(S^n)} is an equality, it follows that both~\eqref{eq:Lambda(M)} and~\eqref{eq:Lambda(M) + Lambda(S^n)}, but then~\eqref{eq:Lambda(M)} implies that $\Lambda_{\omega_0}(M,[g],\sigma) = \lambda_{\omega_0}(\beta)\norm{\beta}_{L^n(M)}$, i.e., $\Lambda_{\omega_0}(M,[g],\sigma)$ is attained.
\end{proof}

Moreover, if $(M,g)$ is conformally diffeomorphic to the round sphere only the first case can take place. The proof proceeds identically to that of~\cite{humbert2025extremisingeigenvaluesgjmsoperators}.

\begin{lemma}\label{lemma:main theorem for S^n}
If $(M,g)$ is conformally diffeomorphic to $(\mathbb{S}^n,g_\mathrm{st})$, then $\beta \not \equiv 0$ and $\omega_0 > 0$, i.e., if $\Lambda_k(\mathbb{S}^n,[g_\mathrm{st}],\sigma_\mathrm{st})$ is not attained, then $\exists \omega_0 \in \{\,1,\ldots,k-1\,\}$ such that $\Lambda_{\omega_0}(\mathbb{S}^n,[g_\mathrm{st}],\sigma_\mathrm{st})$ is attained and
\[
\Lambda_k(\mathbb{S}^n,[g_\mathrm{st}],\sigma_\mathrm{st})^n = \Lambda_{\omega_0}(\mathbb{S}^n,[g_\mathrm{st}],\sigma_\mathrm{st})^n + \Lambda_{k - \omega_0}(\mathbb{S}^n,[g_\mathrm{st}],\sigma_\mathrm{st})^n.
\]
\end{lemma}

\subsection{Existence of a minimizer.}
A direct application of Theorem~\ref{theorem:the main theorem} is the existence of a minimizer provided $\Lambda_k$ is small enough. The proof follows exactly the same argument as that of Theorem 1.2 in~\cite{humbert2025extremisingeigenvaluesgjmsoperators}. We give it here for the sake of completeness.

\begin{theorem}\label{theorem:main existence theorem}
Let $(M,[g],\sigma)$ be a locally conformally flat Riemannian spin manifold, then $\Lambda_k(M,[g],\sigma)$ is attained if 
\begin{equation}\label{ineq:main inequality}
\Lambda_k(M,[g],\sigma) < \inf_{l_0 + \ldots + l_r = k}
\left(\Lambda_{l_0}(M,[g],\sigma)^n + \Lambda_{l_1}(\mathbb{S}^n,[g_\mathrm{st}],\sigma_\mathrm{st})^n + \ldots + \Lambda_{l_r}(\mathbb{S}^n,[g_\mathrm{st}],\sigma_\mathrm{st})^n \right)^\frac{1}{n},    
\end{equation}
where 
\begin{itemize}
    \item either $l_0 = 0$ and $\Lambda_{l_0}(M,[g],\sigma) = 0$ or $\Lambda_{l_0}(M,[g],\sigma)$ is attained and $l_0 < k$,
    \item for all $i >0$, $l_i > 0$ and $\Lambda_{l_i}(\mathbb{S}^n,[g_\mathrm{st}],\sigma_\mathrm{st})$ is attained.
\end{itemize}
\end{theorem}
\begin{remark}
Theorem~\ref{theorem:Aubin-type} implies that $\Lambda_k(M,[g],\sigma)$ is bounded above by the right-hand side of inequality~\eqref{ineq:main inequality}.
\end{remark}
\begin{proof}
Assume that $\Lambda_k(M,[g],\sigma)$ is not attained. Then, by Theorem~\ref{theorem:the main theorem}, we have that either 
$\Lambda_k(M,[g],\sigma)^n = \Lambda_{\omega_0}(M,[g],\sigma)^n + \Lambda_{k - \omega_0}(\mathbb{S}^n,[g_{\mathrm{st}}],\sigma_{\mathrm{st}})^n$,
where $\Lambda_{\omega_0}(M,[g],\sigma)$ is attained and $\omega_0 > 0$, or 
$\Lambda_k(M,[g],\sigma) = \Lambda_k(\mathbb{S}^n,[g_{\mathrm{st}}],\sigma_{\mathrm{st}})$.

Assume that the first possibility takes place. From inequality~\eqref{ineq:main inequality} it follows that $\Lambda_{k - \omega_0}(\mathbb{S}^n,[g_{\mathrm{st}}],\sigma_{\mathrm{st}})^n$ is not attained. Therefore, by Lemma~\ref{lemma:main theorem for S^n}, there exists $\omega_1 > 0$ such that
\[
\Lambda_{k - \omega_0}(\mathbb{S}^n,[g_{\mathrm{st}}],\sigma_{\mathrm{st}})^n = \Lambda_{\omega_1}(\mathbb{S}^n,[g_{\mathrm{st}}],\sigma_{\mathrm{st}})^n + \Lambda_{k - \omega_0 - \omega_1}(\mathbb{S}^n,[g_{\mathrm{st}}],\sigma_{\mathrm{st}})^n,
\]
where $\Lambda_{\omega_1}(\mathbb{S}^n,[g_{\mathrm{st}}],\sigma_{\mathrm{st}})$ is attained. 
Consequently,
\[
\Lambda_k(M,[g],\sigma)^n = \Lambda_{\omega_0}(M,[g],\sigma)^n + \Lambda_{\omega_1}(\mathbb{S}^n,[g_{\mathrm{st}}],\sigma_{\mathrm{st}})^n + \Lambda_{k - \omega_0 - \omega_1}(\mathbb{S}^n,[g_{\mathrm{st}}],\sigma_{\mathrm{st}})^n. 
\]
Once again, inequality~\eqref{ineq:main inequality} implies that $\Lambda_{k - \omega_0 - \omega_1}(\mathbb{S}^n,[g_{\mathrm{st}}],\sigma_{\mathrm{st}})$ is not attained. We now repeat this process for $\Lambda_{k - \omega_0 - \omega_1}(\mathbb{S}^n,[g_{\mathrm{st}}],\sigma_{\mathrm{st}})$.  After $m$ iterations, we obtain the following equality:
\[
\Lambda_k(M,[g],\sigma)^n = \Lambda_{\omega_0}(M,[g],\sigma)^n 
+ \sum_{j = 1}^m \Lambda_{\omega_j}(\mathbb{S}^n,[g_{\mathrm{st}}],\sigma_{\mathrm{st}})^n 
+ \Lambda_{k - \omega_0 - \sum_{j = 1}^m \omega_j}(\mathbb{S}^n,[g_{\mathrm{st}}],\sigma_{\mathrm{st}})^n.
\]
Since $\omega_0 + \omega_1 + \ldots + \omega_m \geqslant m +1$, it follows that we can repeat this process only finitely many times. However, note that $\Lambda_1(\mathbb{S}^n,[g_\mathrm{st}],\sigma_\mathrm{st})$ is attained. Thus, if $\Lambda_{k - \omega_0 - \sum_{j = 1}^m \omega_j}(\mathbb{S}^n,[g_{\mathrm{st}}],\sigma_{\mathrm{st}})$ is not attained, then $k - \omega_0 - \sum_{j = 1}^m \omega_j > 1$ and we can do one more step. This leads to a contradiction.

If the second case occurs, then inequality~\eqref{ineq:main inequality} implies that $\Lambda_k(\mathbb{S}^n,[g_\mathrm{st}],\sigma_\mathrm{st})$ is not attained. Following the same line of reasoning as above, we reach a contradiction.
\end{proof}

Note that Theorem~\ref{main theorem for surfaces} is a particular case of Theorem~\ref{theorem:main existence theorem}. It only rests to prove the estimates for nodal sets. It will be done in the upcoming section. 

\section{Applications.}\label{section:applications}

In this section we give two applications of the techniques that we have developed throughout the article.

\subsection{Conformal eigenvalues of a Riemann surface.}
 in two-dimensional case we have that 
\[
\Lambda_{2k - 1}(M,[g],\sigma) = \Lambda_{2k}(M,[g],\sigma).
\]
Our result in this section is that the sequence $(\Lambda_{2k}(M,[g],\sigma)_{k \geqslant 1}$ is strictly increasing.

First, we need to prove the following lemma which is similar to Lemma $8.1$ in~\cite{humbert2025extremisingeigenvaluesgjmsoperators}.

\begin{lemma}\label{lemma:linear algebra}
Let $(M,g,\sigma)$ be a two-dimensional closed Riemannian spin manifold. Let $\beta$ be a function on $M$. If there exists a finite-dimensional vector space $X$ of sections of $\Sigma_g M$ such that for any subspace $Y \subset X$ of codimension $2$ there exist spinors $\phi_1,\ldots,\phi_r \in Y$ such that
\[
\beta = \abs{\phi_1}_g^2 + \ldots + \abs{\phi_r}_g^2,
\]
then $\beta \equiv 0$.
\end{lemma}

\begin{proof}
Fix $x \in M$. Then we can consider a map $\mathrm{ev}_x \colon X \to \Sigma_xM$ that assigns to each spinor its value at $x$. Since $\dim M = 2$, it follows that $\dim_\mathbb{C} \Sigma_x M = 2$, it follows that $\dim_\mathbb{C} \ker \mathrm{ev}_x \geqslant \dim_\mathbb{C} X - 2$. Thus, there exists a vector subspace $Y \subset \ker \mathrm{ev}_x$ such that $\dim_\mathbb{C} Y = \dim_\mathbb{C}X - 2$. By assumption, there exist spinors $\phi_1,\ldots,\phi_r \in Y$ such that $\beta = \abs{\phi_1}_g^2 + \ldots + \abs{\phi_r}_g^2$,
but note that, by construction,
\[
\beta(x) = \abs{\phi_1(x)}_g^2 + \ldots + \abs{\phi_r(x)}_g^2 = 0.
\]
Since $x \in M$ was arbitrary, it follows that $\beta \equiv 0$.
\end{proof}

This observation permits us to show that $\Lambda_{2k}(M,[g],\sigma) < \Lambda_{2k + 2}(M,[g],\sigma)$.

\begin{lemma}\label{lemma:strict monotonicity}
Let $(M,g,\sigma)$ be a two-dimensional closed Riemannian spin manifold. Assume that $\Lambda_{2k}(M,[g],\sigma)$ is attained at $\beta$, then $\lambda_{2k}(\beta) < \lambda_{2k + 2}(\beta)$.
\end{lemma}
\begin{proof}
We argue by contradiction, i.e., we assume that $\lambda_{2k}(\beta) = \lambda_{2k + 2}(\beta)$.

Without loss of generality, we may assume that $\norm{\beta}_{L^n(M)} = 1$. Consequently, $\lambda_{2k}(\beta) = \Lambda_{2k}(M,[g],\sigma)$. 

Recall that $i(k) = \min \{\,r \geqslant 1 \colon \lambda_r(\beta) = \lambda_{2k}(\beta)\,\}$. Let $\varphi_{i(k)},\ldots,\varphi_{2k}$ be linearly independent generalized eigenspinors corresponding to $\lambda_{2k}(\beta)$. Define $X \overset{\mathrm{def}}{=} \spanning{\varphi_{i(k)},\ldots,\varphi_{2k + 2}}$. Note that $\dim_\mathbb{C} X = 2k + 2 - i(k) + 1 = 2k + 3 - i(k)$.

Now, take a subspace $Y \subset X$ such that $\dim_\mathbb{C} Y = 2k + 1 - i(k)$. By Lemma~\ref{lemma:Euler-Lagrange for lambda_k}, there exists $Q(\beta,\cdot)$-orthonormal family $\psi_{i(k)},\ldots,\psi_{2k}$ such that
\[
\beta^{n-1} = \abs{\psi_{i(k)}}_g^2 + \ldots \abs{\psi_{2k}}_g^2 \quad \text{almost everywhere on } M.
\]
Recall that, since $\beta$ is a minimizer, the generalized eigenspinors are in fact continuous (see Lemma~\ref{lemma:regularity} and Remark~\ref{remark:regularity}). Hence, this equality holds pointwise. Thus, by Lemma~\ref{lemma:linear algebra}, $\beta \equiv 0$. This is a contradiction, since we assumed that $\norm{\beta}_{L^n(M)}=1$. 
\end{proof}

We can now state and prove the main theorem of this section.

\begin{proof}[Proof of the Theorem~\ref{theorem:strict mononicity}]
Assume that $\Lambda_{2k}(M,[g],\sigma) = \Lambda_{2k + 2}(M,[g],\sigma)$. 

If $\Lambda_{2k + 2}(M,[g],\sigma)$ is attained at $\beta$, then $\Lambda_{2k}(M,[g],\sigma)$ is also attained at $\beta$. Indeed,
$\Lambda_{2k}(M,[g],\sigma) \leqslant \lambda_{2k}(\beta) \norm{\beta}_{L^n(M)} \leqslant \lambda_{2k+2}(\beta)\norm{\beta}_{L^n(M)} = \Lambda_{2k + 2}(M,[g],\sigma)$, hence all inequalities are actually equalities. In particular, $\lambda_{2k}(\beta) = \lambda_{2k + 2}(\beta)$ and this contradicts Lemma~\ref{lemma:strict monotonicity}.

Suppose that $\Lambda_{2k + 2}(M,[g],\sigma)$ is not attained. By Theorem~\ref{theorem:main existence theorem},
\[
\Lambda_{2k+2}(M,[g],\sigma)^2 = \Lambda_{l_0}(M,[g],\sigma)^2 + \Lambda_{l_1}(\mathbb{S}^2,[g_\mathrm{st}],\sigma_\mathrm{st})^2 + \ldots + \Lambda_{l_r}(\mathbb{S}^2,[g_\mathrm{st}],\sigma_\mathrm{st})^2,
\]
where $\Lambda_{l_i}(\mathbb{S}^2,[g_\mathrm{st}],\sigma_\mathrm{st})^n$ is attained for all $i \in \{\,1,\ldots,r\,\}$ and $\Lambda_{l_0}(M,[g],\sigma)$ is attained if $l_0 > 0$ (recall that if $l_0 = 0$, then $\Lambda_{l_0}(M,[g],\sigma) = 0$). Moreover, $l_0 + l_1 + \ldots + l_r = 2k + 2$.

First of all, since $\Lambda_{l_0}(M,[g],\sigma) < \Lambda_{2k+2}(M,[g],\sigma) = \Lambda_{2k}(M,[g],\sigma)$. Thus, $l_0 < 2k$.
Without loss of generality, we may assume that $l_1 \leqslant l_2 \leqslant \ldots \leqslant l_r$. Assume that $l_r = 1$. Then $l_1 = \ldots = l_r = 1$ and, since $l_0 < 2k$, $r \geqslant 3$. Note that Theorem~\ref{theorem:Aubin-type} implies that
\[
\Lambda_{2k}(M,[g],\sigma)^2 \leqslant \Lambda_{l_0}(M,[g],\sigma)^2 + \Lambda_{l_1}(\mathbb{S}^2,[g_\mathrm{st}],\sigma_\mathrm{st})^2 + \ldots + \Lambda_{l_{r - 2}}(\mathbb{S}^2,[g_\mathrm{st}],\sigma_\mathrm{st})^2.
\]
However, we clearly have that
\[
\Lambda_{l_0}(M,[g],\sigma)^2 + \sum_{i = 1}^{r-2}\Lambda_{l_i}(\mathbb{S}^2,[g_\mathrm{st}],\sigma_\mathrm{st})^2 < \Lambda_{l_0}(M,[g],\sigma)^2 + \sum_{i = 1}^{r}\Lambda_{l_i}(\mathbb{S}^2,[g_\mathrm{st}],\sigma_\mathrm{st})^2,
\]
and the right-hand side of this inequality is equal to $\Lambda_{2k}(M,[g],\sigma)$. This leads to a contradiction. Thus, we may suppose that $l_r \geqslant 2$. Then, using Theorem~\ref{theorem:Aubin-type}, we obtain an inequality
\[
\Lambda_{2k}(M,[g],\sigma)^2 \leqslant \Lambda_{l_0}(M,[g],\sigma)^2 + \sum_{i = 1}^{r-1}\Lambda_{l_i}(\mathbb{S}^2,[g_\mathrm{st}],\sigma_\mathrm{st})^2 + \Lambda_{l_r - 2}(\mathbb{S}^2,[g_\mathrm{st}],\sigma_\mathrm{st})^2.
\]
Combining it with the equality $
\Lambda_{2k}(M,[g],\sigma)^2 = \Lambda_{l_0}(M,[g],\sigma)^2 + \sum\limits_{i = 1}^{r}\Lambda_{l_i}(\mathbb{S}^2,[g_\mathrm{st}],\sigma_\mathrm{st})^2$, we deduce that
\[
\Lambda_{l_r - 2}(\mathbb{S}^2,[g_\mathrm{st}],\sigma_\mathrm{st}) = \Lambda_{l_r}(\mathbb{S}^2,[g_\mathrm{st}],\sigma_\mathrm{st}).
\]
Recall that $\Lambda_{l_r}(\mathbb{S}^2,[g_\mathrm{st}],\sigma_\mathrm{st})$ is attained. Therefore, $\Lambda_{l_r - 2}(\mathbb{S}^2,[g_\mathrm{st}],\sigma_\mathrm{st})$ is also attained and, arguing as in the beginning of the proof, we proceed to a contradiction with Lemma~\ref{lemma:linear algebra}.

Thus, $\Lambda_{2k}(M,[g],\sigma) < \Lambda_{2k+2}(M,[g],\sigma)$. 
\end{proof}

\subsection{Conformal eigenvalues of a two-dimensional sphere.}
We can now compute explicit values of $\Lambda_{2k}(\mathbb{S}^2,[g_\mathrm{st}],\sigma_\mathrm{st})$. 

\begin{proof}[Proof of Theorem~\ref{theorem:spectrum of S^2}]
It is well-known, see for example~\cite{Baer1992}, that $\Lambda_1(\mathbb{S}^2,[g_\mathrm{st}],\sigma_\mathrm{st}) = \Lambda_2(\mathbb{S}^2,[g_\mathrm{st}],\sigma_\mathrm{st}) = 2\sqrt{\pi}$. Note that by Theorem~\ref{theorem:Aubin-type}, we have that for all $k \in \mathbb{Z}_{>0}$
\begin{equation}\label{ineq:upper bound for Lambda(S^2)}
\Lambda_{2k}(\mathbb{S}^2,[g_\mathrm{st}],\sigma_\mathrm{st})^2 \leqslant k\Lambda_{2}(\mathbb{S}^2,[g_\mathrm{st}],\sigma_\mathrm{st})^2 = 4k\pi.
\end{equation}
In particular, we always have an inequality $\Lambda_{2k}(\mathbb{S}^2,[g_\mathrm{st}],\sigma_\mathrm{st}) \leqslant 2\sqrt{\pi k}$.

We now prove the theorem by induction on $k$. 
Take $k > 1$. Assume that for all $m < k$ we have $\Lambda_{2m}(\mathbb{S}^2,[g_\mathrm{st}],\sigma_\mathrm{st}) = 2\sqrt{m\pi}$. Therefore, for all $l_0,\ldots,l_r \in \mathbb{Z}_{\geqslant 0}$ such that $l_i < r$ for all $i$ in $\{\,0,\ldots,r\,\}$ we have that 
\[
\Lambda_{l_0}(\mathbb{S}^2,[g_\mathrm{st}],\sigma_\mathrm{st})^2 + \ldots + \Lambda_{l_r}(\mathbb{S}^2,[g_\mathrm{st}],\sigma_\mathrm{st})^2 \in 4\pi\mathbb{Z}_{\geqslant 0}.
\]
Consequently, by Theorem~\ref{theorem:main existence theorem}, we have that if $\Lambda_{2k}(\mathbb{S}^2,[g_\mathrm{st}],\sigma_\mathrm{st})$ is not attained, then $\Lambda_{2k}(\mathbb{S}^2,[g_\mathrm{st}],\sigma_\mathrm{st})^2$ is an integer multiple of $4\pi$. Recall that, by Theorem~\ref{theorem:strict mononicity}, the sequence $\left(\Lambda_{2m}(\mathbb{S}^2,[g_\mathrm{st}],\sigma_\mathrm{st})\right)_{m \in \mathbb{Z}}$ is strictly increasing. In particular,
\[
\Lambda_{2k}(\mathbb{S}^2,[g_\mathrm{st}],\sigma_\mathrm{st})^2 > \Lambda_{2k-2}(\mathbb{S}^2,[g_\mathrm{st}],\sigma_\mathrm{st})^2 = 4\pi(k-1).
\]
Therefore, combining this inequality with~\eqref{ineq:upper bound for Lambda(S^2)} and the fact that $\frac{\Lambda_{2k}(\mathbb{S}^2,[g_\mathrm{st}],\sigma_\mathrm{st})^2}{4\pi}$ is an integer, we deduce that $\Lambda_{2k}(\mathbb{S}^2,[g_\mathrm{st}],\sigma_\mathrm{st}) = 2\sqrt{k\pi}$.

Thus, we may suppose that $\Lambda_{2k}(\mathbb{S}^2,[g_\mathrm{st}],\sigma_\mathrm{st})$ is attained, i.e., there exist $\beta \in L^2_{\geqslant 0}(\mathbb{S}^2)$ such that $\bar{\lambda}_{2k}(\beta) = \Lambda_{2k}(\mathbb{S}^2,[g_\mathrm{st}],\sigma_\mathrm{st})$. Note that actually $\beta$ is smooth, because $\mathbb{S}^2$ is a surface, then as it was shown in Section $4$ of~\cite{KarpukhinMetrasPolterovich2024}, there exists a branched minimal immersion $\Psi \colon M \to (\mathbb{CP}^{2m-1},g_{\mathbb{CP}^{2m-1}})$, where $g_{\mathbb{CP}^{2m-1}} = 4g_\mathrm{FS}$ is the standard Fubini-Study metric of holomorphic sectional curvature $4$ such that $\exists \alpha > 0$  such that $\beta g = \alpha \Psi^*g_{\mathbb{CP}^{2m-1}}$ and $\lambda_{2k}(\Psi^*g_{\mathbb{CP}^{2m-1}}) = 1$. Thus,
\[
\bar{\lambda}_{2k}(\beta)^2 = \mathrm{Area}\left(\mathbb{S}^2,\Psi^*g_{\mathbb{CP}^{2m-1}}\right).
\]
As it is explained, for example, in Section $3$ of~\cite{BoltonJensenRigoliWoodward1988}, $\mathrm{Area}(\mathbb{S}^2,\Psi^*g_\mathrm{FS}) \in \pi \mathbb{Z}$. 
Thus, we have shown that even if $\Lambda_{2k}(\mathbb{S}^2,[g_\mathrm{st}],\sigma_\mathrm{st})$ is attained, $\Lambda_{2k}(\mathbb{S}^2,[g_\mathrm{st}],\sigma_\mathrm{st})^2 \in 4\pi\mathbb{Z}_{\geqslant 0}$. Arguing as above, we deduce that
\[
\Lambda_{2k}(\mathbb{S}^2,[g_\mathrm{st}],\sigma_\mathrm{st}) = 2\sqrt{k\pi}.
\]
\end{proof}

This result can be seen as an isoperimetric inequality on $\mathbb{S}^2$. Namely, for any metric $g$ on $\mathbb{S}^2$ we have
\[
\lambda_{2k}(D_g)\mathrm{Vol(\mathbb{S}^2,g)}\geqslant 2\sqrt{\pi k}.
\]

\subsection{Remark on the nodal set of minimizers.}
In this section we prove the last claim of Theorem~\ref{main theorem for surfaces}.
Let $(M,g,\sigma)$ be a $2$-dimensional closed spin Riemannian manifold. It follows from Theorems~\ref{theorem:Aubin-type} and~\ref{theorem:spectrum of S^2} that 
\[
\Lambda_{2k}(M,[g],\sigma) \leqslant 2\sqrt{k\pi}.
\]
We will use this observation in order to extend Ammann's result on the nodal set of minimizer corresponding to the first positive Dirac eigenvalue (see Theorem $1.6.$ of ~\cite{ammann2009smallest}).

Let $\beta^2 g$ be a generalized metric such that $\bar{\lambda}_{k}(\beta) = \Lambda_{k}(M,[g],\sigma)$. Then $\beta$ is smooth and, by Lemma~\ref{lemma:Euler-Lagrange for lambda_k}, we can find generalized eigenspinors $\varphi_1, \ldots, \varphi_l$ with $l \leqslant k$ such that
\begin{equation}\label{eq:nodal set}
\beta = \sum_{i = 1}^l \abs{\varphi_i}_g^2.    
\end{equation}
Recall that each $\varphi_i$ satisfies $D_g \varphi_i = \lambda_k(\beta) \beta \varphi_i$. Since $\beta$ is smooth, then the nodal set of $\varphi_i$ is discrete (see Main Theorem of~\cite{Bar1997}). Consequently, if follows from~\eqref{eq:nodal set} that the set $\{\,x \in M \colon \beta(x) = 0\,\}$ is discrete. 

The following lemma is the straightforward generalization of Lemma $4.4.$ from~\cite{ammann2009smallest} to our setting.

\begin{lemma}\label{lemma:Gauss curvature estimate}
Let $(M,g,\sigma)$ be a $2$-dimensional closed spin Riemannian manifold. Assume that there exist spinors $\varphi_1,\ldots,\varphi_l$ such that $D_g \varphi_i = \lambda \varphi_i$ for all $i \in \{\,1,\ldots,l\,\}$ and 
\[
\sum_{i = 1}^l \abs{\varphi_i}_g^2 = 1.
\]
Then $K_g \leqslant \lambda^2$, where $K_g$ is the Gauss curvature of the metric $g$.
\end{lemma}
\begin{proof}
Define the Friedrich connection $\tilde{\nabla}_X \psi \overset{\mathrm{def}}{=} \nabla_X \psi + \frac{\lambda}{2}X \cdot \psi$, where $\nabla$ is the standard spin connection. One can check that 
\begin{equation}\label{eq:gen S-L formula}
\left(D - \frac{\lambda}{2}\right)^2 = \tilde{\nabla}^*\tilde{\nabla} + \frac{K_g}{2} - \frac{1}{4}\lambda^2. 
\end{equation}
Moreover, Friedrich connection is a metric connection. Thus,
\begin{align*}
0 = \frac{1}{2}d^*d \left(\sum_{i = 1}^l \abs{\varphi_i}_g^2\right) = 
\sum_{i = 1}^l \frac{1}{2}d^*d \abs{\varphi_i}_g^2
=
\sum_{i = 1}^l \left(\Re\dprod{\tilde{\nabla}^* \tilde{\nabla} \varphi_i,\varphi_i} - \abs{\tilde{\nabla}\varphi_i}_g^2\right)
\end{align*}
Using that $D_g \varphi_i = \lambda \varphi$, We deduce from~\eqref{eq:gen S-L formula} that 
\[
\tilde{\nabla}^* \tilde{\nabla} \varphi_i =  \frac{\lambda^2}{2}\varphi_i - \frac{K_g}{2}\varphi_i.
\]
Thus,
\[
0 = \sum_{i = 1}^l \left(\frac{\lambda^2}{2}\dprod{\varphi_i, \varphi_i}_g - \frac{K_g}{2}\dprod{\varphi_i,\varphi_i}  -\abs{\tilde{\nabla}\varphi_i}_g^2\right).
\]
Since 
$\sum\limits_{i = 1}^l \abs{\varphi_i}_g^2 = 1$ and $\sum\limits_{i = 1}^l\abs{\tilde{\nabla}\varphi_i}_g^2 \geqslant 0$, it follows that
\[
K_g \leqslant \lambda^2.
\]
\end{proof}

We can now conclude with the following proposition whose proof is basically the same as the proof of Proposition $8.4$ from~\cite{ammann2009smallest}. 

\begin{proposition}
\[
\#\{\,x \colon \beta(x) = 0\,\} \leqslant \frac{k}{2} + \gamma - 1,
\]  
where $\gamma$ is the genus of $M$.
\end{proposition}
\begin{proof}
Define $M' \overset{\mathrm{def}}{=} M \setminus \{\,x \colon \beta(x) = 0\,\}$. Then $\tilde{g} = \beta^2 g$ is a Riemannian metric on $M'$. Without loss of generality, we may assume that $\norm{\beta}_{L^4(M)} = 1$. Thus, $\mathrm{Vol}(M',\tilde{g}) = 1$. Moreover, by Lemma~\ref{lemma:Gauss curvature estimate}, $K_{\tilde{g}} \leqslant \lambda_k(\beta)^2$. 

Assume that $\beta(p) = 0$ for some $p \in M$. Then the integral of the geodesic curvature with respect to $\tilde{g}$ over small simply closed loop around $p$ is equal to $-2\pi(1 + \mathrm{ord}_p(\beta))$, where $\mathrm{ord}_p(\beta)$ is the order of $\beta$ at $p$. We now remove small open disks around each point where $\beta$ vanishes and obtain a new surface $M'' \subset M'$. By Gauss-Bonnet Theorem,
\[
\chi(M'') = \int_{M''} K_{\tilde{g}} + \int_{\partial M''}k_{\tilde{g}} \leqslant \lambda_k(\beta)^2 - \sum_{p \colon \beta(p) = 0}2\pi(1 + \mathrm{ord}_p(\beta)).
\]
Consequently,
\[
\chi(M) \leqslant \lambda_k(\beta)^2 - 2\pi\sum_{p \colon \beta(p) = 0} \mathrm{ord}_p(\beta).
\]
Since $\lambda_k(\beta) = \bar{\lambda}_k(\beta) = \Lambda_k(M,[g],\sigma)$, it follows that $\lambda_k(\beta)^2 \leqslant 4\pi\lceil\frac{k}{2}\rceil \leqslant 2\pi(k + 1)$.

Thus,
\[
\sum_{p \colon \beta(p) = 0} \mathrm{ord}_p(\beta) \leqslant k + 2\gamma - 2. 
\]
Note that $\beta$ is a nonnegative function. Therefore, if $\beta(p) = 0$, then $\mathrm{ord}_p(\beta) = 0$, because $p$ is a local minimum of $\beta$. Consequently,
\[
\#\{\,x \colon \beta(x) = 0\,\} \leqslant\sum_{p \colon \beta(p) = 0} \frac{\mathrm{ord}_p(\beta)}{2} \leqslant  \frac{k}{2} + \gamma - 1.
\]
\end{proof}

\appendix
\section{Sobolev-like inequality.}\label{appendix}
The goal of the appendix is to prove the following theorem:
\begin{theorem}
Let $(M,[g],\sigma)$ be an $n$-dimensional closed compact Riemannian spin manifold. Then for all $\varepsilon > 0$, there exists a constant $B_{\varepsilon}$ such that 
\begin{equation*}\label{ineqaulity: Sobolev-like inequality}
\abs{\int_M \dprod{D_g\varphi,\varphi}_g dv_g} \leqslant (K(n) + \varepsilon)\left(\int_M\abs{D_g \varphi}_g^\frac{2n}{n+1} dv_g\right)^\frac{n+1}{n} + B_{\varepsilon} \left(\int_M \abs{\varphi}_g^\frac{2n}{n+1} dv_g \right)^\frac{n + 1}{n},
\end{equation*}\label{theorem: Sobolev Constant}
where $K(n) \overset{\mathrm{def}}{=} \Lambda_1(\mathbb{S}^n,[g_{\mathrm{st}}],\sigma_{\mathrm{st}})^{-1}$.
\end{theorem}
A similar result was proved in~\cite{RAULOT20091588}, under the assumption that $D_g$ is invertible. The idea of the proof heavily relies on Raulot’s original argument, but we need to extend Raulot’s result to accommodate a possibly non-trivial kernel. 

The Dirac operator in $\mathbb{R}^n$ is 
\[
D = \sum_{j = 1}^n \gamma^j \partial_j,
\]
where $\partial_j = \frac{\partial}{\partial x^j}$ and $\gamma^j$ are skew-Hermitian matrices satisfying the Clifford relation:
\begin{equation}\label{eq:Clifford relation}
\gamma^k \gamma^j + \gamma^j\gamma^k = -2\delta_{kj}\mathrm{Id},
\end{equation}
where $\mathrm{Id}$ is the identity matrice. 

First of all we prove the following elliptic estimate.
\begin{lemma}\label{lemma:Fourier}
Let $\nabla$ be a spin connection in $\mathbb{R}^n$ and $D$ the corresponding Dirac operator. Then there exists a constant $C > 0$ such that 
\[
\norm{\nabla \psi}_{L^p(\mathbb{R}^n)} \leqslant C\norm{D \psi}_{L^p(\mathbb{R}^n)} 
\]
\end{lemma}
\begin{proof}
Denote by $\mathcal{F}$ the Fouirier transform in $\mathbb{R}^n$, then
\[
\mathcal{F}(D \psi) = \sum_{j = 1}^n i \xi_j \gamma^j \mathcal{F}(\psi).
\]
Note that since $(\gamma^j)_j$ satisfy~\eqref{eq:Clifford relation}, it follows that 
\[
\left(\sum_{j = 1}^n i \xi_j \gamma^j\right)^2 = \abs{\xi}^2 \mathrm{Id}
\]
Hence, for $\xi \neq 0$ we have that
\[
\mathcal{F}(\psi) =  \frac{\sum_{j = 1}^n i \xi_j \gamma^j} {\abs{\xi}^2} \mathcal{F}(D\psi).
\]
Moreover, 
$\mathcal{F}(\partial_k\psi) = i \xi_k \mathcal{F}(\psi)$. Consequently,
\[
\mathcal{F}(\partial_k \psi) =  -\frac{\sum_{j = 1}^n \xi_k \xi_j \gamma^j} {\abs{\xi}^2} \mathcal{F}(D\psi).
\]
Recall that the Riesz transform is given by
\[
\mathcal{F}(R_j\psi) = \frac{i \xi_j}{\abs{\xi}}\mathcal{F}(\psi).
\]
Hence, 
\[
\partial_k \psi = \sum_{j = 1}^n \gamma^j R_kR_j(D\psi).
\]
Since Riesz transform is a bounded map from $L^p$ to $L^p$, it follows that there exists a constant $C_k > 0$ such that
\[
\norm{\partial_k \psi}_{L^p(\mathbb{R}^n)} \leqslant C_k \norm{D \psi}_{L^p(\mathbb{R}^n)}. 
\]
\end{proof}

To prepare for the proof of Theorem~\ref{theorem: Sobolev Constant}, we state a key inequality in $\mathbb{R}^n$. This result follows from the invertibility of the Dirac operator on the standard sphere $(\mathbb{S}^n,g_\mathrm{st},\sigma_\mathrm{st})$, combined with the fact that $\mathbb{R}^n$ is conformally equivalent to the punctured sphere $\mathbb{S}^n \setminus \{\,\mathrm{pt}\,\}$. For a detailed proof, we refer the reader to Lemma $1$ and Remark $1(2)$ in~\cite{RAULOT20091588}. 

\begin{lemma}\label{lemma:L^pq Dirac in R^n}
There exists $C > 0 $ such that for any compactly supported spinor field $\psi \in \Gamma(\Sigma_\xi\mathbb{R}^n)$ we have
\[
\norm{\psi}_{L^\frac{2n}{n-1}(\mathbb{R}^n)} \leqslant C \norm{\psi}_{L^\frac{2n}{n+1}(\mathbb{R}^n)}
\]
\end{lemma}

\begin{proof}[Proof of Theorem~\ref{theorem: Sobolev Constant}.]
Fix $\varepsilon > 0$ and $x \in M$. Then there exists an open neighborhood $U \subset M$ of $x$ and open neighborhood $V \subset T_xM \simeq \mathbb{R}^n$ of $0$ such that
\begin{itemize}
    \item $\exp_x \colon V \to U$ is a diffeomorphism,
    \item $\frac{1}{(1 + \varepsilon)^2}\xi \leqslant g \leqslant (1 + \varepsilon)^2\xi$ on $U$, where $\xi$ is a standard Euclidean metric. 
\end{itemize}
By~\cite{Bourguignon1992}, there is an isomorphism of vector bundles
\[
T \colon \Sigma_{\xi}V \to \Sigma_g U, 
\]
which is a fiberwise isometry. Moreover, as explained in~\cite{Ammann2008}, Dirac operators $D_g$ and $D_\xi$ are related by the following formula,
\begin{equation*}
D_g T(\psi) = T(D_\xi \psi) + X \cdot T(\psi)  + \sum_{i,j}(b^j_i - \delta^j_i) T(\gamma^i \partial_j\psi),
\end{equation*}
where $B = (b^j_i)_{i,j}$ is a matrix satisfying $B^2 = G^{-1}$ with $G = (g_{ij})_{i,j}$, $X$ is a smooth section of the Clifford algebra bundle $\mathrm{Cl}(TV)$ and $(\gamma^i)_i$ are skew-Hermitian matrices satisfying~\eqref{eq:Clifford relation}.

Let $\varphi \in \Gamma(\Sigma_gM)$ be a smooth spinor field. We fix a function $\eta \in \mathcal{C}^{\infty}(M)$ such that $\mathrm{supp}(\eta) \subset U$. Then
\begin{align*}
\int_M \eta^2\dprod{D_g\varphi, \varphi}_gdv_g
&= \int_M \dprod{D_g(\eta \varphi), \eta\varphi}_g dv_g - \int_M \dprod{\grad_g \eta \cdot \varphi, \eta \varphi}_gdv_g.
\end{align*}
Note that $\int_M \dprod{D_g \varphi, \varphi}_g dv_g \in \mathbb{R}$ and  $\Re\left(\int_M \dprod{\grad_g \eta \cdot \varphi, \eta\varphi}dv_g\right) = 0$. Consequently,
\begin{align*}
\mathrm{Re}\int_M \eta^2\dprod{D_g\varphi, \varphi}_gdv_g 
= \int_M \dprod{D_g(\eta\varphi), \eta\varphi}_gdv_g.
\end{align*}
Define a spinor field $\psi \overset{\mathrm{def}}{=}T^{-1}(\eta \varphi)$, then
\begin{align*}
\int_M \dprod{D_g(\eta\varphi), \eta\varphi}_gdv_g= \mathbf{A} + \mathbf{B} + \mathbf{C},
\end{align*}
where
\begin{align*}
    &\mathbf{A} \overset{\mathrm{def}}{=} \int_V \dprod{D_{\xi}\psi,\psi}_\xi dv_g,\\
    &\mathbf{B} \overset{\mathrm{def}}{=} \int_V \dprod{T^{-1}(X)  \cdot \psi,\psi}_\xi dv_g,\\
    &\mathbf{C} \overset{\mathrm{def}}{=} \sum_{i,j}\int_V \dprod{(b^j_i - \delta^j_i) \gamma^i \partial_j\psi,\psi}_\xi dv_g.
\end{align*}
We now estimate the last summand. Note that
\begin{align*}
\abs{\int_V \dprod{(b^j_i - \delta^j_i) \gamma^i \partial_j\psi,\psi}_\xi dv_g} 
&=
\abs{\int_V \dprod{\partial_j\psi,(b^j_i - \delta^j_i) \gamma^i \psi}_\xi dv_g} \\
&\leqslant \int_V \abs{\partial_j\psi}_\xi \abs{(b_i^j - \delta^j_i)\psi}_{\xi} dv_g \\
&\leqslant (1 + \varepsilon)^n \int_V \abs{\partial_j\psi}_\xi \abs{(b_i^j - \delta^j_i)\psi}_{\xi} dx\\
&\leqslant (1 + \varepsilon)^n \norm{\nabla^\xi \psi}_{L^\frac{2n}{n+1}(\mathbb{R}^n)} \norm{(b^j_i - \delta^j_i) \psi}_{L^\frac{2n}{n-1}(\mathbb{R}^n)}\\
&\leqslant (1 + \varepsilon)^n \left(\varepsilon \norm{\nabla^\xi \psi}_{L^\frac{2n}{n+1}(\mathbb{R}^n)}^2 + \frac{1}{\varepsilon}\norm{(b^j_i - \delta^j_i) \psi}^2_{L^\frac{2n}{n-1}(\mathbb{R}^n)}\right).
\end{align*}
By replacing $U$ by a smaller neighborhood, we may assume that $\sup\abs{b^j_i - \delta^j_i} < \varepsilon^2$ for all $i,j$. Moreover, by Lemma~\ref{lemma:Fourier}, $\exists C_1 > 0$ such that for all $\phi \in \mathcal{C}^\infty_c(\mathbb{R}^n)$ we have
\[
\norm{\nabla^\xi \phi}_{L^\frac{2n}{n+1}(\mathbb{R}^n)} \leqslant C_1 \norm{D_\xi \phi}_{L^\frac{2n}{n+1}(\mathbb{R}^n)}
\]
Therefore, we deduce that
\[
\mathbf{C} \leqslant n^2(1 + \varepsilon)^n \left(\varepsilon C_1 \norm{D_\xi \psi}_{L^\frac{2n}{n+1}(\mathbb{R}^n)}^2 + \varepsilon\norm{\psi}^2_{L^\frac{2n}{n-1}(\mathbb{R}^n)}\right).
\]
Combining this inequality with Lemma~\ref{lemma:L^pq Dirac in R^n}, we deduce that
\[
\mathbf{C} \leqslant \varepsilon n^2(1 + \varepsilon)^n  (C_1 + C^2) \norm{D_\xi \psi}_{L^\frac{2n}{n+1}(\mathbb{R}^n)}^2.
\]
In order to simplify notations we denote $n^2(C_1 + C^2)$ by $\tilde{C}$. Hence, we can write 
\begin{equation}\label{eq:|C|}
\mathbf{C} \leqslant \varepsilon (1 + \varepsilon)^n  \tilde{C} \norm{D_\xi \psi}_{L^\frac{2n}{n+1}(\mathbb{R}^n)}^2
\end{equation}
By sharp Sobolev-like inequality in $\mathbb{R}^n$, we obtain that 
\begin{equation}\label{eq:|A|}
\mathbf{A} \leqslant K(n) (1 + \varepsilon)^n\norm{D_\xi \psi}^2_{L^\frac{2n}{n+1}(\mathbb{R}^n)}.    
\end{equation}

Let us now estimate $\mathbf{B}$. Firstly, $\exists C_2 > 0$ such that $B \leqslant C_2(1 + \varepsilon)^n \int_V \abs{\psi}_\xi^2 dv_\xi$. Note that
\[
\int_{\mathbb{R}^n}\abs{\psi}^2_\xi dv_\xi \leqslant \norm{\psi_\xi}_{L^\frac{2n}{n-1}(\mathbb{R}^n)}\norm{\psi_\varepsilon}_{L^\frac{2n}{n+1}(\mathbb{R}^n)} \leqslant \varepsilon \norm{\psi_\xi}_{L^\frac{2n}{n-1}(\mathbb{R}^n)}^2 + \frac{1}{\varepsilon}\norm{\psi_\varepsilon}_{L^\frac{2n}{n+1}(\mathbb{R}^n)}^2
\]
Consequently, by Lemma~\ref{lemma:L^pq Dirac in R^n}, we have
\begin{equation}\label{eq:|B|}
\mathbf{B} \leqslant (1 + \varepsilon)^n C_2 \left(\varepsilon C\norm{D_\xi\psi}_{L^\frac{2n}{n + 1}(\mathbb{R}^n)}^2 + \frac{1}{\varepsilon} \norm{\psi}_{L^\frac{2n}{n+1}(\mathbb{R}^n)}^2\right). 
\end{equation}
Combining~\eqref{eq:|A|},~\eqref{eq:|B|} and~\eqref{eq:|C|}, we dedeuce that
\[
\abs{\int_M \dprod{D_g(\eta\varphi), \eta\varphi}_gdv_g} \leqslant (1 + \varepsilon)^n\left((K(n) + \varepsilon \bar{C})\norm{D_\xi \psi}^2_{L^\frac{2n}{n+1}(\mathbb{R}^n)} + \frac{1}{\varepsilon} \norm{\psi}^2_{L^\frac{2n}{n+1}(\mathbb{R}^n)}\right),
\]
where $\bar{C} = \tilde{C} + CC_2$.

Note that $\norm{\psi}^2_{L^\frac{2n}{n+1}(\mathbb{R}^n)} \leqslant (1 + \varepsilon)^{2n}\norm{\eta \varphi}^2_{L^\frac{2n}{n+1}(M)}$. 

Now we want to express $D_\xi$ in terms of $D_g$ and $\nabla^g$. Recall that (see~\cite{Ammann2008}) 
\[
\nabla^g_{e_i} T(\psi) = T(\nabla^\xi_{e_i}\psi) + \frac{1}{4}\sum_{l,k}\tilde{\Gamma}^k_{il}e_i \cdot e_l \cdot T(\psi),
\]
where $e_i = b^j_i\partial_j$ and $\tilde\Gamma^k_{ij} \overset{\mathrm{def}}{=} \dprod{\nabla^g_{e_i} e_j, e_k}_g$. Define a matrice $(c^i_j)_{i,j} \overset{\mathrm{def}}{=} (b^j_i)_{i,j}^{-1}$. Therefore,
\[
T(\partial_j\psi) = c^i_j\nabla^g_{e_i} T(\psi) -\frac{1}{4}\sum_{l,k}c^i_j\tilde{\Gamma}^k_{il}e_i \cdot e_l \cdot T(\psi). 
\]
For the sake of simplicity denote $\frac{1}{4}\sum_{l,k}c^i_j\tilde{\Gamma}^k_{il}e_i \cdot e_l$ by $X_j$. Therefore,
\begin{align*}
T(D_\xi \psi) 
&= \sum_{j=1}^n e_j \cdot c^i_j \nabla^g_{e_i}T(\psi) - e_j \cdot X_j \cdot T(\psi)\\ 
&= D_g(T(\psi)) + \sum_{j=1}^n(c^i_j - \delta^i_j)e_j\cdot \nabla^g_{e_i}T(\psi) - e_j \cdot X_j \cdot T(\psi).
\end{align*}
Once again, in order to simplify notations, we denote $\sum\limits_{j = 1}^n e_j \cdot X_j$ by $X$. Also, we may assume that for all $i,j$ we have $\sup \abs{c^i_j - \delta^i_j} < \varepsilon$. 
Therefore, there exists $C_3 > 0$ such that
\begin{align*}
&\norm{D_\xi \psi}_{L^\frac{2n}{n+1}(\mathbb{R}^n)} \\
&\leqslant
(1 + \varepsilon)^n \left(\norm{D_g(T(\psi))}_{L^\frac{2n}{n+1}(M)} + \varepsilon n^2 \norm{\nabla^g T(\psi)}_{L^\frac{2n}{n+1}(M)} + C_3 \norm{T(\psi)}_{L^\frac{2n}{n+1}(M)}
\right)   
\end{align*}
Moreover, since $(M,g)$ is a closed manifold, there exists a constant $C_4 > 0$ such that 
\[
\norm{\nabla^g T(\psi)}_{L^\frac{2n}{n+1}(M)} \leqslant C_4\left(
\norm{D_g (T(\psi))}_{L^\frac{2n}{n+1}(M)} + \norm{T(\psi)}_{L^\frac{2n}{n+1}(M)} \right)
\]
Hence, there exists a constant $C_5 > 0$ such that
\[
\norm{D_\xi \psi}_{L^\frac{2n}{n+1}(\mathbb{R}^n)} \leqslant (1 + \varepsilon)^n(1 + \varepsilon n^2C_4)\norm{D_g (T(\psi))}_{L^\frac{2n}{n+1}(M)} + C_5\norm{T(\psi)}_{L^\frac{2n}{n+1}(M)}.
\]
Using an inequality: $(a + b)^2 \leqslant (1 + \varepsilon)a^2 + (1 + \varepsilon^{-1})b^2$, we deduce that
\[
\norm{D_\xi \psi}_{L^\frac{2n}{n+1}(\mathbb{R}^n)} \leqslant (1 + \varepsilon)^{2n + 1}(1 + \varepsilon n^2C_4)^2\norm{D_g (T(\psi))}_{L^\frac{2n}{n+1}(M)}^2 + (1 + \varepsilon^{-1})C_5^2\norm{T(\psi)}_{L^\frac{2n}{n+1}(M)}^2. 
\]

Consequently, we deduce that there exist a constant $C_\varepsilon > 0$ such that around each point $x \in M$ there exists a neighbourhood $U_x$ such that for all $\varphi \in \Sigma_gM$ and $\eta \in \mathcal{C}^\infty_c(U)$ we have
\[
\abs{\int_M \dprod{D_g(\eta\varphi), \eta\varphi}_gdv_g} \leqslant (K(n) + \varepsilon)\norm{D_g (\eta \varphi)}_{L^\frac{2n}{n+1}(M)}^2  + C_\varepsilon \norm{\eta \varphi}^2_{L^\frac{2n}{n+1}(M)}.
\]

Since $M$ is compact, we can extract a finite subcover $(U_i)_{1 \leqslant i \leqslant N}$. Let $(\eta_i)_{1 \leqslant i \leqslant N}$ be a smooth subordinate partition of unity, that is
\begin{itemize}
    \item $0 \leqslant \eta_i \leqslant 1$ on $M$,
    \item $\mathrm{supp}(\eta_i) \subset U_i$,
    \item $\sum\limits_{i = 1}^N \eta_i^2 = 1$.
\end{itemize}
Since $\int_M \dprod{D_g \varphi, \varphi}_g dv_g \in \mathbb{R}$, it follows that
\[
\int_M \dprod{D_g \varphi, \varphi}_gdv_g = \sum_{i = 1}^N \Re\int_M \eta_i^2 \dprod{D_g(\varphi), \varphi}_gdv_g = \sum_{i = 1}^N \int_M \dprod{D_g(\eta_i \varphi), \eta_i \varphi}_gdv_g.
\]
Thus,
\begin{align*}
\abs{\int_M \dprod{D_g \varphi, \varphi}_gdv_g} 
&\leqslant \sum_{i = 1}^N \abs{\int_M \dprod{D_g(\eta_i \varphi), \eta_i \varphi}_gdv_g} \\
&\leqslant \sum_{i = 1}^N (K(n) + \varepsilon)\norm{D_g (\eta_i \varphi)}_{L^\frac{2n}{n+1}(M)}^2  + C_\varepsilon \norm{\eta_i \varphi}^2_{L^\frac{2n}{n+1}(M)}.
\end{align*}
Now,
\begin{align*}
\sum_{i = 1}^N \norm{D_g (\eta_i \varphi)}_{L^\frac{2n}{n+1}(M)}^2 
&= \sum_{i = 1}^N \left(\int_M \left(\abs{D_g (\eta_i \varphi)}_g^2\right)^\frac{n}{n+1} dv_g\right)^\frac{n+1}{n}\\ 
&\leqslant \left(\int_M \left(\sum_{i = 1}^N\abs{D_g (\eta_i \varphi)}_g^2\right)^\frac{n}{n+1} dv_g\right)^\frac{n+1}{n}.
\end{align*}
Recall that
\begin{align*}
\abs{D_g (\eta_i\varphi)}^2_g 
&\leqslant \abs{\eta_i D_g (\varphi)}_g^2 + 2\abs{\eta_i D_g\varphi}_g \abs{\grad_g\eta_i\cdot \varphi}_g + \abs{\grad_g\eta_i\cdot \varphi}_g^2\\
&\leqslant (1 + \varepsilon)\abs{\eta_i D_g (\varphi)}_g^2 + \left(1 + \varepsilon^{-1}\right) \abs{\grad_g \eta_i \cdot \varphi}_g^2.
\end{align*}
Hence, there exists a constant $C' > 0$ such that
\[
\sum_{i = 1}^N \norm{D_g (\eta_i \varphi)}_{L^\frac{2n}{n+1}(M)}^2  
\leqslant  \left(\int_M \left(1 + \varepsilon)\abs{ D_g (\varphi)}_g^2 + C'\abs{\varphi}_g^2\right)^\frac{n}{n+1} dv_g\right)^\frac{n+1}{n}.
\]
Since $\frac{n}{n+1} < 1$, it follows that there exists $C_\varepsilon > 0$ such that
\[
\left(1 + \varepsilon)\abs{ D_g (\varphi)}_g^2 + C'\abs{\varphi}_g^2\right)^\frac{n}{n+1} 
\leqslant (1 + \varepsilon)^\frac{2n + 1}{n+ 1}\abs{D_g \varphi}_g^\frac{2n}{n+1} + C_\varepsilon\abs{\varphi}_g^\frac{2n}{n+1}.
\]
Hence,
\[
\sum_{i = 1}^N \norm{D_g (\eta_i \varphi)}_{L^\frac{2n}{n+1}(M)}^2 \leqslant \left(\int_M (1 + o(\varepsilon))\abs{D_g \varphi}_g^\frac{2n}{n+1} + C_\varepsilon\abs{\varphi}_g^\frac{2n}{n+1}dv_g\right)^\frac{n+1}{n}.
\]
Recall that for all $\delta > 0$, there exists a constant $C_\delta$ such that for all $a,b \geqslant 0$ we have
\[
(a + b)^\frac{n+1}{n} \leqslant (1 + \delta)a^\frac{n+1}{n} + C_\delta b^\frac{n+1}{n}.
\]
Using this inequality, we obtain that
\[
\sum_{i = 1}^N \norm{D_g (\eta_i \varphi)}_{L^\frac{2n}{n+1}(M)}^2 \leqslant  (1 + o(\varepsilon))\norm{D_g \varphi}^2_{L^\frac{2n}{n+1}(M)} + \tilde{C}_{\varepsilon}\norm{\varphi}_{L^\frac{2n}{n+1}(M)}^2.
\]
Similarly,
\begin{align*}
\sum_{i = 1}^N \norm{\eta_i \varphi}_{L^\frac{2n}{n+1}(M)}^2 
&= \sum_{i = 1}^N \left(\int_M \left(\abs{\eta_i \varphi}_g^2\right)^\frac{n}{n+1} dv_g\right)^\frac{n+1}{n}\\ 
&\leqslant \left(\int_M \left(\sum_{i = 1}^N\abs{\eta_i \varphi}_g^2\right)^\frac{n}{n+1} dv_g\right)^\frac{n+1}{n}\\ 
&= \norm{\varphi}^2_{L^\frac{2n}{n+1}(M)}.
\end{align*}
Thus, there exists $B_{\varepsilon} > 0$ such that for all $\varphi \in \Gamma(\Sigma_g M)$ we have
\[
\abs{\int_M \dprod{D_g \varphi, \varphi}_g dv_g} \leqslant K(n)(1 + o(\varepsilon)) \norm{D_g \varphi}^2_{L^\frac{2n}{n+1}(M)} + B_\varepsilon \norm{\varphi}_{L^\frac{2n}{n+1}(M)}^2.
\]
\end{proof}

\bibliographystyle{alphaurl} 
\bibliography{refs} 

\end{document}